\documentclass[aop]{imsart}

\usepackage{latexsym}
\RequirePackage{amsthm,amsmath,amsfonts,amssymb}
\RequirePackage[numbers]{natbib}
\usepackage[normalem]{ulem}
\usepackage{xcolor} 
\usepackage{dsfont}
\RequirePackage[colorlinks,citecolor=blue,urlcolor=blue]{hyperref}
\RequirePackage{graphicx}

\startlocaldefs

\newtheorem{axiom}{Axiom}
\newtheorem{claim}[axiom]{Claim}
\newtheorem{theorem}{Theorem}[section]
\newtheorem{lemma}[theorem]{Lemma}
\newtheorem{prop}[theorem]{Proposition}
\newtheorem{remark}[theorem]{Remark}
\newtheorem{open}[theorem]{Problem}
\newtheorem{corollary}[theorem]{Corollary}
\theoremstyle{remark}
\newtheorem{definition}[theorem]{Definition}

\numberwithin{equation}{section}

\usepackage{etex}

\usepackage{bbm}
\allowdisplaybreaks

\def\beq{ \begin{equation} }
	\def\eeq{ \end{equation} }

\def\ep{\varepsilon}

\def\M{t_{\mathrm{meet}}} 
\def\RM{\tau_{\mathrm{meet}}}
\def\coal{t_{\mathrm{coal}}}
\def\Rcoal{\tau_{\mathrm{coal}}}
\def\E{\mathbb{E}}
\def\C{\mathcal{C}}

\def\G{\mathcal{G}}
\def\U{\mathcal{U}}

\def\R{\mathbb{R}}
\def\P{\mathbb{P}}
\def\RR{\mathbb{R}}
\def\N{\overline{N}}
\def\f{\bar{f}}
\def\r{\mathbf{r}}
\def\1{\mathbbm{1}}
\def\TV{\mathrm{TV}}
\def\ZZ{\mathbb{Z}}
\def\d{\bar{d}}
\def\rel{t_{\mathrm{rel}}}
\def\mix{\mathrm{mix}}
\def\t{\boldsymbol{t}}
\def\t{\boldsymbol{t}}
\def\s{\boldsymbol{s}}

\def\Poi{\hbox{Poisson}}

\def\u{\mathfrak{u}}

\def\scG{\mathfrak{G}}

\DeclareMathOperator{\Var}{Var}
\DeclareMathOperator{\Unif}{Unif}
\DeclareMathOperator{\Exp}{Exp}

\DeclareMathOperator{\CM}{\mathbb{CM}}
\DeclareMathOperator{\UGT}{\mathbb{UGT}}

\newcommand{\norm}[1]{\left\Vert#1\right\Vert}
\newcommand{\abs}[1]{\left\vert#1\right\vert}

\DeclareMathSymbol{\eset}{\mathalpha}{AMSb}{"3F}

\endlocaldefs

\begin{document}

\begin{frontmatter}
\title{Mean Field Behavior during the Big Bang Regime for Coalescing Random Walks}
\runtitle{Mean Field Behavior for CRW}

\begin{aug}
\author[A]{\fnms{Jonathan} \snm{Hermon}\ead[label=e1]{jonathan.hermon@gmail.com}},
\author[B]{\fnms{Shuangping} \snm{Li}\ead[label=e2]{sl31@princeton.edu}},
\author[C]{\fnms{Dong} \snm{Yao}\ead[label=e3]{wonderspiritfall@gmail.com}}
\and
\author[D]{\fnms{Lingfu} \snm{Zhang}\ead[label=e4]{lingfuz@math.princeton.edu}}
\address[A]{Department of Mathematics, University of British Columbia, Vancouver, Canada.
\printead{e1}}

\address[B]{PACM, Princeton University, Princeton, NJ, USA.
\printead{e2}}

\address[C]{Research Institute of Mathematical Science, Jiangsu Normal University, Xuzhou, China.
		\printead{e3}}
	
\address[D]{Department of Mathematics, Princeton University, Princeton, NJ, USA. 
	\printead{e4}}
\end{aug}

\begin{abstract}
In this paper, we consider coalescing random walks on a general connected graph $\G=(V,E)$.
We set up a unified framework to study the leading order of the decay rate of $P_t$, the expectation of the fraction of occupied sites at time $t$, particularly for the `Big Bang' regime where $$t \ll \coal:=\E(\inf\{s:\text{There is only one particle at time }s \}).$$
Our results show that $P_t$ satisfies certain `mean-field behaviors' if the graphs satisfy certain `transience-like' conditions.

We apply this framework to two families of graphs: (1) graphs given by the configuration model whose degree distribution is supported in the interval $[3,\bar d]$ for some $\bar d\geq 3$, and (2) finite and infinite transitive graphs.
In the first case, we show that for $1 \ll t \ll \abs{V}$, $P_t$ decays in the order of $t^{-1}$, and $(tP_t)^{-1}$ is approximately the probability that two particles starting from the root of the corresponding unimodular Galton-Watson tree never collide after one of them leaves the root. The number $(tP_t)^{-1}$ is also roughly $|V|/(2t_{\text{meet}})$, where $t_{\text{meet}}$ is the mean meeting time of two independent walkers.
By taking the local weak limit, we prove convergence of $tP_t$ as $t\to\infty$ for the corresponding unimodular Galton-Watson tree. 
For the second family of graphs, we consider a growing sequence of finite transitive graphs $\mathcal{G}_n=(V_n, E_n)$, satisfying that $t_{\text{meet}}=O(\abs{V_n})$ and the inverse of the spectral gap $t_{\text{rel}}$ is $o(\abs{V_n})$. We show that 
$t_{\text{rel}} \ll t \ll t_{\text{coal}}$, $(tP_t)^{-1}$ is approximately the probability that two random walks never meet before time $t$, and it is also roughly $|V|/(2t_{\text{meet}})$.
In addition, we define a natural `uniform transience' condition, and show that in the transitive setup it implies the above estimates of $tP_t$ for all $1 \ll t\ll t_{\text{coal}}$.
Estimates of $tP_t$ are also obtained for all infinite transient transitive unimodular graphs, in particular, all transient transitive amenable graphs.
\end{abstract}



\end{frontmatter}
\tableofcontents

\section{Introduction}

We study \emph{coalescing random walks} (\textbf{CRW}) on a connected graph $\G=(V,E)$, defined as follows. 
Initially each vertex $v \in V$ is occupied by one particle.
Each particle performs an independent continuous-time edge simple random walk (\textbf{SRW}); i.e., for each edge incidental to the current location of the particle, with rate $1$ it jumps across that edge. Whenever two particles meet (or  `collide', i.e., one particle jumps to the current location of the other one), they merge into one particle which continues to perform a SRW. Whenever such a collision happens, the total number of particles in the system decreases by 1 if the graph is finite.

One can consider the more general case in which the jump rate depends on an edge. A \emph{graphical representation} of the model is obtained by letting each directed edge $(x,y)$ ring at times given by a Poisson process of rate $r_{x,y}$, independently for different directed edges. Whenever a particle is at $x$ at time $t_-$ and the directed edge $(x,y)$ rings at time $t$, the particle moves to $y$ at time $t$. It is easy to see that each particle follows a continuous-time Markov chain denoted by $(V,\r)$, where $\r=\{r_{x,y}\}_{x, y\in V}$. Any two particles evolve independently until they collide; and once they collide, they will be at the same location at all later times (and thus they have effectively coalesced).

Unless otherwise stated, throughout this paper we assume that $V$ is either finite or countable; and the Markov chain has symmetric jump rates, i.e., $r_{x,y}=r_{y,x}$ for all $x,y$.
This implies that the Markov chain $(V,\r)$  is reversible w.r.t.\ the counting measure, and when $V$ is finite, its stationary distribution $\pi$ is the uniform distribution on $V$. We also assume $(V,\r)$ is irreducible, i.e., for any $x,y\in V$ and any $t>0$, we have $p_t(x,y)>0$ where $p_t(\cdot, \cdot)$ is the time-$t$ transition probability corresponding to $\r$.  In particular, when $(V,\r)$ is a SRW on a graph $\G$, we assume $\G$ is connected. 

This model has received a lot of attention in the literature. 
For graphs such as $\ZZ^d$ and $(\ZZ/n\ZZ)^d$, the asymptotic behaviors of the particle densities in CRW have been studied in many works, including \cite{bramson1980asymptotics,cox1989coalescing,van2000asymptotic,BK}. Bounds on the particle densities and on the coalescence time (to be defined in \eqref{def:rcoal} below) for general graphs or some other families of graphs have also been proven, for instance in \cite{oliveira2012coalescence,oliveira2013mean,cooper2013coalescing,benjamini2016site,foxall2018coalescing,meet}.
See also the references therein. Part of the motivation of studying the CRW model stems from its connection with the voter model via duality.
We refer interested reader to the survey \cite[Chapter 2 and Chapter 5]{berestycki2009recent} for a detailed introduction of the CRW model and its mean field version Kingman's coalescent, including many known results for CRW on $\ZZ^d$, the duality between CRW and the voter model, and the decay rate
of $1/t$ for the particle densities in Kingman's coalescent. We will explain some of these results in the introduction.

We denote by $\xi_t$ the set of vertices in $V$ that are occupied at time $t$.
When $V$ is finite, a quantity of interest  is the \emph{coalescence time} 
\begin{equation} \label{def:rcoal}
\Rcoal:=\inf\{t:|\xi_t|=1 \},    
\end{equation}
which is the first time when there is only one particle left. 
By duality, the law of $\Rcoal$ is the same as that of the consensus time for the voter model (see Section \ref{s:prelim} for details).

We are interested in the rate of decay (as $t \to \infty$) of the probability $P_t(x)$ of some specific site $x$ being occupied at time $t$
\[P_t(x):=\P(x \in \xi_t). \]
In the finite setup, we can define the expected density $P_t$ of particles at time $t$
\[P_t:=\frac{\E |\xi_t|}{|V|}=\E P_t(\u), \quad \text{where} \quad \u \sim \Unif(V) \text{ and is independent of the CRW}. \]

Under transitivity\footnote{We say that $(V,\r)$ is \emph{transitive}, if for all $u,v \in V$ there exists a permutation $\phi=\phi_{u,v}:V \to V$ such that $\phi(u)=v$ and $r_{\phi(x),\phi(y)}=r_{x,y}$ for all $x,y \in V$.
	When $(V,\r)$ is transitive,
	$P_t(x)$ does not depend on $x$.}, we simply write $P_{t}$ for $P_t(x)$, even when $|V|=\infty$. We also use the notation $P_t(\G)$ sometimes, when we consider a graph $\G$ where each particle takes a SRW.
Clearly,
$P_t$ is a decreasing function of $t$, which  vanishes in the infinite setup and converges to $1/|V|$ in the finite setup as $t \to \infty$, whenever $(V,\r)$ is irreducible. 

In the finite setup, loosely speaking, the evolution of CRW can be divided into two time periods. The \emph{coalescing regime} $t \asymp \coal:=\E \Rcoal$ during which there are typically $O(1)$ particles; and the regime $t \ll \coal$, known as the \emph{Big Bang Regime}, a term coined (in this context) by Rick Durrett, during which (typically) there are many particles and the model evolves rapidly (see, e.g., \cite{durrett2010some}). 

The aim of this work is to improve the understanding of the CRW model beyond the setup of $\ZZ^d$, and in the finite setup to improve the understanding of the model during the big bang regime. The initial motivation of this work was to answer questions posed by Rick Durrett in an AMS workshop on Stochastic Spatial Models
(RI 2019), which we paraphrase here:
\begin{itemize}
	\item[(i)] Prove results for the decay of $P_t$ as $t\to\infty$, for CRW on random graphs generated by the configuration model or on Galton-Watson trees. In particular, since random walks on these graphs are transient, do we expect the density to decay in the order of $t^{-1}$, or even $\lim_{t \to \infty} t P_t=c$ for some constant $c$?
	\item[(ii)] For the constant in the first question, can we identify it abstractly? Moreover, can we show that it is given by the probability that two independent walks starting at a random vertex and a random neighbor of that vertex (respectively) do not collide?
\end{itemize}
In this paper we go beyond random graphs and set up a unified framework to determine the  decay rate  of $P_t$ for general $(V,\r)$, up to smaller order terms, under certain `transience-like' conditions (Proposition \ref{p:gengraph}). 
For this we consider sequences of finite graphs $\G_n=(V_n,E_n)$ and times $t_n$, restricted to $1 \ll t_n \ll |V_n|$, as well as infinite graphs.
We apply our framework to the configuration model with minimal degree $3$ and bounded maximal degree in Theorem \ref{t:config}, or the corresponding unimodular Galton-Watson tree in Theorem \ref{t:gwtree}. In addition, in Theorems \ref{transgraph}-\ref{infinitetrans} we apply our framework to  transitive Markov chains $(V,\r)$  which either are transient and unimodular (in the infinite setup), or satisfy certain finitary notions of transience (in the finite setup), as well as a certain spectral condition proposed by Aldous and Fill in this context \cite{aldous1995reversible}.

For finite graphs and Markov chains, the novelty of our results is that they are applicable to small values of $t$. For instance, for the configuration model and finite transitive graphs satisfying a `uniform transience' condition (Definition \ref{d:uniftran} below), our results apply to the entirety of the big bang period (more precisely, for any $t \gg 1$).
Such exact asymptotic results for the rate of decay of $P_t$ were previously known only for the $d$-dimensional integer lattice $\ZZ^d$.  
In a seminal work of Bramson and Griffeath \cite{bramson1980asymptotics} they showed that for the SRW on $\ZZ^d$, we have
\begin{equation}
P_t=
 \begin{cases} 
 \frac{1}{\sqrt{\pi t}}(1 \pm o(1)),
&\hbox{ for $d=1$;} \\
\frac{1}{\pi}\frac{\log t}{t}(1 \pm o(1)),
&\hbox{ for $d=2$;}\\
\frac{1}{\psi_d t}(1 \pm o(1)),
&\hbox{ for $d\geq 3$.}
\end{cases}
\label{BBSIR}
\end{equation}
Here $\psi_d$ is the probability that a simple random walk starting from the origin never returns to it and $o(1)$ is some error term that vanishes as $t\to \infty$.

The starting point of our arguments are rough upper and lower bounds on the decay of $P_t(x)$ (Theorem \ref{t:upperbound}), which also holds for general $(V,\r)$ (which need not satisfy any `transience-like' assumptions).
When $(V,\r)$ is transitive, these bounds give the first characterization of $P_t$ sharp up to a constant factor, as previously only non-matching upper and lower bounds were available (see Section \ref{s:related}). 

We then bootstrap such general bounds into bounds sharp up to smaller order terms. Our arguments are inspired by Bramson and Griffeath \cite{bramson1980asymptotics}. We determine the asymptotics of certain moments of a quantity $\hat N_t$, which has a natural interpretation in terms of both the voter model and CRW:
for the voter model it is the size of the set of vertices which have the same opinion as the origin at time $t$, while for CRW it is the number of particles which have the same location as the particle whose initial position is the origin at time $t$ (in the sense of the graphical representation).
For $\ZZ^d$ such moments calculation is due to Sawyer; see
\cite[Theorem 2.2]{sawyer1979limit}. Alas Sawyer's proof  involves a `magical' combinatorial calculation which relies heavily on the specific structure of $\ZZ^d$ and hence does not seem to generalize to other graphs. While Bramson and Griffeath were able to use Sawyer's theorem,  which they described as `coming tantalizingly close to determining the asymptotics
of $P_t$' as a `black box', we on the other hand have to study these moments directly in terms of CRW. This is the most difficult part of our arguments. As a corollary, we generalize Sawyer's theorem to many other graphs (see Theorems \ref{voterthm1} and \ref{saw1}), thus deducing that in those cases $\hat N_t/\E \hat N_t$ converges in distribution to a Gamma distribution.

We comment that determining the asymptotics of $P_t$ (or $P_t(x)$ in the non-transitive setup) actually gives one information about the distribution of the number of vertices occupied by particles in an arbitrary large finite set $A$, i.e., $| \xi_t \cap A |$. Determining the expectation of a random variable is generally more interesting when the random variable is concentrated around the expectation, and this is the case for $| \xi_t \cap A |$ when $\sum_{a \in A} P_t(a) \gg 1$.  Indeed, Arratia \cite[Lemma 1]{Arratia} showed that for all $x \neq y$,\[\mathrm{Cov}(\1[x \in \xi_t],\1[y \in \xi_t]) \le 0. \] 
Actually Arratia proved this when the underlying graph is $\ZZ^d$, but the proof works for general Markov chains.
It follows that  $\mathrm{Var}(| \xi_t \cap A |) \le \E | \xi_t \cap A |$. Thus if $A \subset V$ satisfies that $\E | \xi_t \cap A |=\sum_{a \in A} P_t(a) \gg 1$,  by Chebyshev's inequality $| \xi_t \cap A |$  has  $O((\E | \xi_t \cap A |)^{1/2})$  fluctuations around its expectation.

\subsection{Statements of our results}\label{ss;mainresults}
To formally state our results we need some further notations and terminologies. We start by introducing the key notion of the \emph{meeting time}.
\begin{definition}
	\label{def:meeting time}
	Consider a Markov chain $(V,\r)$. For any two measures $\mu, \nu$ on $V$, we write $\mathbb{P}_{\mu}$ for the law of a walk where its starting location is sampled from $\mu$, and $\mathbb{P}_{\mu,\nu}$ for the joint law of two independent walks $X, Y:\RR_{\ge 0} \rightarrow V$, where their starting locations are sampled from $X(0) \sim \mu$ and $Y(0) \sim \nu$ independently. When $\mu$ is the Dirac measure on some $x \in V$, we also write $\mathbb{P}_{x}$ for $\mathbb{P}_{\mu}$ and $\mathbb{P}_{x,\nu}$ for $\mathbb{P}_{\mu,\nu}$. Similarly, when $\mu$ and $\nu$ are the Dirac measures on some $x, y \in V$ respectively we write $\mathbb{P}_{x,y}$ for $\mathbb{P}_{\mu,\nu}$.
	We define the \emph{meeting time} as the first collision time of the two independent walks $X$ and $Y$:
	\[  
	\RM:=\inf\{t:X(t)=Y(t) \},
	\]
	which is a random variable whose law depends on $\mu, \nu$. An important quantity in the study of CRW is $\M(\r)=\M:=\E_{\pi,\pi}\RM$, where $\pi$ is the uniform distribution on $V$. We also write $\M(\G)$ when we consider a graph $\G$ where each particle performs a SRW. By an abuse of terminology, we also refer to $\M$ as the meeting time.
\end{definition}

\begin{definition}\label{defSRW}
	For a Markov chain $(V,\r)$ and any $x, y \in V$, we let $\nu_x(y):=\frac{r_{x,y}}{r(x)}$, where $r(x):=\sum_y r_{x,y}$ is the total jump rate from $x$ (by convention $r_{x,x}=0$). 
	For any $t\ge 0$ we define
	\begin{equation}\label{alpha_t(x)}
		\alpha_t(x):=r(x)\P_{x,\nu_x}(\RM >t).
	\end{equation}
	Roughly speaking, $\alpha_t(x)$ is the weighted probability that the two walkers, starting from $x$ and a random neighbor of $x$ respectively, don't meet by time $t$. 
	We also define
	\begin{equation}\label{alpha_inf(x)}
		\alpha_\infty(x):=r(x)\P_{x,\nu_x}(\RM =\infty),
	\end{equation}
	which is the weighted probability that the two walkers, starting from $x$ and a random neighbor of $x$ respectively, don't meet forever.
	When $|V|$ is finite, we let $\alpha_t:=\sum_{x\in V}\alpha_t(x)/|V|$ be its average over all vertices.
	Note that for a transitive Markov chain $(V,\r)$, $\alpha_t(x)$ is independent of $x$, so it should equal $\alpha_t$ when $|V|$ is finite. 
	For infinite transitive chains we also use $\alpha_t$ (resp. $\alpha_\infty$) to denote $\alpha_t(x)$ (resp. $\alpha_\infty(x)$), for any $x \in V$.
\end{definition}

Our main results are of the following forms:
\begin{align}
	&P_t\approx\frac{1}{t \alpha_t},  \tag{A1}\\
	&P_t\approx \frac{2\M}{t|V|}. \tag{A2}
\end{align}
The first form is predicted in a heuristic argument of van der Berg and Kesten \cite{van2000asymptotic,BK}, and the second form comes from a mean field picture conjectured by Aldous and Fill \cite{aldous1995reversible} and recently verified by Oliveria \cite{oliveira2013mean}.
See Section \ref{s:motivation} and \ref{sssec:meanfb} respectively for more discussions.

In many of the results below, we are taking a sequence of graphs or chains, and study their asymptotic behaviors. 
For simplicity we set up some asymptotic notations.
Consider sequences of quantities $A_n$ and $B_n$, $n\in \ZZ_+$.
We write $A_n=o(B_n)$ or $A_n \ll B_n$ if $\lim_{n\to \infty}\frac{A_n}{B_n}=0$. We write $A_n=O(B_n)$ or $A_n \lesssim B_n $ (and also $B_n=\Omega(A_n)$ and $B_n \gtrsim  A_n$) if there exists a constant $C>0$ such that $|A_n| \le C |B_n|$ for all $n$. We write
$A_n \asymp B_n$ or $A_n=\Theta(B_n)$ if we have both $A_n=O(B_n)$ and $A_n=\Omega(B_n)$.

\subsubsection{The configuration model}

We first study random graphs given by the \emph{configuration model}. As  a convention in this paper and unless otherwise noted, whenever we consider a `graph', we always mean that the jump rate across each edges equals $1$ (i.e., $r_{x,y}=\1[x \textrm{ is a neighbor of }y]$). 
\begin{definition}  \label{defn:configm}
	Let $D$ be a probability measure on $\ZZ_{\geq 0}$, and $n\in \ZZ_{+}$.
	We define $\CM_n(D)$ as the probability measure on the set of all $n$-vertices graphs, given by the configuration model with degree distribution $D$. In other words, we construct $\G_n \sim \CM_n(D)$ as follows.
	We take $n$ vertices labelled $1, \ldots, n$ and $d_1, \ldots, d_n$ i.i.d. sampled from $D$.
	For each vertex $i$ we attach $d_i$ half edges to it.
	Then we get $\G_n$ by uniformly matching all half edges, conditioned on $\sum_{i=1}^n d_i$ being even.
\end{definition}
In this paper we always assume that the support of $D$ is contained in $[3, \bar d]$ for some $\bar d\geq 3$.
Note that under this condition, as $n\to \infty$, the probability for $\G_n \sim \CM_n(D)$ being connected converges to $1$ (see, e.g. Theorem 4.15 in  \cite{vanrandomvol2}).
We refer interested readers to Chapter 7 of
\cite{van2016random} 
for a detailed introduction of the configuration model. 
\begin{remark}
    The assumption that the degree is at least $3$ is used to guarantee that the vertex expansion constant of $\G_n \sim \CM_n(D)$ is bounded from below by some constant with high probability, i.e., with probability tending to $1$ (as $n\to\infty$). Together with the assumption of 
    the degree being bounded from above by $\bar d$, the relaxation time of $\G_n$ is bounded from above by some constant with high probability, which is used throughout the proof of Theorem \ref{t:config}. See Lemma \ref{l:configroughbound}. 
    The assumption of bounded maximal degree also ensures that with high probability $\G_n$ is $\delta_D$-homogeneous (defined in Definition \ref{def:delta_homo}), i.e.,  the number of vertices with each possible degree is $\Theta(n)$. See also Remark     \ref{rmk:config} for the necessity of the boundedness of maximal degree in our proof. 
    It is of great interest to weaken the assumptions on $D$ 
    to include other types of degree distribution, e.g., the Poisson distribution and the power-law distributions.  We leave this for future investigations. 
\end{remark}
\begin{definition}
	We consider the \emph{local weak limit} of $\CM_n(D)$, which is a unimodular Galton-Watson tree whose root has offspring distribution $D$, and later generations have offspring distribution $D^*$, which is a size biased version of $D$, i.e., $D^*(k) = \frac{(k+1)D(k+1)}{\sum_{i=0}^{\infty} iD(i)}$ for each $k\in \ZZ_{\geq 0}$. We use $\UGT(D)$ to denote this law on the space of all rooted trees.
\end{definition}

For more details on the local weak convergence of $\CM_n(D)$ to $\UGT(D)$, see Appendix \ref{finitetoinfinite}.
For $\G_n \sim \CM_n(D)$, the density
$P_t$ is a random variable in the space of all graphs with $n$ vertices. Likewise, for $\G\sim \UGT(D)$, if we let $o$ be its root, $P_{t}(o)$ (which is the probability of $o$ being occupied at time $t$ in the CRW) is also a random variable in the space of all infinite trees.
To verify \eqref{e:MF1} and \eqref{e:MF2} for these random variables, we define
\begin{equation}\label{eq:alpha(D)}
	\alpha(D):= \E(\alpha_{\infty}(o)),
\end{equation}
where $\alpha_{\infty}(o)$ is defined in equation \eqref{alpha_inf(x)}.

\begin{theorem}  \label{t:config}
	Let $\G_n$ be sampled from $\CM_n(D)$, conditioned on that $\G_n$ is connected.
	For any sequence of times $t_n$ such that $1 \ll t_n \ll n$, we have
	\begin{align}
		&\lim_{n\to \infty} t_n P_{t_n}(\G_n) = \frac{1}{\alpha(D)}, \label{eq:thm11}\\
		&\lim_{n\to \infty} \frac{nt_n}{2\M(\G_n)}P_{t_n}(\G_n)= 1, \label{eq:thm12}
	\end{align}
	in probability.
\end{theorem}

We will show in Section \ref{ssc:tmeetnconfig} that $\M(\G_n)/n$
converges to $1/(2\alpha(D))$ in probability. This was recently proven in the particular case of random $d$-regular graphs in \cite{chen2017meeting} using different techniques. This implies that \eqref{eq:thm11} is equivalent to \eqref{eq:thm12}.
\begin{remark}\label{rmk:config}
	We remark that our proof of Theorem \ref{t:config} works for more general graphs beyond the configuration model. Essentially we only need the following conditions:
	\begin{itemize}
		\item There exists $\delta>0$, such that for any degree $j$  present in $\G_n$, the number of vertices with degree $j$ is at least $\delta n$. 
		\item The relaxation time of $\G_n$ (defined as the inverse of its spectral gap) is upper bounded by a constant independent of $n$. 
		\item As $n\to\infty$, $\G_n$ converges to a transient (random) graph in the local weak sense. 
	\end{itemize}
	Note that the first condition and the third condition imply that the maximal degree of $\G_n$ should be bounded uniformly in $n$. The second condition implies that for the configuration model, the degree needs to be at least $3$, since otherwise there would be a path of length in the order of $\log n$, and the relaxation time would be unbounded as $n\to\infty$.
\end{remark}
A similar result holds for the unimodular Galton-Watson tree.
\begin{theorem}  \label{t:gwtree}
	For $\G \sim \UGT(D)$ with root $o$, we have 
	\[
	\lim_{t\to \infty} t \E P_{t}(o) = \frac{1}{\alpha(D)}.
	\]
\end{theorem}

\subsubsection{Transitive Markov chains}
Our next series of results are about transitive Markov chains $(V,\r)$. As a convention in this paper, whenever we consider a transitive Markov chain $(V,\r)$, we always assume that the total jump rate $r(x)=1$ for each vertex $x$. 

We denote by $\rel(\r)=\rel$ the \emph{relaxation time} of a Markov chain $(V,\r)$, defined as the inverse of its spectral gap (i.e., minus of the second largest eigenvalue of the infinitesimal generator of $(V,\r)$). We also write $\rel(\G)$ when we study the SRW on a graph $\G$.

\begin{theorem}  \label{transgraph}
	Suppose we are given a sequence of finite transitive Markov chains $(V_n,\r_n)$ such that $|V_n|\to \infty$ as $n\to\infty$, and $\M(\r_n)\lesssim|V_n|$, $\rel(\r_n)\ll  \M(\r_n)$. Then for any sequence of times $t_n$ satisfying 
	$\rel(\r_n)\ll t_n\ll \M(\r_n)$ we have that
	\[
	\lim_{n\to\infty}\frac{|V_n| t_n}{2\M(\r_n)}P_{t_n}=1,        
	\]
	\[\lim_{n \to \infty} t_n\alpha_{t_n}P_{t_n} = 1.
	\]
\end{theorem}

We next prove a stronger result under the following stronger `transience condition'. A related notion in the context of cover times of random walks was proposed in \cite[Section 2]{hermon2020spectral}.
\begin{definition}
	\label{d:uniftran}
	A sequence of finite Markov chains $(V_n,\r_n)$ with $\max_{x \in V_n} \r_n(x) \asymp 1 \asymp  \min_{x \in V_n} \r_n(x)$ is said to be \emph{uniformly transient} if
	\begin{equation}\label{defuniftrans}
		\lim_{s\to \infty}
		\limsup_{n\to\infty} \max_{x,y\in V_n}\int_{s\wedge \rel(\r_n)}^{\rel(\r_n)}
		p_t^{(n)}(x,y)\mathrm{d}t=0,
	\end{equation}
	where $p_t^{(n)}(x,y)$ is the time $t$ transition probability from $x$ to $y$ corresponding to $(V_n,\r_n)$. 
\end{definition}

As will be discussed in Section \ref{ss:assumptions}, this condition is surprisingly only slightly stronger than the assumptions in Theorem \ref{transgraph}. 
\begin{theorem}\label{uniformtran}
	Let $(V_n,\r_n)$ be a sequence of finite transitive Markov chains which is uniformly transient, and $|V_n|\to\infty$ as $n\to\infty$.
	Then for every sequence of times $t_n$ satisfying $1\ll t_n\ll \M(\r_n)$ we have
	\begin{equation}
		\label{e:thm31}
		\lim_{n\to\infty}\frac{|V_n|t_n}{2\M(\r_n)}P_{t_n}=1,
	\end{equation}
	\begin{equation}
		\label{e:thm32}\lim_{n\to\infty} t_n \alpha_{t_n}
		P_{t_n}=1.
	\end{equation}
\end{theorem}
Our methods of proving Theorems 
\ref{transgraph} to  \ref{uniformtran} can be adapted to certain infinite transitive graphs.
As a convention of this paper, all infinite graphs are assumed to be locally finite.
We will consider unimodular graphs, which are graphs that satisfy the mass-transport principle. See \cite[Chapter 8]{lyons2017probability} for details of unimodular graphs and mass-transport principle. Unimodular transitive graphs are a large sub-class of transitive graphs. All Cayley graphs and
all \emph{amenable}\footnote{An infinite transitive graph $\G=(V,E)$ is called `amenable', if
	for any $x\in V$,
	\[
	\lim_{t\to\infty} \frac{1}{t}\log p_t(x,x)=0,
	\]
	where $p_t$ denotes the transition probability of a SRW on this graph. 
	In other words, the probability of a SRW returning to the starting point does not decay exponentially in $t$. In particular, the integer lattice $\ZZ^d$ is amenable for all $d$. 
	For this and other equivalent spectral and geometric definitions of amenability, see \cite[Chapter 6]{lyons2017probability}.} transitive graphs are unimodular. 

\begin{theorem}\label{infinitetrans}
	For any infinite transient unimodular transitive graph $\G$, we have
	\[
	\lim_{t\to\infty}tP_t=\frac{1}{\alpha_{\infty}}.
	\] 
\end{theorem}

\subsubsection{Applications to the voter model: generalizing Sawyer's theorem}
\label{s:voterintro}
In this subsection, we present results about the voter model obtained by duality (see Section \ref{s:prelim} for more details).

Generally, the voter model is defined as follows. Given a Markov chain $(V,\r)$, we think of each vertex as a `voter'. Initially each vertex $x \in V$ is assigned a mutually different number, which we think of as its `opinion'. Then for any voters $x,y$, with rate $r_{x,y}$ the voter $y$ adopts the opinion of $x$.  

Now we state Sawyer's Theorem \cite[Theorem 2.2]{sawyer1979limit}.
\begin{theorem}[Sawyer's Theorem]
    Consider the voter model on a Markov chain with state space $\ZZ^d$, with $d\geq 2$. Let $\hat{N}_t$ be the number of voters that have the same opinion as the voter at the origin at time $t$.  Assume there exists a function $g(x)\geq 0$ so that $r_{x,y}=g(x-y)$ and $\sum_{x\in \ZZ^d} g(x)=1$. Also assume that 
    \begin{equation*}
    \begin{split}
                &\sum_{x\in \ZZ^d}\norm{x}^2 g(x)<\infty,\quad  \mbox{ if }d=2;\\
   & \sum_{x\in \ZZ^d}\norm{x}^{2+\ep} g(x)<\infty\;  \mbox{ for some }\epsilon>0,\quad \mbox{ if }d>2.
    \end{split}
    \end{equation*}
    Here $\norm{x}=(\sum_{i=1}^d x_i^2)^{1/2}$ for $x=(x_1,\ldots, x_d)$.
    Then we have 
    	\begin{align*}
		\frac{\hat{N}_{t}}{\E\hat{N}_{t}} \xrightarrow{d} \mathrm{Gamma}\left(2,2\right),
	\end{align*}
		as $n\to\infty$, where  $\mathrm{Gamma}\left(2,2\right)$ is the Gamma distribution with probability density function $f_{\mathrm{Gamma}\left(2,2\right)}(x)=4x\exp(-2x)\1[x\geq 0]$.
\end{theorem}
We note that in this theorem $g$ is not required to be an even function so that $\r$ may not be symmetric.
One merit of Sawyer's Theorem is that it  gives  the limiting distribution of 
$\hat{N}_t/\E \hat{N_t}$ in the explicit form. This is very helpful in determining the asymptotics of $P_t$.

We now describe our extensions of Sawyer's Theorem to more general Markov chains. We consider either (1) a sequence of finite transitive Markov chains $(V_n,\r_n)$ with unit jump rate at each state, or (2) random graphs sampled from the configuration model $\CM_n(D)$ (from Definition \ref{defn:configm}) conditioned to be connected. Let $t_n$ be a sequence of times. 
These times and Markov chains or graphs satisfy the conditions in either Theorem \ref{t:config} or \ref{transgraph} or  \ref{uniformtran}.
Let  $\hat{N}_{n,t}$ be the number of voters that have the same opinion as the voter at $\U$ at time $t$, where $\U\sim \mbox{Unif}(V_n$).
\begin{theorem}\label{voterthm1}
	Under either setting, we have
	\[
	\forall k \in \mathbb{N}, \qquad \lim_{n\to\infty} \E\left(\left(\frac{\hat{N}_{n,t_n}}{\E \hat{N}_{n,t_n}}\right)^k\right)=\frac{(k+1)!}{2^k}.
	\]
	Here, if we are under the setting of Theorem \ref{t:config}, the expectation is conditional on the random graph, and the convergence is in the sense of  convergence in probability (i.e., it is a quenched result).
	Moreover, 
	\begin{align*}
		\frac{\hat{N}_{n,t_n}}{\E\hat{N}_{n,t_n}} \xrightarrow{d} \mathrm{Gamma}\left(2,2\right),
	\end{align*}
	as $n\to\infty$.
\end{theorem}
Theorem \ref{voterthm1} follows from our proof of the corresponding theorems: under the setting of Theorem \ref{t:config} or \ref{transgraph} or \ref{uniformtran}, the corresponding results can be obtained from \eqref{eq:509} and Proposition \ref{p:control_ek}, or from \eqref{eq:509}, or from \eqref{eq:intest}, respectively.

We have a similar result for infinite graphs. 
For a unimodular Galton-Watson tree $\G\sim\UGT(D)$, or an infinite transient unimodular transitive graph, let  $\hat{N}_{t}$ be the number of voters that have the same opinion as the voter at the root (in the Galton-Watson tree setting) or a certain vertex (in the infinite transitive graph setting) at time $t$.
\begin{theorem}\label{saw1}
	Under either setting, we have
	\[
	\forall k \in \mathbb{N}, \qquad \lim_{t\to\infty} \E\left(\left(\frac{\hat{N}_{t}}{\E \hat{N}_{t}}\right)^k\right)=\frac{(k+1)!}{2^k}.
	\]
	Moreover, 
	\begin{align*}
		\frac{\hat{N}_{t}}{\E\hat{N}_{t}} \xrightarrow{d} \mathrm{Gamma}\left(2,2\right),
	\end{align*}
	as $t\to \infty$.
\end{theorem}
In the setting of $\G\sim \UGT(D)$ this result follows from Theorem \ref{voterthm1} in the configuration model setting, using the same arguments in the proof of Theorem \ref{t:gwtree}.
In the setting of an infinite transient unimodular transitive graph, Theorem \ref{saw1} follows from equations \eqref{folner} and \eqref{unimoment} in the proof of Theorem \ref{infinitetrans}.

\subsection{Discussions of our results and related previous works}

\subsubsection{The big bang regime}
Recall that in the finite setup, the big bang regime refers to the setting where $t\ll t_{\mathrm{coal}}$, while our results are stated for $t\ll n$ for the configuration model, and $t\ll \M$ for transitive Markov chains. There is in fact no loss at all because in our setups, we always have $t_{\mathrm{coal}}\asymp \M$. We explain this now.

First,  we note that $\coal \ge \M$ holds for any Markov chain. This follows from the monotonicity of the CRW model w.r.t. initial conditions, as can be seen from the graphical representation which allows one to couple CRW with different starting conditions. 

For the other direction,
let
$t_{\mathrm{hit}}(\r)=t_{\mathrm{hit}}:=\max_{x,y}\E_x T_y$ be the maximal expected hitting time of a Markov chain $(V,\r)$, where $T_y$ is the hitting time of the state $y$. 
In \cite{oliveira2012coalescence}, using the powerful random target lemma, Oliveira proved that for a finite reversible chain $(V,\r)$, $t_{\mathrm{coal}}$ is upper bounded by a constant multiplying $t_{\mathrm{hit}}$. 

For transitive chains we have
\begin{equation}  \label{eq:hit-meet}
	t_{\mathrm{hit}}\leq 2\max_x \E_{\pi}(T_x) \leq 4\M, 
\end{equation}
where $\E_\pi$ means that the Markov chain starts from the stationary measure $\pi$. 
Here the first inequality is due to \cite[Lemma 10.2]{levin2017markov} and the second inequality follows from \cite[Proposition 14.5]{aldous1995reversible}.
Thus we conclude that $t_{\mathrm{coal}}\asymp \M$.

For the configuration model, the assumption that $D$ is at least three and upper bounded implies that the relaxation time $\rel$ is bounded by some constant with high probability (see Lemma \ref{l:configroughbound}). Hence the proof of Proposition \ref{p:treluniform} below shows that for $\G_n\sim \CM_n(D)$,
$t_{\mathrm{hit}}\leq Cn$ for some constant $C$ with high probability. Actually $\M(\G_n)$ grows linearly in $n$ (see Lemma \ref{tmmetlim}). Thus for the configuration model we also have $t_{\mathrm{coal}}\asymp \M$. 

There is a recent work \cite{meet} of Kanade, Mallmann-Trenn and Sauerwald, where they provide essentially optimal conditions for $t_{\mathrm{coal}} \asymp \M$. We note that they work in discrete-time settings and that the jumps rates are not assumed to be symmetric.

We also mention that while our results (in the finite setup) concentrate on the big bang regime, results on the regime $t \asymp |V|$ in various setups are already proven in \cite{oliveira2013mean} (see discussions in Section \ref{sssec:meanfb}). 

\subsubsection{A heuristic for the rate of decay}
\label{s:motivation}
The exact decay rate of $P_t$ is predicted by a heuristic of van der Berg and Kesten. They masterfully turned it into a proof for $\ZZ^d$ for $d \ge 3$  \cite{van2000asymptotic,BK} (as well as for a model generalizing CRW when $d \ge 5$). This heuristic involves an `independence approximation' in which the events that far away vertices are occupied are assumed to be approximately independent. In justifying this heuristic van der Berg and Kesten relied on the transience of $\ZZ^d$ for $d \ge 3$ as well as on some other properties of $\ZZ^d$.\footnote{It seems plausible that their arguments can be extended to transitive graphs satisfying a uniform polynomial growth estimate, but we do not believe that their arguments extend to all amenable transitive graphs without some substantial modifications, and likely also some new ideas. In the finite setup, it is not clear if their approach can be pushed to graphs which are just `barely transient' like a $n \times n \times m$ torus with $m \asymp \log n$. Our Theorems \ref{transgraph} and \ref{uniformtran} can treat such graphs.} The starting point is the equation\footnote{For a finite chain $(V,\r)$,  $-\frac{\mathrm{d}}{\mathrm{d}t} P_t=\frac{1}{|V|} \sum_{\text{all unordered }(x,y):\, x \neq y}r_{x,y}\mathbb{P}(x \text{ and }y \text{ are occupied at time }t)$. In the infinite setup \eqref{e:BKheuristic} 
	is valid for unimodular transitive graphs, since applying the mass transport principle to the transport functions $f(x,y)=r_{x,y}\1[xy \in E, x \in \xi_t,y \notin \xi_t ]$ yields     $\sum_{y:\, y \sim x}r_{x,y}\mathbb{P}(x \in \xi_t,y \notin \xi_t )=\sum_{y:\, y \sim x}r_{x,y}\mathbb{P}(x \notin \xi_t,y \in \xi_t )$, from which \eqref{e:BKheuristic} easily follows. }
\begin{equation}
	\label{e:BKheuristic}
	-\frac{\mathrm{d}}{\mathrm{d}t} P_t(x)=\sum_{y:\, y \sim x}r_{x,y}\mathbb{P}(x, y \in \xi_t). 
\end{equation}
\noindent Here recall that $\xi_t$ is the set of vertices that are occupied at time $t$.
Using the graphical representation of the CRW model we can write for any $s \in [0,t]$,
\[\mathbb{P}(x, y \in \xi_t)\approx \sum_{w,z}\mathbb{P}(w, z \in \xi_s)\mathbb{P}(B_{t,s}(w,z,x,y)), \]
where $B_{t,s}(w,z,x,y)$ is the event that the particles which are at $w$ and $z$ reach locations $x$ and $y$ respectively at time $t$. Note that here we write `$\approx$' instead of `=' because there might be particles at locations other than $w$ and $z$ that also reach $x$ or $y$ at time $t$. However, such events make a negligible contribution to $P_t$  provided we let $1\ll t-s \ll t$.
Again by the graphical representation, the probability of $B_{t,s}(w,z,x,y)$ is the same as the probability that two independent random walks started at $w$ and $z$ reach locations $x$ and $y$ respectively at time $t-s$, and do so without colliding. By reversibility, this is the same as $\alpha_{t-s}(x,y,w,z)$, which is defined to be the probability of two independent particles starting from $x$ and $y$ reaching locations $w$ and $z$ at time $t-s$ and doing so without colliding. 
Then we have
\[\sum_{w,z}\mathbb{P}(B_{t,s}(w,z,x,y)) \approx \sum_{w,z} \alpha_{t-s}(x,y,w,z)  =: \alpha_{t-s}(x,y).\] 
We also have the `independence approximation' in the transitive setup: for all $w,z$ that are far way from each other, 
\[\mathbb{P}(w, z \in \xi_s) \approx\mathbb{P}(w \in \xi_s)\mathbb{P}(z \in \xi_s)=P_{s}^2.  \] 
We note that actually, for any vertices $w,z$ one has that $\mathbb{P}(w, z \in \xi_s) \le P_{s}^2  $ \cite{Arratia}. 
When $t-s$ is large, for the sum $\sum_{w,z} \alpha_{t-s}(x,y,w,z)$, the main contribution is from those $w, z$ that are far away from each other.
Thus we have that
\[\mathbb{P}(x, y \in \xi_t) \approx \alpha_{t-s}(x,y)P_s^2.\] 
Note that $\alpha_{t-s}=\sum_{y:\, y \sim x}r_{x,y}\alpha_{t-s}(x,y)$ (where $\alpha_{t-s}=\alpha_{t-s}(x)$ is defined in \eqref{alpha_t(x)} and is independent of $x$ by transitivity). In many cases it is possible to verify that for some $s$ close $t$ we have that $\alpha_{t-s} \approx \alpha_t$ and $P_s \approx P_t$. This gives that $-\frac{\mathrm{d}}{\mathrm{d}t} P_t\approx \alpha_t P_t^2$ which  leads to the prediction:
\begin{equation}
	\label{e:MF1}
	\tag{A1}
	P_t\approx\frac{1}{t \alpha_t}.
\end{equation}
We note equation \eqref{BBSIR} (in the case $d\geq 2$) due to Bramson and Griffeath can also be re-written as \eqref{e:MF1}.
While we do not directly try to justify this independence approximation, our results confirm the conclusion \eqref{e:MF1} in many cases. We take a different approach which has the advantage of also generalizing Saywer's Theorem to other instances.

\subsubsection{Mean field behavior}  \label{sssec:meanfb}

Another motivation for our work comes from a conjecture of Aldous and Fill \cite[Open Problem 14.12]{aldous1995reversible} who suggested to study the distribution of $\Rcoal$ in the (finite) transitive setup under the spectral condition $\rel \ll t_{\mathrm{hit}}$. They conjectured that under this condition the coalescence time should exhibit certain mean-field behaviors. Our results show that in some sense, under appropriate transience assumptions, such mean-field behaviors are valid also during a large part of the big bang regime (for transitive graphs and the configuration model).  

Consider a sequence of graphs with the number of vertices growing to infinity. Aldous and Fill conjectured that in the transitive setup, under the spectral condition  $\rel \ll t_{\mathrm{hit}}$, the law of $\Rcoal$ is close in total variation to the law of $\M \sum_{k=2}^n\binom{k}{2}^{-1}\tau_{k}$ where $\tau_2,\ldots,\tau_n$ are i.i.d.\ Exp$(1)$ and $n$ is the number of vertices in the graph.   
It is instructive to note that for the complete graph this random variable equals $\Rcoal$ in law, and the evolution of the number of particles is the same as in Kingman's coalescent from the moment there are $n$ equivalence classes.
This conjecture was nearly solved by Oliveira \cite{oliveira2013mean} (for the stronger $W_1$ Wasserstein metric), but under the seemingly marginally stronger assumption that the mixing time $t_{\mix}$ satisfies $t_{\mix} \ll t_{\mathrm{hit}}$. 
This condition was recently shown to be equivalent to $\rel \ll t_{\mathrm{hit}}$ in \cite{hermon2019intersection}.

Let $\rho_k$ be the first time when there are $k$ particles. 
Using his earlier work \cite{oliveira2012coalescence}, Oliveira argues that for $k \gg 1$, the time $\rho_{k}=O(\M/k)=o(\M)$ with high probability. Hence the period up to $\rho_k$ can be seen as a short burn-in period. He then argues (in some precise quantitative sense) that if $k$ diverges slowly  enough in terms of some function of $\M/\rel$ then from that point on, the evolution of the process is close to the evolution on the complete graph of the same size $n$ 
with $k$ initial particles, up to a scaling of the time by a factor of $\M/n$. More precisely, he shows that with high probability, simultaneously for all $2 \le i \le k$, once there are $i$ particles left, the time until the next collision is much larger than the $k^{-2}$ mixing time of $i$ independent random walks. 
Hence the mixing time can be viewed as a burn-in period (compared to the time until the next collision), so the coalescence time is close in distribution to the sum $\sum_{i=2}^{k}\hat \tau_i$ where $\hat \tau_2,\ldots,\hat \tau_{k}$ are independent, and $\hat \tau_i$ is the first collision time of $i$ independent walks starting from i.i.d.\ stationary initial distributions. 

By a classic result of Aldous and Brown \cite{aldous}, the assumption $\rel \ll \M$ implies that the collision time between each pair of particles is roughly exponentially distributed with mean $\M$. Oliveira showed that in fact the law of the first collision time among the remaining $i$ particles is close to the law of the minimum of $\binom{i}{2}$ independent Exp$(1/\M)$ random variables, which is the law of $\binom{i}{2}^{-1}\tau_{i}$. He also showed that this holds simultaneously for all $2 \le i \le k$, in the sense that the sum of the error terms over $2 \le i \le k$ is $o(1)$.

Observe that for the complete graph, $\rho_k$ has the same distribution as $\M \sum_{i =k+1}^{n}\binom{i}{2}^{-1}\tau_{i}$. This leads to  the following mean field approximation for the density of particles: for $t \ll \M$ we have 
\begin{equation}
	\label{e:MF2}
	\tag{A2}
	P_t\approx \frac{2\M}{tn}.
\end{equation}
Here $t \ll \M$ corresponds to $k \gg 1$, which implies that $\M \sum_{i =k+1}^{n}\binom{i}{2}^{-1}\tau_{i}$ is concentrated, from which \eqref{e:MF2} follows by a simple calculation.

The number $k$ from Oliveira's arguments is allowed to diverge only as some function of $t_{\mathrm{hit}} /t_{\mix}$. Hence it could be the case that it is only allowed to diverge arbitrarily slowly, and the fastest it may be allowed to diverge is like $o(\sqrt{n})$ (for expanders).

The results in this work verify the validity of \eqref{e:MF2} under certain finitary notions of transience assumptions as well as the spectral condition $\rel \ll t_{\mathrm{hit}}$ proposed by Aldous and Fill for times $t \gg 1$ or $t \gg \rel$ (depending on the assumptions being made).  For instance, for transitive expanders and for the configuration model we allow any $1 \ll t \ll n$. This corresponds to taking any $1 \ll k \ll n$. 

Interestingly, like Oliveira, a key quantity of interest for us is $\mathbb{P}(\sum_{i=2}^{k}\hat \tau_i \le t)$, as it happens to be the $(k-1)$-th moment from Sawyer's Theorem at time $t$. We verify that, up to a scaling of the time by $\M/n$, this probability is well approximated by the corresponding probability for the complete graph, i.e., the mean field counterpart. As in \cite{oliveira2013mean}, this essentially shows that $\hat \tau_2,\ldots,\hat \tau_{k}$ are asymptotically independent, and that each $\hat \tau_i$ can be thought of as the minimum of $\binom{i}{2}$ independent collision times. The main difference is that, here we have to do some more careful analysis of this probability for $t$ much smaller than $\M$. While Oliveira estimates this probability up to some $o(1)$ additive error term, for the small $t$ that we consider this error term may be of a much higher order than the probability being estimated. One strategy we take is to reverse time and study a certain type of random walks with branching. See Section \ref{s:method} for details.

\subsubsection{Assumptions for transitive chains} \label{ss:assumptions}

In this section, we discuss the assumptions that appeared in Theorem \ref{transgraph} and Theorem \ref{uniformtran}.

We note that under transitivity the condition $\M(\r_n) \lesssim |V_n|$ is equivalent to $t_{\mathrm{hit}}(\r_n) \lesssim |V_n| $, by the fact that $2\M(\r_n) \le t_{\mathrm{hit}}(\r_n)$ and \eqref{eq:hit-meet}.
These conditions are also equivalent to several other conditions. For example, if the transitive Markov chain $(V_n, \r_n)$ is replaced by a $d_n$-regular transitive graph $\G_n$, it is equivalent to  $d_n\mathcal{R}_{\mathrm{max}}(\G_n)\lesssim 1$, where $\mathcal{R}_{\mathrm{max}}(\G_n)$ is the maximal effective resistance between a pair of vertices. It can also be shown to be equivalent to $\int_0^{\rel(\G_n)}p_t^{(n)}(x,x)\mathrm{d} t \lesssim 1$ (here we let $r_{x,y}=\1[x \sim y ]/d_n$ instead of $\1{[x\sim y]}$, and recall that $p_t^{(n)}$ is the transition probability of $(V_n,\r_n)$).

We now briefly explain the equivalence between $\int_0^{\rel(\G_n)}p_t^{(n)}(x,x)\mathrm{d} t \lesssim 1$ and $\M(\r_n) \lesssim |V_n|$. 
First one can show that
$\int_0^{\rel(\G_n)}p_t^{(n)}(x,x)\mathrm{d} t \lesssim 1$ if and only if $\int_0^{|V_n|}(p_t^{(n)}(x,x) - 1/|V_n|) \mathrm{d} t \lesssim 1$, which is also  equivalent to
$\int_0^{\infty}(p_t^{(n)}(x,x) - 1/|V_n|) \mathrm{d} t \lesssim 1$.
Then the equivalence follows from the facts that $\int_0^{\infty}(p_t^{(n)}(x,x) - 1/|V_n|) \mathrm{d} t=\E_{\pi}T_x/|V_n|$ and $\max_{x} \E_{\pi}T_x \ge 2 \max_{x,y} \E_{x,y}\RM$ (this holds even without transitivity). 
See the proof of Proposition \ref{p:treluniform} for details. 

We next compare the assumptions in
Theorem \ref{transgraph} and Theorem \ref{uniformtran}.
In the case where the transitive Markov chain $(V_n, \r_n)$ is replaced by a $d_n$-regular transitive graph $\G_n$ with $d_n=O(1)$, following recent works \cite{TT, tessera2020sharp}, it is shown in  \cite[Proposition 2.10]{Gumbel} that the condition \eqref{defuniftrans} is implied by 
\begin{equation}\label{implyunif}
	\mathrm{diam}(\G_n)^2\ll \abs{V_n}/\log \abs{V_n}.
\end{equation}
Condition \eqref{implyunif} provides a simple and easily verifiable sufficient condition for uniform transience.
For example, the $d$-dimensional torus of side length $n$ with $d\geq 3$ satisfies \eqref{implyunif}.
We note that it follows from \cite[Proposition 2.10]{Gumbel} that the condition $\mathrm{diam}(\G_n)^2\lesssim \abs{V_n}/\log \abs{V_n}$ implies that $\M(\G_n) \asymp |V_n|$, which implies part of the assumptions in Theorem \ref{transgraph}.
On the other hand, if $\mathrm{diam}(\G_n)^2 \gg \abs{V_n}$ and $\G_n$ is transitive of (uniformly in $n$) bounded degree, we have $\M(\G_n)\gg |V_n|$, so the assertion of Theorem \ref{transgraph} fails.\footnote{It follows from \cite{TT} that $\G_n$ satisfies a certain technical condition due to Diaconis and Saloff-Coste known as moderate growth, for which it is known that the mixing time is proportional to the diameter square. In our case, by the assumption $\mathrm{diam}(\G_n)^2 \gg \abs{V_n}$, this means that the mixing time and the meeting time are much larger than the number of vertices.}  It follows that there are relatively few instances satisfying the assumptions in Theorem \ref{transgraph} but not those in Theorem \ref{uniformtran} (i.e., uniform transience). Two such examples are an $n \times n \times \lceil \log n \rceil$ (discrete) torus and a Cartesian product of the $n$-cycle with a transitive expander of size  $n$.
Finally, we remark that in \cite[Theorem 1.8]{tessera2020sharp}, using results from \cite{TT} a different analog of transience is proven for finite transitive graphs satisfying a condition similar to \eqref{implyunif}.

\subsubsection{Some other existing bounds on $P_t$}  \label{s:related}

For a graph $\G=(V,E)$, Benjamini, Foxall, Gurel-Gurevich, Junge and Kesten \cite{benjamini2016site} proved the lower bound $P_t(x)\geq 1/(1+d_{\max} t)$ for all $x \in V$, where $d_{\max}$ is the maximal degree of $\G$. They also pointed out in \cite[Proposition 1.4]{benjamini2016site} that, for CRW on any infinite  graph with unit jump rate along each edge, $P_t(x)\leq (1+\ep) /(2\sqrt{\pi t})$ for all $x\in V$, $\ep>0$ and $t$ large enough (depending only on $\ep$). 
Later Foxall, Hutchcroft, and Junge \cite{foxall2018coalescing} showed, using an elegant application of the mass transport principle, that for any unimodular random rooted graph (with root $o$) we still have the bound $\E(P_t(o))\geq 1/(C_1+C_2t )$ for some constants $C_1,C_2$. Here  the expectation is taken over the law of the rooted graph.

As we mentioned before, 
in \cite{oliveira2012coalescence}  Oliveira  proved that for any finite reversible chain $(V,\r)$ one has that
$t_{\mathrm{coal}}\leq C t_{\mathrm{hit}}$ for some universal constant $C$.   Oliveira's proof  actually shows that for some absolute constant $C>0$, we have $\min\{P_t,1/\abs{V}\} \le \delta$ if $t \ge Ct_{\mix}+C\frac{t_{\mathrm{hit}}}{|V|\delta}$.

\subsection{Ingredients of our proofs}   \label{s:method}
As mentioned before in Section \ref{ss;mainresults}, we extend the analysis from \cite{bramson1980asymptotics} beyond $\ZZ^d$.
In the finite setup, we uniformly pick a particle at the beginning, and let $N_t$ be the number of particles that coalesce with it by time $t$, including itself. It has the same distribution as $\hat{N}_{t}$, the size of the set of vertices which have the same opinion as the origin at time $t$ in the voter model, via duality; see Section \ref{s:prelim} for a formal definition and discussions.
In the finite setup we also have that $P_t=\E[N_t^{-1}]$ (see Lemma \ref{P_t=E(N_tinverse)} below). Hence it suffices to study the distribution of $N_t$ via computing the moments of $N_t$, and controlling the behavior of $N_t/\E N_t$  near $0$
(specifically, showing that $\E((N_t/\E N_t)^{-1}\1[N_t/\E N_t \leq \ep])$ goes to 0 as $\ep\to 0$, uniformly in $t$), since then one can go from the distributional convergence of $N_t/\E N_t$ to  the convergence of $\E((N_t/\E N_t)^{-1})$.

In our paper we give an upper bound of $P_t$ for general Markov chains (Theorem \ref{t:upperbound} below), which enables us to close the gap between the moments of $N_t$ and the asymptotics of $P_t$ in various setups, whereas the proof of such bounds in \cite{bramson1980asymptotics} relies heavily on the geometric structure of $\ZZ^d$ (e.g., they used a partition of $\ZZ^d$ into boxes).

As in \cite{bramson1980asymptotics}, we also use moment calculations in our proofs. Such calculations in \cite{sawyer1979limit} are done using the voter model, and the arguments also rely on the lattice structure and are difficult to generalize. 
Instead, we directly analyze the CRW model, using a quite different approach. We mostly work in the finite setup, where some care is required since the limiting arguments have to be replaced by quantitative ones, and in particular when considering non-transitive graphs as in Theorem \ref{t:config}. 
The calculation of the $k$-th moments of $N_t$ can be reduced to the probability that $k+1$ particles starting from independent stationary positions coalesce by time $t$. 

By classic results of Aldous and Brown (see Section \ref{s:aldous} for details), under the assumption  $\rel\ll t\ll \M$ the meeting time  $\RM$ of two independent stationary walks is roughly distributed as $\Exp(1/\M)$, and the probability density is roughly $\frac{\exp(-t/\M)}{\M}$ for $t \gg \rel$. 
For $k+1$ particles, we consider the $k$ collision times.
We will show that its probability density function at $(t_1,\ldots t_k)$ is roughly $\frac{1+o(1)}{\M^k}$. Here the tuple $(t_1,\ldots t_k)$ satisfies that $0<t_1<\cdots <t_k<t$, and that each $t_{i+1}-t_i$ is not `too small'.
Specifically, in establishing this we wish to have that for all $1 \le i \le k$, the joint distribution of the locations of the remaining particles at the $i$-th collision time is close to $\pi^{\otimes (k+1-i)}$.
However, given information on past collisions, 
the joint distribution of the locations of particles can be complicated. 
We get around of this by showing that this information will have negligible effect after a short amount of time since the previous collision. Here `a short amount of time' sometimes means a time interval large compared to $\rel$, and sometimes means any large time interval.
For transitive graphs, such analysis relies on the transitivity and transience conditions.
For the configuration model, the relevant analysis is more involved and subtle due to the lack of transitivity. In this case we rely on several facts: (i) the configuration model graphs are expanders with high probability (this is the reason we require that the minimal degree is at least $3$) (ii) the law of the local structure of the transitive graph is the same around each vertex. See Remark \ref{rmk:config} for more general graphs that our approach can be applied to.

Another ingredient in these analysis is using the reversibility of Markov chains.
Thus we transform the probability density function of the $k$ collision times into the probability of having no collisions in a certain branching random walk process, in which $k$ new particles are created at time $t-t_k,t-t_{k-1},\ldots,t-t_1$ (see Section \ref{ss:branchstruc}). By a careful analysis of this non-collision probability we can single out all the factors that contribute to the error, and such analysis is summarized as Proposition 
\ref{p:gengraph}, which is our general framework. 
For transitive chains we can make some refinements and obtain Proposition \ref{p:uniftrans}. 
We then apply these frameworks to different families of graphs and obtain the above mentioned results.

\subsection{Further problems}
\label{s:open}

Starting from this work there are several directions that could be further developed. The first is on whether the assumptions on the degree distribution in Theorem \ref{t:config} can be relaxed. In particular, since the configuration model with Poisson degree distribution is contiguous to Erd\"os-R\'enyi graphs (see, e.g., \cite[Theorem 2.1]{kim2007poisson}), it would be very interesting to investigate the following problem. 
\begin{open}
    Fix any $\mu>1$ and let $D$ be $\mathrm{Poisson}(\mu)$.
Let $\G_n$ be sampled from $\CM_n(D)$ and $\bar{\G}_n$ be the giant component of $\G_n$. Let $\G\sim \UGT(D)$ conditioned to survive forever and define $\alpha(D)$ as in \eqref{eq:alpha(D)}. Do \eqref{eq:thm11} and \eqref{eq:thm12} hold  for $P_t(\bar{\G}_n)$?
\end{open}

Then one can ask whether the results for (finite or infinite transitive graphs) still hold under some other conditions.
\begin{open}
	Let $\G=(V,E)$ be an infinite transitive graph. Assume that the size of the balls of radius $r$ is at least $cr^2$ for some constant $c$. Does \eqref{e:MF1} hold? Here the $r^2$ comes from the fact that \eqref{e:MF1} holds for $\ZZ^d$ with $d\geq 2$. 
\end{open}
\begin{open}
	Let $\G_n=(V_n,E_n)$ be a sequence of finite  transitive graphs such that $|V_n|\to \infty$ as $n\to\infty$, and all graphs have the same degree. 
	As mentioned in Section \ref{ss:assumptions}, under the condition  $\mathrm{diam}(\G_n)^2 \ll |V_n|/\log |V_n|$, \eqref{e:MF1} holds for a sequence of times $t_n$ with $1 \ll t_n \ll \M(\G_n)$.
	Can we weaken this condition to $\mathrm{diam}(\G_n)^2 \lesssim |V_n|$?
\end{open}
One could also ask if our results can extend to some other random graphs.
\begin{open}
	Can we extend our results on unimodular Galton-Watson trees to general almost surely transient reversible random graphs, where the root has finite expected degree? (See Chapter 3 of \cite{curien2018} for the notion of reversible random graphs)
\end{open}

The next direction is to get more refined information on the fluctuation of the number of particles (in a finite subset $A$ of the states).
\begin{open}
	As commented before Section \ref{ss;mainresults}, for a large finite subset $A$ of the state space,  $| \xi_t \cap A |$ has  $O((\E | \xi_t \cap A |)^{1/2})$  fluctuations around its expectation. Can we upgrade this to show that in fact (for appropriately large $t$) the fluctuations are asymptotically Gaussian? If so, can we determine the variance? 
\end{open}

\subsection*{Organization of the remaining text}
The rest of this paper is organized as follows. In Section \ref{s:prelim} we set up CRW and the voter model, and their relations, and provide basic tools for Markov chains. 
Then we first prove non-sharp general bounds for the density $P_t$, in Section \ref{s:generalbounds}.
In Section \ref{s:framework} we lay out the framework for the more careful analysis.
In Section \ref{s:gengraph} we carry out the estimates for general Markov chains, and all the arguments are packed into Proposition \ref{p:gengraph}.
In Section \ref{s:pfrandom} we provide estimates that are specific for the configuration model, and prove Theorems \ref{t:config} and \ref{t:gwtree}.
In Section \ref{s:finitetrans} we prove Theorems 
\ref{transgraph}, \ref{uniformtran}, and Theorem \ref{infinitetrans} for finite and infinite transitive Markov chains. 

\section{Preliminaries: the model and Markov chains}

\subsection{CRW and the voter model}\label{s:prelim}
	In the voter model, for any $t\in \R_{\ge 0}$, and any vertex $x\in V$, we denote $\zeta_t(x)$ as the opinion of $x$ at time $t$.
	Let $\zeta^x_t$ be the set of voters that at time $t$ have the same opinion as voter $x$ at time $0$,  
	i.e., 
	\[
	\zeta^x_t:=\{y\in V:\zeta_t(y)=\zeta_0(x)\}.
	\]
	For CRW starting with one particle at each vertex, recall that we use $\xi_t \subset V$ to denote the locations of the remaining particles at time $t$.
The duality between the voter model and CRW implies
\begin{equation}\label{duality}
	P_t(x)=\P(x\in \xi_t)=\P(\zeta^x_t\neq \emptyset),
\end{equation} 
see, e.g., \cite[Section 3]{durrett1995ten} or \cite[Chapter 14]{aldous1995reversible} for an introduction to the duality. 

Consider paths $P_1, \ldots, P_k:\RR_{\ge 0} \rightarrow V$ satisfying the following property: for all $i,j \in \{1,\ldots,k\}$ and $s\ge 0$, if $P_i(s)=P_j(s)$ then $P_i(t)=P_j(t)$ for all $t \ge s$. We denote the coalescence time as
\[
\C(P_1,\ldots, P_k):=
\inf\left\{t \in\R_{\ge 0} \colon P_1(t)=\cdots=P_k(t)\right\}\cup\{\infty\}.
\]

Throughout this paper, we use the random functions $X_1, \ldots, X_k:\RR_{\ge 0} \rightarrow V$ to denote the locations of $k$ particles  performing CRW via the graphical representation of the model.
The starting locations are $X_i(0), 1\leq i\leq k$, which can be random (while independent of the walks) or deterministic.

We denote by $N_t$ the following random variable.
Suppose that $n=|V|$ and $X_1(0), \ldots, X_n(0)$ is an enumeration of all vertices in $V$.
Let $\iota$ be uniformly chosen from $[n]:=\{1,\ldots,n\}$. We define
\[      N_t:=|\{i:\C(X_i, X_\iota)\leq t \}|.
\]
Note that we always have $N_t\geq 1$ since $\C(X_\iota, X_\iota) = 0$. 

Our analysis of the decay of $P_t$ relies on its following relation with $N_t$.
\begin{lemma}\label{P_t=E(N_tinverse)}
For finite graphs we have $P_t= \E(N_t^{-1})$.
\end{lemma}

\begin{proof}
	Take $\U\sim \Unif(V)$.
	Define $\hat{N}_t:=\abs{\{x\in V:\zeta_t(x)=\zeta_t(\U) \}}$.
	Then $\hat{N}_t$ has the same distribution as $N_t$ by duality (see, e.g.,  \cite[Section 3]{durrett1995ten}), which implies that it suffices to prove $P_t= \E(\hat{N}_t^{-1})$. 
	Now we claim
	\begin{equation}\label{Nteq}
		\forall k \ge 1, \qquad  \P(\hat{N}_t=k)=k\P( \abs{\zeta^\U_t}=k).
	\end{equation}
	Indeed, this is by an interchange of summations:
	\[
	\begin{split}
		\P(\hat{N}_t=k)&=\frac{1}{n}\sum_y \P(\abs{\{x:\zeta_t(x)=\zeta_t(y) \}}=k )\\
		&=\frac{1}{n}\sum_y \sum_z \E(\1[y\in \zeta^z_t, \abs{\zeta^z_t}=k ])\\
		&=\frac{1}{n}\sum_z \E \left( \sum_y \1[y\in \zeta^z_t, \abs{\zeta^z_t}=k ]\right)\\
		&=\frac{k}{n} \sum_z \E(\1[\abs{\zeta^z_t}=k ])\\
		&=k\P( \abs{\zeta^\U_t}=k ).
	\end{split}
	\]
	Now the assertion of this lemma follows from \eqref{duality} and \eqref{Nteq}:
	\[
	P_t=\P(\zeta^\U_t \neq \emptyset)=\sum_{k=1}^{\infty}
	\P( \abs{\zeta^\U_t}=k )=\sum_{k=1}^{\infty} \frac{1}{k}
	\P(\hat{N}_t=k)=\E \hat{N}_t^{-1}=\E N_t^{-1}.
	\]
	Thus the conclusion follows.
\end{proof}
Throughout we write $n^x_t:=\abs{\zeta^x_t}$ and $n_t := n^{\U}_t$ where $\U\sim \Unif(V)$ is independent of  $\{n^x_t,x\in V\}$.
For future reference we also record the following corollary.
\begin{corollary}  \label{c:Ntnt1}
	For any $M \in \ZZ_+$, we have $\E(N_t^{-1}\1[N_t \ge M]) =\P(n_t \ge M)$ and $\E(N_t^{-1}\1[N_t\le M]) =\P(0<n_t \le M)$ in finite graphs.
\end{corollary}
This is a direct consequence of \eqref{Nteq}.

\subsection{Classical results on Markov chains}\label{ss:relaxmixtime}
In this subsection, we recall some important classical concepts and results on Markov chains. Some other detailed results and proofs are deferred to Section \ref{as:markov}.

\subsubsection{The Poincar\'{e} inequality}
\begin{definition}
	For any Markov chain $(V, \r)$ with a stationary distribution $\pi$, let $\ell_2(V,\pi)$ be the space of all functions $f:V\rightarrow \R$ with $\|f\|_2^2:=\sum_{x\in V}f(x)^2\pi(x) < \infty$.
	We define the inner-product $\langle \cdot,\cdot \rangle_{\pi}$ induced by $\pi$ on $\ell_2(V,\pi)$ as
	\[
	\langle f,g \rangle_{\pi}:=\sum_{x\in V}\pi(x) f(x)g(x).
	\]
	And we define
	\[ 
	\Var_{\pi}(f):=\|f-\E_{\pi}(f)\|_2^2,\quad \E_{\pi}(f):=\sum_{x\in V}\pi(x) f(x).
	\]
	Let $p_t=e^{tQ}$ be the time $t$ transition kernel of the Markov chain $(V, \r)$, where $Q(x,y)=r_{x,y}$ for $x \neq y$ and $Q(x,x)=-r(x)$. For a function $f \in \ell_2(V,\pi)$, the function $p_tf \in \ell_2(V,\pi)$ is defined as 
	\[
	p_t f(x):=\sum_{y\in V} p_t(x,y)f(y).
	\]
\end{definition}

Recall that $\rel$ is the inverse of the spectral gap (the second smallest eigenvalue of $-Q$).
By the Poincar\'e  inequality (see, e.g., \cite[Lemma 20.5]{levin2017markov}), we have the following result.
\begin{lemma}
	Given any continuous-time Markov chain $(V,\r)$ (not necessarily symmetric),
	for $t \ge 0$ and $f\in \ell_2(V,\pi)$, 
	\begin{equation}\label{prepoincare}
		\mathrm{Var}_{\pi}(p_{t} f)\leq \exp\left(-\frac{2t}{\rel}\right) \mathrm{Var}_{\pi}(f).
	\end{equation}
\end{lemma}
This lemma implies the following result.

\begin{lemma}  \label{l:poincareineq}
	Let $(V,\r)$ be a finite-state continuous-time Markov chain with symmetric transition rates $\r$.
	For $t>0$ and a probability distribution $\mu$ on $V$, we denote by $\mu_t$ the distribution on $V$ given by
	\[\mu_t(x):=\sum_y \mu(y)p_t(y,x)=\sum_y p_t(x,y)\mu(y)=p_t\mu(x), \quad \forall x \in V\]
	which is the law of $X(t)$ for a walker $X$ with initial location $X(0) \sim \mu$.
	Then for any $t,s>0$, we have
	\begin{equation}\label{tvd}
		\sum_x \left(\mu_{t}(x)-\frac{1}{|V|}\right)^2 \leq \exp\left(-\frac{2t}{\rel}\right)\sum_x\left(\mu(x)-\frac{1}{|V|}\right)^2,
	\end{equation}
	and
	\begin{equation}\label{poincare}
		\max_{x,y}p_{t+s}(x,y) - \frac{1}{|V|} \leq \exp\left(-\frac{t}{\rel}\right)
		\left(\max_x p_s(x,x)-\frac{1}{|V|}\right).
	\end{equation}
\end{lemma}

We define the $\ell_{\infty,\pi}$ distance between two probability measures $\mu$ and $\lambda$ w.r.t.\ $\pi$ as
\[
\norm{\mu-\lambda}_{\infty,\pi}:=\sup_{x} \frac{\abs{\mu(x)-\lambda(x)}}{\pi(x)}.
\]
Using the Poincar\'{e} inequality, we have the following corollary of Lemma \ref{l:poincareineq}.

\begin{corollary}
	\label{linftyctl}
	For any finite-state continuous-time symmetric Markov chain, 
	starting from any initial distribution $\mu$, we have 
	$\norm{\mu_t-\pi}_{\infty,\pi} \leq \frac{1}{|V|}$ for $t\geq 2\rel \log |V|$.
\end{corollary}

\subsubsection{Mixing time and relaxation time}
Now we define the mixing time of a Markov chain.
\begin{definition}\label{d:mixtime}
	The \emph{total variation distance} between two distributions $\mu$ and $\nu$ on $(V, \r)$ is defined as 
	\[
	\norm{\mu-\nu}_{\mathrm{TV}}=\frac{1}{2}\sum_x \abs{\mu(x)-\nu(x)}.
	\]
	For $\varepsilon \in (0,1)$, the $\varepsilon$-\emph{total variation mixing time} is defined as 
	\[t^{\mathrm{TV}}_{\mix}(\varepsilon):=\inf\{t\geq 0: \max_{x \in V} \norm{ p_t(x,\cdot)-\pi}_{\mathrm{TV}}\leq \varepsilon \}
	,\] while the $\varepsilon$-$\ell_{\infty}$ mixing time is defined as
	\[
	t^{\infty}_{\mix}(\varepsilon):=\inf\{t\geq 0:\max_{x \in V} \norm{p_t(x,\cdot)-\pi}_{\infty,\pi} \leq \varepsilon \}.\]
\end{definition}

Under Definition \ref{d:mixtime}, Corollary \ref{linftyctl} can be restated as $t^{\infty}_{\mix}(1/|V|)\leq 2\rel \log |V|$. We also have the following inequality
\begin{equation}\label{ltmix}
	\max_{x \in V} \norm{ p_{Lt^{\infty}_{\mix}(\varepsilon)}(x,\cdot)-\pi}_{\mathrm{TV}}\leq (2\varepsilon)^L
\end{equation}
for any integer $L\geq 1$.
See \cite[Equation (4.30)]{levin2017markov}. 

In the case of a SRW on a graph $\G=(V,E)$,
one way to control $\rel$ is through the  \emph{vertex expansion constant} $\kappa(\G)$, defined as
\begin{equation}\label{d:kappa}
	\kappa(\G):= \min_{0 < |S| \leq \frac{|V|}{2}} \frac{|\partial S|}{|S|},
\end{equation}
where $\partial S:= \{v \in V\setminus S: \exists u \in S, u\sim v\}$ is the out-boundary of $S$. We have the continuous-time analogue of Cheeger's inequality. 
\begin{lemma}\label{gapandtrel}
	For any graph $\G$, its spectral gap for a
	SRW
	is lower bounded by $\frac{\kappa(\G)^2}{2d_{\max}}$ where $d_{\max}$ is the maximal degree of vertices in $\G$. 
\end{lemma}
For a proof of the discrete-time version of  Lemma \ref{gapandtrel} one can see Lemma 3.3.7 in \cite{saloff1997lectures}.

For future reference we record the following result about transition probabilities on graphs.
It is a direct consequence of Lemma  \ref{l:poincareineq} and Lemma \ref{gapandtrel}.

\begin{lemma}  \label{l:boundH}
	For any $\kappa_0$ and $\bar{d}>0$, there is a constant $C_{\kappa_0, \bar{d}}$ such that the following is true.
	For any graph $\G$ such that the maximum degree $d_{\max} \leq \bar{d}$ and $\kappa(\G) > \kappa_0$, 
	we have
	\[
	\forall \, t \leq |V|, \qquad \frac{\max_x \int_0^t  p_{s}(x,x)\mathrm{d} s}{\min_x \int_0^t p_{s}(x,x)\mathrm{d} s}+\max_x \int_0^t  p_{s}(x,x)\mathrm{d} s < C_{\kappa_0, \bar{d}}.
	\]
\end{lemma}

\subsubsection{Kac's formula and an exponential approximation for Markov chains}\label{s:aldous}
In this paper we will frequently use the following Kac's formula. Consider the case where $V$ is finite. Let $A,B \subset V$. Let $ Q(A,B):=\sum_{x\in A,y\in B} \pi(x) r_{x,y}$. The `exit-measure' $\nu_A$ is defined as
\begin{equation}\label{exitmeasure}
	\nu_A(x):=\frac{\sum_{y\in A} \pi(y)r_{y,x}}{Q(A,A^c)}=\lim_{t \to 0}\mathbb{P}_{\pi}(X_t=x \mid X_0 \in A,X_t \notin A).
\end{equation}
Note that when $A$ is a singleton, this definition coincides with the one given in Definition \ref{defSRW}. 

\begin{lemma}\label{kaceqns} For all $t \ge 0$ and $A \subsetneq V$,
	\begin{equation}\label{kac1}
		\P_{\pi}(T_A\in (t,t+\mathrm{d} t))=Q(A,A^c)\P_{\nu_A}(T_A>t)\mathrm{d} t,
	\end{equation}
	where $T_A:=\inf\{s\geq 0: X_s\in A\}$ is the hitting time of $A$.
	Integrating this  equation over $t>0$ we get 
	\begin{equation}\label{kac2}
		\pi(A^c)=Q(A,A^c)\E_{\nu_A}(T_A).
	\end{equation}
\end{lemma}
This lemma can be found in many places, e.g.,
\eqref{kac1} is \cite[Equation (2.23)]{aldous1995reversible} and \eqref{kac2} is \cite[Equation (2.24)]{aldous1995reversible}.\footnote{Note there is a typo in their definition of $Q(A,A^c)$ between (2.23) and (2.24) in \cite{aldous1995reversible}: $q_{ij}$ should be replaced by $\pi_iq_{ij}$ there.}   

A powerful tool for controlling the cumulative distribution function as well as the probability density function of $T_A$ with initial condition $\pi$ is
the following classical exponential approximation result, due to Aldous and Brown \cite{aldous}. It will also be used extensively throughout this paper. We state it as a lemma here. 
\begin{lemma}[Exponential approximation for hitting times \cite{aldous}]\label{aldous1}
	For an irreducible 
	reversible Markov chain on a finite state space $V$ with stationary distribution $\pi$ and 
	$A \subset V$, if we denote the hitting time of $A$ by $T_A$ and its density function w.r.t.\ the stationary chain by $f_{T_{A}}$, then
	\begin{equation}\label{probestimate1}
		\abs{\P_{\pi}(T_A>t)-\exp\left(-\frac{t}{\E_{\pi}T_A}   \right)}
		\leq \frac{\rel}{\E_{\pi}T_A},
	\end{equation}
	and
	\begin{equation}\label{ABdensity}
		\frac{1}{\E_{\pi}T_A}\left( 1-\frac{2\rel+t}{\E_{\pi}T_A}\right)
		\leq 
		f_{T_A}(t)\leq \frac{1}{\E_{\pi}T_A}\left(1+\frac{\rel}{2t}\right). 
	\end{equation}
	As a consequence of \eqref{probestimate1}, for $t\leq \E_{\pi}T_A$, we have
	\begin{equation}\label{probestimate}
		\abs{\P_{\pi}(T_{A}\leq t)-\frac{t}{\E_{\pi}T_A}}\leq
		\left(\frac{t}{\E_{\pi}T_A}\right)^2+\frac{\rel}{\E_{\pi}T_A}.
	\end{equation}
\end{lemma}
Equations \eqref{probestimate1} and \eqref{ABdensity} are due to Theorem 1 and Lemma 13 in \cite{aldous}, whereas \eqref{probestimate} follows from \eqref{probestimate1} and the inequality that $1-x\leq e^{-x} \leq 1-x+x^2$ for $0\leq x\leq 1$.

We note that, very recently  Berestycki, Hermon and Teyssier give an elementary proof of a stronger version of Lemma \ref{aldous1}. See \cite[Chapter 4]{Gumbel}. 
\subsubsection{Relations between Markov chain parameters}
We will need the following lemma, giving relations  between $\M,\rel$ and $r_{\max}:=\max_{x \in V}  r(x)$ for a finite Markov chain.
\begin{lemma}\label{walkparamaters}
	For any Markov chain $(V,\r)$  with $\abs{V}=n\geq 2$ and symmetric jump rates $\r$, we have
	\begin{equation}\label{relandrmax}
		\rel\geq 1/(2r_{\max}), 
	\end{equation}
	\begin{equation}\label{mrmax/n}
		\frac{\M r_{\max}}{n} \geq \left(1-\frac{1}{n}\right)^2\geq \frac{1}{4},
	\end{equation}
	\begin{equation}\label{maxxyp}
		p_s(x,y)\leq \frac{1}{2}(p_s(x,x)+p_s(y,y)), \,\forall x,y\in V. 
	\end{equation}
\end{lemma}

\subsubsection{Bounds for meeting probabilities}
For future reference, we record the following lemma which controls the probability for two walks to collide within some time interval.
\begin{lemma}\label{meetprob}
	Assume that $\r$ is symmetric. For any $x,y \in V$ with $x\neq y$ and any $0<T<t$, we have
	\[\P_{x,y}(T< \RM<t) \le 2\exp(-T/\rel)
	\frac{\max_{z}\int_0^{2T}  p_{s}(z,z)\mathrm{d} s }{\min_{z}\int_0^{2T}  p_s(z,z)\mathrm{d} s  }+\frac{8t}{n} (T^{-1}\vee r_{\max}).\]
	Here recall that $\P_{x,y}$ denotes two independent continuous-time random walks starting from $x$ and $y$ respectively.
\end{lemma}

\subsubsection{Meeting time and integral of transition probability for transitive chains}
In the setting of transitive chains, the quantity $\int_0^t p_{2s}(z,z)\mathrm{d} s$ satisfies, for $t\geq \rel$, 
\begin{equation}\label{b1}
	n \int_0^t p_{2s}(x,x)\mathrm{d} s=\int_0^{t}\mbox{Tr}(p_{2s})\mathrm{d} s 
	= \int_0^t \sum_{i=1}^n \exp(-2\lambda_i s)\mathrm{d} s
	\in \left[\frac{1}{4}\sum_{i=2}^n \frac{1}{\lambda_i}+t,\frac{1}{2}\sum_{i=2}^n \frac{1}{\lambda_i}+t\right],
\end{equation}
where $n=|V|$, and $\lambda_1=0<\lambda_2 \leq \ldots \leq \lambda_{n} \le 2 r_{\mathrm{max}}$ are the eigenvalues of $-Q$.
Moreover, 
\[
n\int_0^t p_{2s}(x,x)\mathrm{d} s \le t+ \sum_{i=2}^{n}\frac{1}{2\lambda_i}
\] 
holds for all $t\geq 0$. 
The eigentime identity (see, e.g., \cite[Lemma 12.17]{levin2017markov}) asserts that for all $x$,
\[
\sum_{i=2}^n \frac{1}{\lambda_i}=\sum_y \pi(y) \E_x T_y,
\]
where $T_y$ is the hitting time of $y$. 
For transitive chains we have $\E_x T_y=2\E_{x,y} \RM$ for all $x,y$ (see, e.g., \cite[Section 14.2]{aldous1995reversible}). Hence 
\begin{equation}\label{lambdameet}
	\sum_{i=2}^n \frac{1}{\lambda_i}=2\M. 
\end{equation}
Using \eqref{b1} and \eqref{lambdameet}
we see that for $t\geq \rel$,
\begin{equation}\label{integralofp}
	\frac{\M}{2n}\leq \int_0^t p_{2s}(x,x)\mathrm{d} s=\frac{1}{n}\int_0^{t}\mbox{Tr}(p_{2s})\mathrm{d} s
	\leq \frac{\M+t}{n}.
\end{equation}

\subsubsection{Uniformly transient Markov chains}
We state the following proposition which we will use in the proof of Theorem \ref{uniformtran}. 
\begin{prop}\label{p:treluniform}
	Let $(V_n,\r_n)$ be a sequence of uniformly transient transitive Markov chains with $r_n(x)=1$ for all $n$ and $x\in V_n$, and $|V_n|\to\infty$ as $n\to\infty$. Then we have
	\begin{equation}
		\rel(\r_n)=o(\M(\r_n)) \mbox{ and } \M(\r_n)=\Theta(\abs{V_n}),
	\end{equation}
	as  $n\to\infty$.
\end{prop}

\section{General bounds on $P_t$}  \label{s:generalbounds}
In this section, we present an upper and a lower bound for $P_t$. For many families of graphs, including all transitive graphs, the two bounds differ only by some absolute constant.
As discussed in the introduction, while these bounds are interesting in their own rights, the upper bound will play a key role in the proofs of our main results. 

\medskip

Recall that $p_t(x,y)$ is the time $t$ transition probability from $x$ to $y$ for $(V,\r)$.
\begin{theorem}\label{t:upperbound}
	There exists an absolute constant $C_0$ such that for every Markov chain $(V,\r)$, 
	\begin{equation}
		\label{e:cheapbounds}
		\frac{m_t}{C_0t} \leq P_t(x)\leq \frac{C_0M_t}{t},\qquad \forall x \in V,\; t>0,
	\end{equation}
	where
	\begin{equation}\label{def_mt}
		M_t=\sup_{x\in V} \int_0^t  p_s(x,x)\mathrm{d}s, \quad \quad m_t=\inf_{x\in V} \int_0^t  p_s(x,x)\mathrm{d}s.
	\end{equation}
\end{theorem}
\begin{corollary}
	\label{cor:transitivecheapbounds}
	There exists an absolute constant $C_0$ such that for every transitive Markov chain,
	\begin{equation}
		\label{e:trancheapbounds}
		\frac{1}{C_0} \frac{\int_0^t p_s(o,o)\mathrm{d}s}{t} \leq P_t\leq C_0\frac{\int_0^t p_s(o,o)\mathrm{d}s}{t}, \qquad \forall t>0.
	\end{equation}
\end{corollary}
\begin{remark}
	\label{r:alternative interpretation}
	We comment that the ratio $\frac{\int_0^t p_s(o,o)\mathrm{d}s}{t}$ has two neat probabilistic interpretations:
	\begin{itemize}
		\item It is the probability that a random walk is at the origin at time $T \sim \mathrm{Unif}[0,t]$.
		\item
		For a transitive chain, it can be shown that up to a universal constant (independent of the chain and $t$) its inverse is proportional to the expected number of vertices visited by time $t$. 
	\end{itemize}
\end{remark}
\begin{proof}[Proof of the second item in Remark \ref{r:alternative interpretation}]
We assume that $\sum_y r_{x,y}=1$.   Using a standard coupling between the continuous-time Markov chain and the corresponding discrete-time Markov chain, it is easy to show that $\sum_{i=0}^{t}P^i(o,o) \asymp \int_0^t p_s(o,o)\mathrm{d}s  $, where $P^t$ is the transition probability for the discrete time Markov chain; and that the expected number of vertices visited by time $t$ for the discrete- and continuous-time chains are also comparable up to a universal constant. We omit the details.

Hence we may work in discrete time.
Let $J_i$ be the indicator that the vertex visited at time $i$ is visited at least $16\sum_{s=0}^{t}P^s(o,o) $ times during $\{i,i+1,\ldots,t \}$. The probability of this event is at most $1/16$ by Markov's inequality. Hence with probability at least $3/4$ there are at most $(t+1)/4$ times such that the event $J_i$ occurs. On this event the number of vertices visited by time $t$ is at least  $\frac{t+1}{64 \sum_{i=0}^{t}P^i(o,o) }$ deterministically.

We now prove the upper bound on $\E |R_t|$, where $R_{t}$ is the set of states visited by  the chain by time $t$. Let $N_x(s)$ be the amount of time spent at state $x$ by time $s$. Then
\[\P(x \in R_t)=\mathbb{P}(|N_{x}(t)| \ge 1) \le \frac{\E |N_{x}(2t)|}{\E(|N_{x}(2t)| \mid  |N_{x}(t)| \ge 1)} \le \frac{\E |N_{x}(2t)|}{\sum_{s=0}^{t}P^s(o,o)}.  \]
Summing over $x$ yields $\E |R_t| \le 2t/\sum_{s=0}^{t}P^s(o,o)$ as desired.
\end{proof}

\medskip

We note that the assertion of Corollary \ref{cor:transitivecheapbounds} is valid for many other graphs. In some cases, the constant $C_0$ has to depend on the graph or on the family of considered graphs. For instance, with some effort one can prove this is the case for a roughly transitive graph $\G=(V,E)$ (meaning that for some $K,K'>0$, for all $x,y \in V$, the rooted graphs $(V,E,x)$ and $(V,E,y)$ are $(K,K')$-quasi isometric \footnote{Two rooted graphs $\G=(V,E,o)$ and $\G'=(V',E',o')$ are said to be $(K,K')$- quasi isometric if there exists a map $\phi$ from $V$ to $V'$ s.t. $\phi(o)=o'$ and $K^{-1}\mbox{dist}(x,y)-K'\leq \mbox{dist}(\phi(x),\phi(y))\leq K \mbox{dist}(x,y)+K'$ for all $x,y\in V$.}) of bounded degree. 

\begin{proof}[Proof of the upper bound in \eqref{e:cheapbounds}]
We define a modified Poisson model in the following way. We fix a time $T>0$. Recall the definition of $M_t$ in Equation \eqref{def_mt}. Define $\lambda=\log (T/M_T)+2$. At the beginning, we put $\mathrm{Poisson}(\lambda)$ particles independently on each site. We assign to each particle $x$ a uniform random variable $U_x\sim \text{Unif} [0,1]$ independently. We partition the particles into disjoint sets $A_1, A_2, \ldots, A_i, \ldots$ by defining $A_i=\{x:U_x\in (2^{-i},2^{-i+1}]\}$. Define $m\in \ZZ_{\geq 0}$ to be such that $2^{m-1}<\lambda/4 \leq 2^m$. Define $t_0=t_1=\cdots=t_m=0$. For $i\in \ZZ_+$, define $t_{i+m}=2^{i+2}M_T$ and $S_i=\sum_{j=0}^it_j$. For the convenience of notations, we set $S_0=0$. For $i\geq m+1$, we define Round $i$ to be the time period $(S_{i-1},S_i]$. And we let Round 1 to Round $m$ all happen at time $0$. In the modified Poisson model, each particle performs an independent SRW as in the original model but we only allow a specific type of coalescence in each round. Namely, in Round $i$, for particles $x\in A_{j}$ and $y\in A_{i}$ with $j<i$, if they meet, the particle $x$ is removed. In other words the particle $x$ `coalesces into' $y$. No other type of coalescence is allowed in Round $i$. 
	
As above, we use $P_t(x)$ to denote the density of surviving particles at site $x$ at time $t$, i.e., $\P(x \mbox{ is occupied at time }t)$ in the CRW model. And we define $\hat P_t(x)$ to be the expected number of surviving particles at site $x$ at time $t$ for the modified Poisson model. We claim that
\begin{align*}
		P_t(x)\leq \hat P_t(x)+e^{-\lambda}.
\end{align*}
To prove it, we define a second modified Poisson model to be the same as the modified Poisson model except we allow all coalescences at any time. Then the second modified Poisson model is the same as putting $\mathrm{Ber}(1-e^{-\lambda})$ particles per site independently at the beginning ($t=0$) and run CRW. Let $R_t$ denote the set of surviving particles at time $t$ in the modified Poisson model and let $R_t'$ be the second modified Poisson model. Then we can couple the two model so that $R_t'\subset R_t$ for any $t$ almost surely. See Proposition 3.2 in \cite{oliveira2012coalescence} for a similar treatment. This readily implies the inequality because for the original CRW model. Indeed, we can color each particle red independently with probability $e^{-\lambda}$. If we do not allow red particles to coalesce, then the density of particles is the same as that for the second modified Poisson model plus $e^{-\lambda}$, which is bounded by that for the modified Poisson model plus $e^{-\lambda}$.
	
Define $\hat P_t(v,j)$ to be the expected number of particles from $A_j$ at site $v$ at time $t$. Then for any $i>0$, we have 
	\begin{align}
		\hat P_{S_{i+m}}(v)&= \E(\text{number of particles at }v \text{ from } \cup_{j=i+m}^\infty A_j \text{ at time } S_{i+m} )+\sum_{j=1}^{i+m-1}\hat P_{S_{i+m}}(v,j)\\
		&=\lambda 2^{-(i+m-1)}+\sum_{j=1}^{i+m-1}\hat P_{S_{i+m}}(v,j),\label{eq:pcontrol}
	\end{align}
	since during the time period $[0,S_{i+m}]$, particles from $\cup_{j=i+m}^\infty A_j$ are not removed, and at time $S_{i+m}$ the number of particles at $v$ from $\cup_{j=i+m}^\infty A_j$ follows Poisson$(\lambda 2^{-(i+m-1)})$.
	
	For $j\leq m$, with $m$ defined as above, we consider the following lemma.
	\begin{lemma}
		For any $v\in V$, the number of particles at site $v$ from $A_j$ that survives at time $0$ has the following distribution: with probability $\exp(-\lambda2^{-j}+\lambda2^{-m})$ it follows Poisson$(\lambda 2^{-j})$ and with probability $1-\exp(-\lambda2^{-j}+\lambda2^{-m})$ it is $0$.
	\end{lemma}
	\begin{proof}
		The event that $x\in A_j$ survives at time $0$ is equivalent to that no particle from $\cup_{\ell=j+1}^m A_{\ell}$ starts at the same position as $x$. By a straightforward computation, we have the lemma.
	\end{proof}
	
	\begin{lemma}\label{ldierate}
		For $i>0$, let $\mathcal{F}_i$ be the $\sigma$-algebra of the trajectories of the particles in $\cup_{j=1}^{i} A_j$ up to time $S_i$. For $\ell>m$ such that $S_\ell\leq T$ and $j<\ell$, for any particle $x$ from $A_j$ that survives to time $S_{\ell-1}$, we have that almost surely,
		\begin{align*}
			\P(x \text{ survives to time } S_{\ell}|\mathcal{F}_{\ell-1}) \leq \frac{1}{4}.
		\end{align*}
	\end{lemma}
	The lemma will be proven later. With this lemma, we have for any $i>0$ with $S_{i+m}\leq T$,
	\begin{align*}
		\sum_{j=1}^{m}\hat P_{S_{i+m}}(v,j)
		\leq & \left(\frac{1}{4}\right)^{i} \sum_{j=1}^{m} \lambda 2^{-j} \exp \left(-\lambda 2^{-j}+\lambda 2^{-m}\right)\\
		= & \left(\frac{1}{4}\right)^{i} \lambda\sum_{j=1}^{m} \exp(-\log(2)j -\lambda 2^{-j}+\lambda 2^{-m})\\
		(\text{using } \lambda> 2^{m+1}) \leq & \left(\frac{1}{4}\right)^{i} \lambda\sum_{j=1}^{m} \exp(-\log(2)j -2^{m-j+1}+2)\\
		\leq & \left(\frac{1}{4}\right)^{i} \lambda\sum_{j=1}^{m} \exp(\log(2)(j-2m))
		\leq 2\lambda 2^{-m}\left(\frac{1}{4}\right)^{i},
	\end{align*}
	as well as
	\begin{align*}
		\sum_{j=m+1}^{i+m-1}\hat P_{S_{i+m}}(v,j)\leq \sum_{j=m+1}^{i+m-1}\left(\frac{1}{4}\right)^{i+m-j}\lambda2^{-j}\leq \lambda 2^{-m}\left(\frac{1}{2}\right)^i.
	\end{align*}
	Combine the two estimates with \eqref{eq:pcontrol}, we have that
	\begin{align*}
		\hat P_{S_{i+m}}(v)\leq 5 \lambda 2^{-m} \left(\frac{1}{2}\right)^i\leq 20 \left(\frac{1}{2}\right)^i.
	\end{align*}
	Note that the same bound holds for $i=0$ trivially. Now, take $i=\lceil \log_2(\frac{T}{M_T}+8) \rceil -4$. Then $S_{i+m}<T\leq S_{i+m+1} $. Therefore,
	\begin{align*}
		\hat P_T(v)\leq \hat P_{S_{i+m}}(v) \leq 20 \left(\frac{1}{2}\right)^{i} \leq \frac{160M_T}{T}.
	\end{align*}
	So
	\begin{align*}
		P_T(v)\leq \frac{160M_T}{T}+e^{-\lambda}\leq \frac{CM_T}{T}.
	\end{align*}
	This implies the upper bound (since $v$ is any vertex in $V$ and $T$ is any positive real number). Now we prove Lemma \ref{ldierate}.
	
	Consider a particle $x\in A_j$. For $\ell>m$ such that $S_\ell\leq T$ and $j<\ell$, we want to study the probability that $x$ survives during the time interval $[S_{\ell-1}, S_{\ell}]$. We let $x$ continue to perform a random walk with rate $\r$ independently after colliding with particles in $A_{\ell}$ if this happens and define $\N_{\ell}(x)$ to be the (Lebesgue measure of the) time units that particles from $A_\ell$ spent together with $x$ during $[S_{\ell-1}, S_\ell]$ (to put it differently, we pretend that the particle $x$ picks its trajectory in advance, and define $\N_{\ell}(x)$ w.r.t.\ this trajectory, without terminating it when $x$ should be removed). As $x$ plays no further rule in the model after a collision, the event $\N_{\ell}(x)>0$ is well defined in the probability space we work with. We have the following
	\begin{align*}
		\mathbb{P}(\N_\ell(x)>0)&=\frac{\E \N_{\ell}(x)}{\E(\N_\ell(x)|\N_{\ell}(x)>0)}\\
		&= \frac{(S_{\ell}-S_{\ell-1})\lambda2^{-\ell}}{\E(\N_\ell(x)|\N_\ell(x)>0)}.
	\end{align*}
	Here the last equality follows from the observation that on each site, the number of particles from $A_\ell$ follows Poisson$(\lambda 2^{-\ell})$ distribution independently. See Fact 2.1 in \cite{MR4108126} for more explanations. Note that the denominator
	\begin{align*}
		\E(\N_\ell(x)|\N_\ell(x)>0)= \int_{S_{\ell-1}}^{S_{\ell}}\E(N_\ell(x)|\text{ the first collision is at time }s)\f(s)\mathrm{d} s,
	\end{align*}
	where $\f(s)$ is the probability density of the first collision time during $[S_{\ell-1},S_\ell]$ given such a collision happens.
	We use the following easy fact to control the expectation in the r.h.s..
	\begin{lemma}\label{lstochas}
		Let $X$ be a random variable whose distribution is given by $\mathrm{Poisson}(\lambda)$ conditioned on being positive, then $X$ is stochastically dominated by $1+\mathrm{Poisson}(\lambda)$.
	\end{lemma} 
	\begin{proof}
	To see this, consider the number of points in $[0, 1]$ for a rate $\lambda$ Poisson process. Given that the first point is at some $x\in [0,1]$, the number of additional points has a Poisson$(\lambda(1-x))$ distribution, and is stochastically dominated by Poisson$(\lambda)$. See Section 3 of \cite{MR4108126} for a similar treatment.
	\end{proof}
	Fix the path that $x$ takes during $[S_{\ell-1},s]$ (for some $S_{\ell-1}\leq s\leq S_\ell$). Let $W_{v,\ell}(s)$ be the set of particles from $A_\ell$ at $v$ at time $s$. Consider the conditional law of $(|W_{v,\ell}|)_{v\in V}$ given that no particle from $A_\ell$ visits the space-time coordinates $\{(x(t),t): t\in [S_{\ell-1},s]\}$. Here $x(\cdot)$ is the trajectory of particle $x$. Then $|W_{v,\ell}|$ for each $v\in V$  is Poisson distributed with parameter less than or equal to $2^{-\ell}\lambda$, and all the $|W_{v,\ell}|$ are independent. This can be obtained by a Poisson thinning argument and Fact 2.1 in \cite{MR4108126}. Now, by Lemma \ref{lstochas}, given that the first collision occurs at time $s$, we have that the joint law of $(|W_{v,\ell}(s)|)_{v\in V}$ is stochastically dominated by $(Y_v)_{v\in V}$, where each $Y_v$ is independent with $\mathrm{Poisson}(\lambda2^{-\ell})+\1[v=x(s)]$ distribution. (To see this, consider the following two cases. The first is that $x$ jumps to $x(s)$ right on time $s$, then we apply Lemma $\ref{lstochas}$. The second case is that $x$ comes to $x(s)$ before time $s$ and stays there until at least $s$, then the number of particle from $A_\ell$ at $x(s)$ is $1$, which is also stochastically dominated by $Y$.) Define $\f(s,y)$ to be the joint probability density of the first collision time and place, given that the first collision happens during $[S_{\ell-1},S_{\ell}]$. Then we have
	\begin{align*}
		&\E(\N_{\ell}(x)|\N_{\ell}(x)>0)\\
		=& \sum_y\int_{S_{\ell-1}}^{S_{\ell}}\E(N_{\ell}(x)|\text{ the first collision is at time }s \text{ and at }y)\f(s,y)\mathrm{d} s\\
		\leq & \sum_y\int_{S_{\ell-1}}^{S_{\ell}} \left(t_\ell\lambda 2^{-\ell}+\sum_z\int_{0}^{t_{\ell}-(s-S_{\ell-1})} p_t(y,z)p_t(y,z)\mathrm{d} t \right)\f(s,y)\mathrm{d} s\\
		=& \sum_y\int_{S_{\ell-1}}^{S_{\ell}} \left(t_\ell\lambda 2^{-\ell}+\int_{0}^{t_{\ell}-(s-S_{\ell-1})} p_{2t}(y,y)\mathrm{d} t \right)\f(s,y)\mathrm{d} s\\
		\leq& \sum_y\int_{S_{\ell-1}}^{S_{\ell}} \left(t_\ell\lambda 2^{-\ell}+\sup_{z\in V}\int_{0}^{t_\ell} p_{2t}(z,z)\mathrm{d} t \right)\f(s,y)\mathrm{d} s\\
		= & t_\ell\lambda 2^{-\ell}+\frac{M_{2t_\ell}}{2},
	\end{align*}
	where we used the fact that $p_t(y,z)=p_t(z,y)$ to get the fourth line from the third line. Then 
	\begin{align*}
		\mathbb{P}(\N_\ell(x)>0)
		&= \frac{\E \N_\ell(x)}{\E(\N_\ell(x)|\N_\ell(x)>0)}\\
		&\geq \frac{ t_\ell\lambda2^{-\ell}}{t_\ell\lambda 2^{-\ell}+M_{2t_\ell}/2}\\
		\text{(using }\lambda \ge 2^{m+1}, t_{\ell}\leq T \text{ and }t_{\ell}=2^{\ell+2-m}M_T \text{)}  &\geq\frac{3}{4}.
	\end{align*}
	This completes the proof of Lemma \ref{ldierate} and hence the proof of the upper bound in \eqref{e:cheapbounds}. 
\end{proof}
\begin{proof}[Proof of the lower bound in \eqref{e:cheapbounds}]
	We consider a CRW model, where at the beginning ($t=0$), at each site independently we have one particle with probability $\frac{m_t}{8t}$. For each $t>0$ and site $x$,
	\begin{align*}
		P_t(x)=&\mathbb{P}(x \text{ is occupied at time } t)\\
		\geq &\mathbb{P}(x \text{ is occupied at time } t \text{ by a particle that hasn't collided with anyone})\\
		=&\mathbb{P}(\text{the particle from } x \text{ at time } 0 \text{ survives without colliding by time } t),
	\end{align*}
	where the last equality comes from the reversibility of $(V,\r)$. Then the above equals
	\begin{align*}
		\frac{m_t}{8t}\mathbb{P}(\text{the particle from } x \text{ has no collision by time } t|\text{there is a particle at $x$ at time } 0).
	\end{align*}
	To bound this conditional probability from below, we let the particles not coalesce with one another upon collisions, i.e., they continue to perform random walks independently. By a straightforward coupling argument, this cannot increase the probability that the particle from $x$ has not collided with another particle by time $t$. Let $\tilde{N}_t$ be the time units that other particles spend together with the particle from $x$ during $[0,t]$, given that there is a particle at $x$ at time $0$. Then we just need to lower bound $\P(\tilde{N}_t=0)$.
	We have
	\begin{align*}
		\mathbb{P}(\tilde{N}_t>0)\leq \frac{\E \tilde{N}_{2t}}{\E(\tilde{N}_{2t}|\tilde{N}_t>0)}.
	\end{align*}
	The numerator satisfies
	\begin{align*}
		\E \tilde{N}_{2t}\leq 
		\frac{m_t}{8t} 2t\leq \frac{m_{2t}}{4}.
	\end{align*}
	The denominator satisfies
	\begin{align*}
		\E(\tilde{N}_{2t}|\tilde{N}_t>0)\geq\int_0^{2t-s}p_{2u}(y,y) \mathrm{d}u \geq \int_0^t p_{2u}(y,y)\mathrm{d} u \geq \frac{m_{2t}}{2},
	\end{align*}
	where we assume the first collision is at time $s\in[0,t]$ at vertex $y$. This implies that
	\begin{align*}
		\mathbb{P}(\tilde{N}_t>0)\leq \frac{1}{2}.
	\end{align*}
	Thus we have
	\begin{align*}
		P_t(x)=\mathbb{P}(x \text{ is occupied at time } t)\geq \frac{m_t}{16t},
	\end{align*}
	which gives us the lower bound.
\end{proof}

\section{A general framework: moments, reversibility, and branching structures}
\label{s:framework}
In this section we set up a framework to analyze the decay of the density of particles.

\subsection{The decay of density and moments of $N_t$}  \label{ss:k-reverse}

Recall in Lemma \ref{P_t=E(N_tinverse)}
we established that $P_t=\E N_t^{-1}$, where $N_t$ is the number of particles that have coalesced with the particle starting from a uniform location $\U$ up to time $t$. We can rewrite this formula as
$$P_t=\E N_t^{-1}=(\E N_t)^{-1} \E\left(\left(\frac{N_t}{\E(N_t)}\right)^{-1}\right).$$

If we could determine 
each term in the above product
up to smaller order terms then we would
have a good control of $P_t$. 
Our general strategy is similar to that in \cite{bramson1980asymptotics}. A crucial difficulty we encounter is that we do not have Sawyer's Theorem (which is a crucial component of the analysis in \cite{bramson1980asymptotics}) at our disposal, which for $\ZZ^d$ approximates all moments of $N_t$. We will
estimate the $k$-th moment $\E N_t^k$ for any $k$ and use the method of moments to show that the distribution of $N_t/\E N_t$ is close to the Gamma$(2,2)$ distribution. We will need to take care of the behavior of $N_t/\E N_t$ near $0$ when computing the limit of $\E ((N_t/\E N_t)^{-1})$. We will show that under some conditions, $\E\left((N_t/\E N_t)^{-1}\right)$ will be close to 2.

For the time being, we concentrate on controlling $\E(N_t^k)$. 
Let the state space $V$ be finite, and be denoted as $V=[n]:=\{1,2,\ldots,n \}$.
It follows from the definition of $N_t$ and the graphical representation of the model, which is implicitly used in the second equality that
\begin{equation}\label{moment}
	\begin{split}
		\E(N_t^k)
		&=\E\left(\sum_{x \in [n]} 
		\1[\mbox{the particles starting at }
		x \mbox{ and }\U \mbox{ have coalesced by time } t]   \right)^k\\
		&=\frac{1}{n}\sum_{x_1, \ldots, x_{k+1} \in V}
		\E\left(\1[X_i(0)=x_i, \forall 1\leq i\leq k+1]\1[\C(X_1,\ldots,X_{k+1})\leq t]\right)\\ &=n^{k}\P_{\pi^{\otimes k+1}}(\C(X_1, \ldots, X_{k+1})\leq t),
	\end{split}
\end{equation}
where $\pi^{\otimes k+1}$ means that initially the $k+1$ particles are independently uniformly  distributed.
Observe that to estimate the r.h.s.\ for small values of $t$ it does not suffice to estimate $\P_{\pi^{\otimes k+1}}(\C(X_1, \ldots, X_{k+1})\leq t)$ up to an $o(1)$ additive error. This is the reason we cannot rely on the elegant techniques developed by Oliveira in \cite{oliveira2013mean}, which allow one to estimate this probability up to some $o(1)$ additive error, when $t_{\mix} \ll \M$. The $o(1)$ error is much larger than the probability in question, unless $\M/t$ is not too large.

In particular, for $k=1$ we have
\begin{equation}
	\E(N_t)=n\P_{\pi^{\otimes 2}}(\C(X_1,X_2)\leq t ).
\end{equation}

To control this probability, we use the exponential approximation result, i.e., equation \eqref{probestimate}, and
consider the product chain of two independent random walks. 
Let $A$  be the diagonal set $\{(x,x):x\in V\}$.
The relaxation time of the product chain is the same as the original chain. The coalescence time for a pair of independent walks has the same distribution as the hitting time of $A$ of the product chain. 
It follows from equation \eqref{probestimate} that for $t\leq \M$,
\begin{equation}\label{N_t}
	\abs{\E N_t-\frac{nt}{\M }} \leq n\left( \left(\frac{t}{\M}\right)^2+\frac{\rel}{\M} \right).
\end{equation}
Note that the right hand side is $o\left(\frac{nt}{\M}\right)$ when $t \vee \rel \ll \M$, so it is also $o(\E N_t)$.

To obtain estimates on $\E N_t^k$ for $k\geq 2$, or equivalently $\P_{\pi^{\otimes(k+1)}}(C_1(X_1, \ldots, X_{k+1})\leq t )$,  we use reversibility of random walks.
Specifically, if $X(s),0\leq s\leq t$ is a random walk evolving according to rates $\r$ and starting from a uniform initial location, then the reversed path $\gamma(s):=X(t-s),0\leq s\leq t$ is also a random walk path that starts from a uniform initial location and has jump rates $\r$ (because $\r$ is symmetric).
After reversal of time, the collision events become `branching' events, i.e., if particle $i$ jumps to particle $j$  at time $t$, then after reversal of time this corresponds to that particle $j$ gives birth to particle $i$. We will define branching structures in the next section to study this.  

\subsection{Reversal of time and branching structures}   \label{ss:branchstruc}

In this section, we define branching structures to study the probability density of the coalescing time of $k+1$ particles. We set
$$
I_k=[i_0, i_1,\ldots, i_k],$$
where $i_0,\ldots, i_k$ is a sequence of integers satisfying $i_0=0$, and $0 \le i_{\ell}\le \ell-1$, $\forall \ell \in \{1, \ldots, k\}$. 
We use $\Phi_k$ to denote  the collection of all such $I_k$. 
We write 
$\R^k_{<, t}:= \{(t_1,\ldots, t_k):0<t_1<\ldots<t_k<t\}$,
and let $\mathrm{d} \t$ be the $k$ dimensional Lebesgue measure on it.
Take $\t=(t_1,\ldots, t_k)\in \R^k_{<, t}$.
For each $0\leq \ell \leq k$, we define $\gamma_{\ell}$ to be a random walk path from time $t_{\ell}$ to $t$. We do this inductively. First, $\gamma_0:[0, t]\to V$ is defined to be a continuous-time Markov chain evolving according to $(V,\r)$, whose starting location $\gamma_0(0)$ is uniformly distributed.
Then for $\ell \geq 1$, let $A_{\ell}:=\gamma_{i_{\ell}}(t_{\ell})$.
Conditioned on $A_{\ell}$, we let
$B_{\ell}$ be a random vertex sampled from the distribution $\nu_{A_\ell}$. Here we recall the exit measure $\nu_v$
defined by
$
\nu_v(u):=\frac{r_{v,u}}{r(v)},
$
where $r(v):=\sum_x r_{v,x}$ is the total transition rate at $v$. We define $\gamma_{\ell}:[t_{\ell}, t]\to V$  to be a random walk starting at $B_{\ell}$. 
Using this branching structure we can define
\begin{equation}\label{cond_form}
	h(\t, t, I_k):= \E
	\left(\prod_{i=1}^{k}  r(A_i) \1[\forall \, 0\leq \ell_1< \ell_2\leq k,\; \forall \, t'\in[ t_{\ell_2},t], \; \gamma_{\ell_1}(t')\neq \gamma_{\ell_2}(t')]\right),
\end{equation}
which, intuitively, describes the probability that no coalescence happens for all branches, weighted by the transition rates. 

We may think of the above as follows. At time $t_j$ a particle is born at a random position $\sim \nu_{A_j}$ (a location adjacent to the current position of the particle $i_j$). This particle can be thought of as an offspring of the particle $i_j$.
The offspring particle is then required to avoid all other particles by time $t$. Upon reversal of time, the times at which particles are created correspond to collision times.
For simplicity of notations, we shall usually omit the dependence on $t$, and just write $h(\t, I_k)=h(\t, t, I_k)$.

We have the following lemma relating $\E(N_t^k)$ and the branching structure.
\begin{lemma}   \label{l:reverdiffinit}
	For $k\geq 1$, let $X_0,\ldots, X_k$ be coalescing random walks. Then we have
	\begin{align*}
	n^{k}\P_{\pi^{\otimes (k+1)}}(\C(X_0, \ldots, X_{k})\leq t ,\forall 0  \leq i<j\leq k, & \, X_i(0) \neq X_j(0))\\ &=(k+1)!\sum_{I_k \in \Phi_k} \int_{\R^k_{<, t}} h(\t, I_k)\mathrm{d}\t,
	\end{align*}
	where $\P_{\pi^{\otimes (k+1)}}$ means that the initial locations $X_0(0),\ldots,X_k(0)$  are random, and are independently uniformly distributed.
\end{lemma}
\begin{proof}
	We say a   $X$ `jumps onto' another particle $Y$ to create a collision if for some $t>0$, $X_{t-}\neq X_t, Y_{t-}=Y_t, X_t=Y_t$. We set the following convention that 
	after the coalescence the new particle keeps the label of $Y$. 
	Now let us
	consider all possible ways  that particles  jump onto other particles to create the  $k$ collisions. For $1\leq m\leq k$, we assume that $X_{j_m}$ jumps to $X_{\ell_m}$ for the $m$-th collision. Then clearly $j_1,\ldots, j_k, \ell_1,\ldots, \ell_k$ have to satisfy the conditions that 
	all $j_m$ are mutually distinct numbers chosen from $\{0,\ldots, k\}$ and each $\ell_m$ belongs to the set $\{0,\ldots, k\} \setminus \{j_1,\ldots, j_m\}$. Conversely, given any sequence of numbers $(j_1,\ldots, j_k, \ell_1,\ldots, \ell_k)$ satisfying the above conditions, we have that for any $1\leq m\leq k$, with positive probability, the particle with label $j_m$ jumps onto the particle with label $\ell_m$ for the $m$-th collision.
	We use $\hat{\Phi}_k$ to denote the collection of all such vectors, and we call each $\hat{I}_k \in \hat{\Phi}_k$ a `collision pattern'. 
	We use $L(a,b,s)$ to denote the event that the particle with label $a$ jumps onto the particle with label $b$ at time $s$. Then for any $(s_1,\ldots, s_k)\in \R^k_{<, t}$ and $\hat{I}_k = (j_1,\ldots, j_k, \ell_1,\ldots, \ell_k)$, we set
	\[
	\hat{L}((s_1,\ldots, s_k), \hat{I}_k)=\cap_{m=1}^k L(j_m,\ell_m,s_m).
	\]
	Using $f((s_1,\ldots, s_k),  \hat{I}_k)$ to denote the probability density of the event $L((s_1,\ldots, s_k), \hat{I}_k)$, we have 
	\begin{align*}
	\P_{\pi^{\otimes (k+1)}}(\C(X_0, \ldots, X_{k})\leq t ,\forall \, 0  \leq i<j\leq k, & \, X_i(0) \neq X_j(0))\\
	&=\int_{\R^k_{<,t}}\sum_{\hat{I}_k\in \hat{\Phi}_k}
	f((s_1,\ldots, s_k),  \hat{I}_k)\mathrm{d}\s.
	\end{align*}
	By the invariance of $f((s_1,\ldots, s_k),  \hat{I}_k)$ w.r.t. the permutation of indices of the particles, we only need to deal with $\hat{I}_k$ where $j_m=k+1-m$, for any $1\leq m\leq k$. For such $\hat{I}_k$ we can associate an $I_k$ with it
	by setting $i_m=\ell_{k+1-m}$ for each $1\leq m \leq k$, and $i_0=0$. One can check that $(i_0,i_1,\ldots, i_k)$ is indeed an element of the set $\Phi_k$. 
	Setting $t_m=t-s_{k+1-m}$, we have $L(j_m,\ell_m,s_m)=L(k+1-m,i_{k+1-m}, t-t_{k+1-m})$, so
	\[
	\hat{L}((s_1,\ldots, s_k), \hat{I}_k)=\cap_{m=1}^k L(m,i_{m}, t-t_{m}), 
	\]
	which we also denote by $L((t_1,\cdots, t_k),I_k)$. We also use $f((t_1,\cdots, t_k),I_k)$ to denote the probability density of the event $L((t_1,\cdots, t_k),I_k)$.
	Thus
	\begin{multline*}
		\P_{\pi^{\otimes (k+1)}}(\C(X_0, \ldots, X_{k})\leq t ,\forall\, 0  \leq i<j\leq k, X_i(0) \neq X_j(0))\\=(k+1)!\int_{\R^k_{<,t}}\sum_{I_k \in \Phi_k}f((t_1,\cdots, t_k),I_k)\mathrm{d} \s.
	\end{multline*}
	Thus to prove Lemma \ref{l:reverdiffinit} it suffices to show that
	\[
	f((t_1,\cdots, t_k),I_k)=n^{-k}h(\t,I_k).
	\]
	Define the event
	\[
	F(m,i_m,B_m,A_m,t-t_m)=\{
	X_m(t-t_m)=B_m, X_{i_m}(t-t_m)=A_m
	\},
	\]
	and set
	\begin{equation}
		\begin{split}
			F_k=F_k(\{A_m,B_m, i_m, t-t_m\}_{m=1}^k)
			=\cap_{m=1}^kF(m,i_m,B_m,A_m,t-t_m).
		\end{split}
	\end{equation}
	We write $f_k(\{A_m,B_m\}_{m=1}^k)$ as a shorthand for the probability density of $F_k$. Then we can write
	$f((t_1,\cdots, t_k),I_k)$ as the sum
	over all possible collision locations:
	$$
	\sum_{\{A_m,B_m\}_{m=1}^k} f_k\left(\{A_m,B_m\}_{m=1}^k\right).
	$$
	Since the jump rate from any $B_m$ to $A_m$ is $r_{B_{m},{A_{m}}}$, and
	$\P(X_m(t-t_m)=B_m) = 1/n$ (as $X_m(t-t_m)$ is uniformly random), by reversal of time, we see the above sum is equal to
	\begin{equation}\label{density}
		n^{-k}\sum_{A_m, B_m, \forall 1\leq m\leq k}\left(\prod_{m=1}^kr_{B_{m},A_{m}}\right)\P(
		\hat{\gamma}_{i_m}(t_m)=A_m, \forall 1\leq m\leq k, 
		\textrm{all }\hat{\gamma}_m  \textrm{'s don't collide}),
	\end{equation}
	where
	$\hat{\gamma}_m:[t_m, t]\to V$ is the reversal of
	$X_m(s)$ for $s\in [0,t-t_m]$ conditioned on
	$X_m(t-t_m)=B_m$. 
	Note that $\hat{\gamma}_m$ also performs a continuous-time random walk by reversibility, and $X_m$ performs an independent random walk until a collision happens. We have the indicator $\1[\textrm{all }\hat{\gamma}_m  \textrm{'s don't collide}]$ in our equation because collisions cannot happen outside $\{t-t_m,1\leq m\leq k\}$.
	Moreover, by writing 
	$r_{B_{m},A_{m}}$ as 
	$r(A_{m}) \cdot (r_{B_{m},A_{m}})/r(A_{m})$ we see 
	that \eqref{density} is exactly equal to $n^{-k}h(\t,I_k)$. This completes the proof of Lemma
	\ref{l:reverdiffinit}. \end{proof}

We have the following corollary of Lemma \ref{l:reverdiffinit}, which is about two or more particles starting from the same place.
Let $\Xi$ be any partition of $\{0,1,\ldots, k-1,k\}$, represented as a collection of subsets of  $\{0,1,\ldots, k-1,k\}$.
For $0\leq i<j\leq k$ we write  $i\leftrightarrow j$ if $i$ and $j$ belong to the same set of $\Xi$ and  $i\nleftrightarrow j$ otherwise. 
\begin{corollary}\label{cor:rev2}   
	For $k\geq 1$ and $\abs{\Xi}\geq 2$ we have
	\begin{multline*}
		n^{k}\P_{\pi^{\otimes (k+1)}}(\C(X_0, \ldots, X_{k})\leq t; \forall i\leftrightarrow j,  X_i(0) = X_j(0);\forall i\nleftrightarrow j, X_i(0) \neq X_j(0)
		)\\ =\abs{\Xi}!\sum_{I_{\abs{\Xi}-1} \in \Phi_{\abs{\Xi}-1}} \int_{\R^{\abs{\Xi}-1}_{<, t}} h(\t, I_{\abs{\Xi}-1})\mathrm{d} \t.
	\end{multline*}
\end{corollary}
\begin{proof}
	Using Lemma \ref{l:reverdiffinit} we have
	\[
	\begin{split}
		&n^{k}\P_{\pi^{\otimes (k+1)}}(\C(X_0, \ldots, X_{k})\leq t; \forall i\leftrightarrow j,  X_i(0) = X_j(0);\forall i\nleftrightarrow j, X_i(0) \neq X_j(0)
		)\\
		=&n^{\abs{\Xi}-1} \P_{\pi^{\otimes (\abs{\Xi})}}(\C(X_0, \ldots, X_{\abs{\Xi}-1})\leq t; \forall 0  \leq i<j\leq \abs{\Xi}-1, X_i(0) \neq X_j(0))\\ =&\abs{\Xi}!\sum_{I_{\abs{\Xi}-1} \in \Phi_{\abs{\Xi}-1}} \int_{\R^{\abs{\Xi}-1}_{<, t}} h(\t, I_{\abs{\Xi}-1})\mathrm{d}\t.
	\end{split}
	\]
	Thus the conclusion follows.
\end{proof}
Lemma \ref{l:reverdiffinit} is the special case of Corollary \ref{cor:rev2},
where $i\nleftrightarrow j$ for any $0\le i < j \le k$.
For the case where $i\leftrightarrow j$ for any $0\le i < j \le k$, i.e., $\abs{\Xi}=1$, we have
\[
\begin{split}
	&n^{k}\P_{\pi^{\otimes (k+1)}}(\C(X_0, \ldots, X_{k})\leq t; \forall i\leftrightarrow j,  X_i(0) = X_j(0))\\ 
	=&n^{k}\P_{\pi^{\otimes (k+1)}}(\forall \,\, 0\leq i<j\leq k,  X_i(0) = X_j(0))=1.
\end{split}
\]
By summing over all possible partitions $\Xi$ and using Corollary \ref{cor:rev2}, 
we get the following lemma.
\begin{lemma}\label{densityexpress}
	For $k\geq 1$, we have
	\begin{equation}\label{meanfield}
		\begin{split}
			0&\leq \P_{\pi^{\otimes (k+1)}}(\C(X_0,\ldots, X_{k})\leq t)-
			n^{-k}(k+1)!\sum_{I_k \in \Phi_k} \int_{\R^{k}_{<, t}} h(\t,I_k)\mathrm{d} \t\\
			&\leq \sum_{1\leq \ell\leq k-1} (k+1)^{\ell+1} n^{-k} (\ell+1)! \sum_{I_\ell\in \Phi_\ell}\int_{\R^{\ell}_{<, t}  }  h(\t,I_{\ell})\mathrm{d} \t+n^{-k}.
		\end{split}
	\end{equation}
\end{lemma}
\begin{proof}
Note that we have the decomposition
	\begin{equation}
		\begin{split}
			&\P_{\pi^{\otimes (k+1)}}(\C(X_0,\ldots, X_{k})\leq t)\\
			=&\sum_{\Xi}\P_{\pi^{\otimes (k+1)}}(\C(X_0, \ldots, X_{k})\leq t ,\forall i\leftrightarrow j,  X_i(0) = X_j(0);\forall i\nleftrightarrow j, X_i(0) \neq X_j(0))
		\end{split}
	\end{equation}
	where the sum of $\Xi$ is over all partitions of the set $\{0,1,\ldots, k-1,k\}$.
	Equation \eqref{meanfield} now follows by applying Corollary \ref{cor:rev2}, and using the trivial upper bound that the number of partition $\Xi$ of the set $\{0,1,\ldots, k-1,k\}$   s.t. $\abs{\Xi}=\ell+1$ is at most $(k+1)^{\ell+1}$.
\end{proof}
We record the following lemma which will be used in Section \ref{s:pfrandom}.
We recall the notation for maximal and minimal jump rate:
$$
r_{\max}=\max_{x\in V} r(x) \qquad \text{and} \qquad
r_{\min}=\min_{x\in V} r(x).
$$
Throughout, we denote their ratio by
$$
R:=r_{\max}/r_{\min}.
$$
Recall that $\gamma_{\ell}:[t_{\ell}, t]\to V$ is a random walk starting at $B_{\ell}$.
\begin{lemma}\label{maxgamma_t(x)}
	For any $\ell$, and any $s\in [t_{\ell},t]$, $x\in V$, we have
	$$
	\P(\gamma_{\ell}(s)=x)\leq \frac{R^{\ell}}{n}.
	$$
\end{lemma}
\begin{proof}[Proof of Lemma \ref{maxgamma_t(x)}]
	We use the fact that for any initial measure $\mu_0$ and any symmetric Markov transition kernel, $\sup_x \mu_s(x)$ is not non-increasing in $s$. Indeed, 
	\begin{equation}
		\begin{split}
			\sup_x \mu_s(x)
			&\leq \sup_x \sum_y \mu_0(y)p_s(y,x)\\
			& \leq \sup_x \sum_y (\sup_z \mu_0(z))p_s(x,y) =
			\sup_z \mu_0(z). 
		\end{split}
	\end{equation}
	We now prove Lemma \ref{maxgamma_t(x)} inductively. The case of $\ell=0$ follows from the fact that $\gamma_0(t_0)$ is sampled from the uniform distribution. Suppose that this lemma holds for $\ell\leq m-1$. Then for $\ell=m$ we have
	\begin{equation}
		\begin{split}
			\P(\gamma_m(t_m)=x)&=\sum_y \P(\gamma_{i_m}(t)=y)
			\frac{r_{y,x}}{r(y)}\\
			\leq &\frac{R^{m-1}}{n}\sum_y \frac{r_{y,x}}{r_{\min}}
			= \frac{R^{m-1}}{n}\frac{r(x)}{r_{\min}} \leq
			\frac{R^m}{n}
		\end{split}
	\end{equation}
	This proves the induction step and hence completes the proof of Lemma \ref{maxgamma_t(x)}. 
\end{proof}

\subsection{Intermediate quantities}  \label{ss:roadmap}
We now estimate the integral of $h(\t, I_k)$. We let $\U\sim\pi$ and $\tilde \U\sim \nu_{\U}$. We also let $W_\U$ and $W_{\tilde \U}$ be two random walks starting from $\U$ and $\tilde \U$ respectively.

We define some intermediate quantities. Let
\begin{equation}\label{gt,phi_k}
	g(\t, t, I_k):= \E
	\left(\prod_{\ell=1}^{k}  r(A_\ell) \1[\forall t'\in[ t_{\ell},t_{\ell+1}],\gamma_{\ell}(t')\neq \gamma_{i_{\ell}}(t')]\right).
\end{equation}
It describes the probability that no collision happens between a newly born branch and its parent until the next branching time. We also define
\begin{equation}\label{cond_form4}
	q(\t, t, I_k):= \prod_{\ell=1}^{k}\E
	( r(\U) \1[\forall t'\in[ 0,t_{\ell+1}-t_\ell],W_\U(t')\neq W_{\tilde \U}(t')]).
\end{equation}
Then $q(\t, t, I_k)$ describes the probability that no collision happens between a newly born branch and its parent until the next branching time, assuming that the place of branching is uniformly distributed. 
Similar to $h(\t, I_k)$, for simplicity of notations, we often omit the dependence on $t$, and just write $g(\t, I_k)=g(\t, t, I_k)$ and $q(\t, I_k)=q(\t, t, I_k)$.

Our general approach for estimating the integral of $h(\t, I_k)$ is to control each of the following.
\begin{enumerate}
	\item $n^{-k}(k+1)!\sum_{I_k \in \Phi_k}\int_{\R^{k}_{<, t}} \left(h(\t,I_k) - g(\t,I_k)\right)\mathrm{d} \t$,
	\item $n^{-k}(k+1)!\sum_{I_k \in \Phi_k}\int_{\R^{k}_{<, t}}
	\left(g(\t,I_k) - q(\t,I_k)\right)\mathrm{d} \t$,
	\item $\sum_{I_k \in \Phi_k}\int_{\R^{k}_{<, t}}q(\t,I_k)\mathrm{d} \t - \left(\frac{nt}{2\M}\right)^k$.
\end{enumerate}
\begin{remark}
The first error term describes the change in the coalescence probability 
	if we ignore collisions that are \emph{not} between a new branch and its parent.  We will  control this error  in Lemma \ref{l:approxhbyg} (for general Markov chains)
	and Lemma \ref{l:goodset:tran} (for transitive Markov chains).
	Going back to the original  CRW system (i.e., undoing the reverse of time), showing that this error term is small is equivalent to saying that conditioned on that two walkers collide in a certain time interval with length $\ll \M$, with high probability we don't see any other collisions in this interval.
\end{remark}

\begin{remark}
	As will be seen in the proof of Lemma \ref{l:ntk-est}, the contribution to the integral of $h(\t,I_k)$ that comes from $\t$  outside the $\Delta$-good set (see \eqref{d:smallt} to \eqref{G(t)} for the definition) is negligible under some conditions.  
	This implies that the typical spacing between successive collisions is much larger than  $\rel$, even if we assume that $\C(X_0,\ldots, X_k)\leq t$.
\end{remark}

\section{Estimates for general Markov chains}  \label{s:gengraph}
In this section we provide estimates for general finite Markov chains.

To control the integral of $h(\t,I_k)-g(\t,I_k)$ and the integral of $q(\t, I_k)$, we consider $\t\in\R^k_{<, t}$ such that the difference
$t_{\ell+1}-t_{\ell}$ is not too small for each $1 \leq\ell \leq k$ (here we set $t_{k+1}=t$). Specifically, let
\begin{equation}\label{d:smallt}
	\Delta_{\min}(\t,t):=\min\{t_{\ell+1}-t_{\ell}:1 \leq \ell \leq k \},\quad \forall \t \in \R^k_{<, t}.
\end{equation}
For simplicity of notations, we also usually write $\Delta_{\min}(\t)=\Delta_{\min}(\t,t)$ since $t$ is mostly fixed.
We define
\begin{equation}\label{defofH}
	H(t):=\sup_{\rel/2 <s<2t}\frac{\max_z \int_0^s  p_{s'}(z,z)\mathrm{d} s'}{\min_z \int_0^s  p_{s'}(z,z)\mathrm{d} s'}.
\end{equation}
Given $\Delta>2$, let
\begin{equation}\label{G(t)}
	G(t,\Delta):=\rel \log(H(t))+\rel \Delta.
\end{equation}
We call $\t$ `$\Delta$-good' if $\Delta_{\min}(\t)>G(t,\Delta)$. We now give an implication of being `$\Delta$-good'.
\begin{lemma}  \label{l:deltagood}
	If $\t\in \R^k_{<, t}$ is $\Delta$-good, we have $\Delta_{\min}(\t) \geq 1/r_{\max}$,
	and
	$$
	\max_{1\leq r\leq k} \exp\left(-\frac{\Delta_{\min}(\t)}{\rel} \right) \frac{\max_z \int_0^{2(t_{r}-t_{r-1})}  p_{s}(z,z)\mathrm{d} s }{\min_z \int_0^{2(t_{r}-t_{r-1})} p_{s}(z,z)\mathrm{d} s}
	\leq e^{-\Delta}.
	$$
\end{lemma}
\begin{proof}
	Let $\t$ be $\Delta$-good. For $\Delta>2$, we have
	$t_{\ell}-t_{\ell-1}>2\rel$ for all $1\leq \ell\leq k$, which implies that
	$$
	\max_{1\leq r\leq k}  \frac{\max_z \int_0^{2(t_{r}-t_{r-1})}  p_{s}(z,z)\mathrm{d} s }{\min_z \int_0^{2(t_{r}-t_{r-1})}  p_{s}(z,z)\mathrm{d} s}
	\leq H(t).
	$$
	It follows that 
	\begin{align*}
	& \max_{1\leq r\leq k}  \exp\left(-\frac{\Delta_{\min}(\t)}{\rel} \right)  \frac{\max_z \int_0^{2(t_{r}-t_{r-1})}  p_{s}(z,z)\mathrm{d} s }{\min_z \int_0^{2(t_{r}-t_{r-1})} p_{s}(z,z)\mathrm{d} s}\\
	&\leq \exp\left(-\frac{(\log H(t)+\Delta)\rel}{\rel}\right)H(t)\leq e^{-\Delta}.
	\end{align*}
	Also, as $
	\rel\geq 1/2r_{\max}$ (see \eqref{relandrmax} in Lemma \ref{walkparamaters}), we have $\Delta_{\min}(\t)\geq 1/r_{\max}$. 
\end{proof}

To control $q(\t, I_k)-g(\t,I_k)$,
for $k\geq 2$ we define the following error parameters:
\begin{equation}\label{e3kt}
	\hat{e}_{k}(\t):=\sup_{I_k\in \Phi_k}\abs{q(\t,I_k)-g(\t,I_k)}, \qquad \text{and}
\end{equation}
\begin{equation}\label{e3kt2}
	e_k(t):=\int_{\R^k_{<,t}}\hat{e}_{k}(\t) \mathrm{d} \t.
\end{equation}

Now we are ready to state our main result of this section. This is a framework to be used in the following sections.
The main thing we establish here is that, given bounds on some error terms, we could estimate the density up to a multiplicative $1+o(1)$ factor.
Recall that
$
R=\frac{r_{\max}}{r_{\min}}=\frac{\max_{x\in V} r(x)}{\min_{x\in V} r(x)}
$.
\begin{prop}  \label{p:gengraph}
	For any $\varepsilon>0$ and $K'>0$, there exist some $K\in \ZZ_+$ and $\delta>0$, such that the following holds.
	Consider an arbitrary Markov chain $(V,\r)$ with $|V|=n$. 
	Take 
	\begin{equation}\label{defDelta}
		\Delta:=K\log\left(\frac{\M r_{\max}}{n}+1\right)+ \log R+ \frac{1}{\delta}.
	\end{equation}
	Define 
	\begin{equation}\label{d:capitalI}
		\begin{split}
			E_1&:=\left(\frac{\M r_{\max}}{n}\right)^{K} \frac{G(t,\Delta)}{t}, \qquad
			E_2:=\left(\frac{\M r_{\max}}{n}\right)^{K} R
			\frac{r_{\max}t}{n},\\
			E_3&:=\sum_{k=2}^K \left(\frac{\M}{tn}\right)^{k}e_k(t), \qquad \qquad
			E_4:=\left(\frac{\M r_{\max}}{n}\right)^{K} \frac{1}{r_{\max}t}.
		\end{split}
	\end{equation}
	If the following two conditions hold
	\begin{align}
		E_1+E_2+E_3+E_4
		&\leq \delta,\label{firstcond}\\
		\sup_{x\in V,\,  0\leq s\leq t}P_s(x)\frac{ns}{\M}&\leq K', \label{secondcond}
	\end{align}
	then we have 
	\[
	\abs{P_t-\frac{2\M}{nt}}\leq \varepsilon \frac{\M}{nt}.
	\]
\end{prop}
\begin{remark}\label{variant}
	Since $G(t,\Delta)$ is increasing in $\Delta$ and hence decreasing in $\delta$, Proposition \ref{p:gengraph}
	still holds if  we let $\Delta$ be
	\begin{equation}\label{newDelta}
		\Delta:=K\log\left(\frac{\M r_{\max}}{n}+1\right)+ \log R+ \frac{1}{\delta'}
	\end{equation}
	for some $\delta'\leq \delta$. 
	We will use this variant in the proof of Theorem \ref{t:config}. 
\end{remark}

\subsection{Estimating $g(\t,I_k)-h(\t,I_k)$}  \label{ss:approxh}

Now we bound the difference $g(\t,I_k)-h(\t,I_k)$ for $\Delta$-good $\t$.
\begin{lemma}\label{l:approxhbyg}
	For any Markov chain $(V,\r)$ with $\abs{V}=n$, and any $k\in \ZZ_+$, $I_k \in \Phi_k$, and $\Delta$-good $\t$, we have
	\begin{equation}
		0\leq g(\t,I_k)-h(\t,I_k) \leq r_{\max}^k \frac{k(k+1)}{2}
		(10R+1)\left(2 e^{-\Delta} + \frac{8tr_{\max}}{n}\right).
	\end{equation}
\end{lemma}
\begin{proof}
	By definition 
	\begin{equation}\label{g-h}
		\begin{split}
			0\leq &g(\t,I_k)-h(\t,I_k)\\
			\leq &r_{\max}^{k}
			\P(\{\forall 1 \leq \ell \leq k \text{ and } t' \in [t_\ell, t_{\ell+1}] \text{, we have }
			\gamma_{\ell}(t')\neq \gamma_{i_\ell}(t')\} \\
			& \cap \{ \exists 0\leq \ell_1 < \ell_2 \leq k \text{ and } t''\in [t_{\ell_2}, t] \text{, such that } \gamma_{\ell_1}(t'')=\gamma_{\ell_2}(t'') \} )\\
			\leq & r_{\max}^{k} \sum_{0\leq \ell_1 < \ell_2 \leq k}
			\E\Big(\1[\{\forall 1 \leq \ell \leq k \text{ and } t' \in [t_\ell, t_{\ell+1}], \text{ we have }
			\gamma_{\ell}(t')\neq \gamma_{i_\ell}(t')\} \\ & \cap \{
			\exists  t''\in [t_{\ell_2}, t], \text{ such that } \gamma_{\ell_1}(t'')=\gamma_{\ell_2}(t'')\}] \vphantom{\frac12}
			\\
			&\times \left(\1[\ell_1= i_{\ell_2}]+\1[\ell_1\neq i_{\ell_2} \text{ and } \forall t''' \in [t_{\ell_2-1},t_{\ell_2}], \text{ we have } \gamma_{\ell_1}(t''')\neq 
			\gamma_{i_{\ell_2}}(t''')]\right) \vphantom{\frac12} \Big).
		\end{split}
	\end{equation}
	
	For the expectation inside the summation in the last expression, depending on whether the collision happens before or after $t_{\ell_2+1}$, we  upper bound it by
	\begin{equation}  \label{e:goodsetpf1}
	\begin{split}
		&\P(\{\gamma_{\ell_1}(t')\neq \gamma_{\ell_2}(t') \text{ for all } t' \in [t_{\ell_2}, t_{\ell_2+1}]\} \cap \{\gamma_{\ell_1}(t'')=\gamma_{\ell_2}(t'') \text{ for some }
		t''\in [t_{\ell_2+1}, t] \})
		\\+&\P(\{\gamma_{\ell_1}(t')\neq \gamma_{i_{\ell_2}}(t') \text{ for all } t' \in [t_{\ell_2-1}, t_{\ell_2}]\} \cap
		\{\gamma_{\ell_1}(t'')=\gamma_{\ell_2}(t'') \text{ for some } t''\in [t_{\ell_2}, t_{\ell_2+1}] \})\\
		:=&P_1+P_2.
		\end{split}
	\end{equation}
	Here we have used the observation that if $\gamma_{\ell_1}$ collides with $\gamma_{\ell_2}$ within the time period
	$[t_{\ell_2},t_{\ell_2+1}]$ then it must be true that $\ell_1\neq i_{\ell_2}$, since otherwise the indicator 
	\begin{align*}
		\1[\{\forall 1 \leq \ell \leq k \text{ and all } t' \in [t_\ell, t_{\ell+1}], &\text{ we have that }
		\gamma_{\ell}(t')\neq \gamma_{i_\ell}(t')\} \\ & \cap \{
		\exists  t''\in [t_{\ell_2}, t] \text{ such that } \gamma_{\ell_1}(t'')=\gamma_{\ell_2}(t'')\}]\end{align*}
	would be $0$. 
	
	For the first term, since $\gamma_{\ell_1}$ and $\gamma_{\ell_2}$ are independent  random walks, by Lemma \ref{meetprob}, we have
	\begin{equation*}
	    \begin{split}
	        	P_1&\leq 2\exp\left(-\frac{t_{\ell_2+1}-t_{\ell_2}}{\rel}\right)
	\frac{\max_z \int_0^{2(t_{\ell_2+1}-t_{\ell_2})}  p_{s}(z,z)\mathrm{d} s }{\min_z \int_0^{2(t_{\ell_2+1}-t_{\ell_2})}  p_s(z,z)\mathrm{d} s  }+\frac{8t}{n} ((t_{\ell_2+1}-t_{\ell_2})^{-1}\vee r_{\max})\\
	&\leq 2 e^{-\Delta} + \frac{8tr_{\max}}{n},
	    \end{split}
	\end{equation*}
	where the second inequality follows from the assumption that $\t$ is $\Delta$-good and Lemma \ref{l:deltagood}.

	For the second term, 
	we modify the proof of Lemma \ref{meetprob} to control it.
	Let
	$N_1$ denote the amount of time $\gamma_{\ell_1}$ and $\gamma_{\ell_2}$ spend together during $[t_{\ell_2}, t_{\ell_2}+(t_{\ell_2+1}-t_{\ell_2-1})]$.
	Then it suffices to upper bound
	$$
	\frac{\E N_1}{\E(N_1 \mid \{\gamma_{\ell_1}(t')\neq \gamma_{i_{\ell_2}}(t'),\forall t' \in [t_{\ell_2-1}, t_{\ell_2}]\} \cap \{
		\gamma_{\ell_1}(t')=\gamma_{\ell_2}(t') \text{ for some } t'\in [t_{\ell_2}, t_{\ell_2+1}] \})}.
	$$
	Note that in $[t_{\ell_2-1},t_{\ell_2}]$, $\gamma_{\ell_1}$ and $\gamma_{i_{\ell_2}}$ perform independent  random walks; in $[t_{\ell_2},t_{\ell_2+1}]$, $\gamma_{\ell_1}$ and $\gamma_{\ell_2}$ perform independent  random walks 
	and $\gamma_{\ell_2}(t_{\ell_2})$ is sampled from $\nu_{\gamma_{i_{\ell_2}}(t_{\ell_2})}$.
	For the denominator, we can lower bound it by
	$$
	\min_z \int_0^{t_{\ell_2}-t_{\ell_2-1}} p_{2s}(z,z) \mathrm{d} s \geq 
	\frac{(t_{\ell_2}-t_{\ell_2-1})\wedge r_{\max}^{-1}}{8},
	$$
	using arguments similar to those in the proof of Lemma \ref{meetprob}.
	Since $\t$ is $\Delta$-good, by Lemma \ref{l:deltagood}, we have $t_{\ell_2}-t_{\ell_2-1} \geq \Delta_{\min}(\t) \geq 1/r_{\max}$, so the denominator is bounded from below by
	$1/(8 r_{\max})$. For the numerator, it is bounded from above by
	$$
	\max_{x,y} \int_{t_{\ell_2}-t_{\ell_2-1}}^{t_{\ell_2+1}+t_{\ell_2}-2t_{\ell_2-1}}\sum_w p_s(x,w)\hat{p}_s(y,w)\mathrm{d} s,
	$$
	where $\hat{p}_s$ is the transition probability of a random walk with rate  $\r$ that makes a jump at time $t_{\ell_2}-t_{\ell_2-1}$. More precisely,
	\begin{equation}\label{defofhatp}
		\hat{p}_s(y,z):=
		\begin{cases}
			p_s(y,z), &\text{if $s<t_{\ell_2}-t_{\ell_2-1}$,}\\
			\sum_{a,b}p_{t_{\ell_2}-t_{\ell_2-1}}(y,a) \frac{r_{a,b}}{r(a)} p_{s-(t_{\ell_2}-t_{\ell_2-1})}(b,z),&\text{if $s>t_{\ell_2}-t_{\ell_2-1}$}.
		\end{cases}
	\end{equation}
	
	To bound $\hat{p}_s(y,z)$, we will need to look at the term $\frac{r_{a,b}}{r(a)}$. We have the following claim.
	\begin{claim}\label{cl1}
		For all $a,b\in V$,
		$$ \frac{r_{a,b}}{r(a)}\leq
		10R p_{1/r_{\max}}(a,b).
		$$
	\end{claim}
	\begin{proof}[Proof of Claim \ref{cl1}]
		Recall that $R=r_{\max}/r_{\min}$.
		Note
		the event that the random walk transits from $a$ to $b$
		in time $1/r_{\max}$ includes the event that it makes a jump from $a$ to $b$ before time $1/r_{\max}$
		and then stays at $b$ until (at least) $1/r_{\max}$. Thus we have 
		$$
		p_{1/r_{\max}}(a,b)
		\geq \left(1-\exp\left(-\frac{r(a)}{r_{\max}}\right)\right)
		\frac{r_{a,b}}{r(a)} \exp\left(-\frac{r(b)}{r_{\max}}\right).
		$$
		Using the inequality  $1-e^{-x} \geq x/2$ for $0\leq x\leq 1$ we see
		$$
		p_{1/r_{\max}}(a,b)\geq \frac{1}{2e} \frac{r(a)}{r_{\max}} \frac{r_{a,b}}{r(a)} 
		=
		\frac{1}{2e}\frac{r_{a,b}}{r_{\max}}.
		$$
		Rearranging this inequality we get
		$$
		\frac{r_{a,b}}{r(a)}\leq 2e\frac{r_{\max}}{r(a)}p_{1/r_{\max}}(a,b)
		\leq  10\frac{r_{\max}}{r_{\min}}p_{1/r_{\max}}(a,b)=10R p_{1/r_{\max}}(a,b),
		$$
		as desired. 
	\end{proof}
	Now, given the upper bound of $\frac{r_{a,b}}{r(a)}$, we have that for $s>t_{\ell_2}-t_{\ell_2-1}$,
	\begin{equation}\label{bdhatp}
		\hat{p}_s(y,z)
		\leq 10R  
		\sum_{a,b}p_{t_{\ell_2}-t_{\ell_2-1}}(y,a) 
		p_{1/r_{\max}}(a,b)
		p_{s-(t_{\ell_2}-t_{\ell_2-1})}(b,z) =10R 
		p_{s+1/r_{\max} }(x,y).
	\end{equation}
	Hence we can upper bound the numerator $\E N_1$ by
	\begin{multline}
		10R \max_{x,y}
		\int_{t_{\ell_2}-t_{\ell_2-1}}^{t_{\ell_2+1}+t_{\ell_2}-2t_{\ell_2-1}}
		\sum_w p_s(x,w)p_{s+1/r_{\max}}(y,w)\mathrm{d} s
		\\
		\leq
		10R \max_{x,y}
		\int_{t_{\ell_2}-t_{\ell_2-1}}^{t_{\ell_2+1}+t_{\ell_2}-2t_{\ell_2-1}}
		p_{2s+1/r_{\max}}(z,w)\mathrm{d} s.   
	\end{multline}
	Using Lemma \ref{l:poincareineq} and equation \eqref{maxxyp} to replace $\max_{x,y}$ with $\max_x$, and writing $a(\ell_2):=\left\lceil \frac{t_{\ell_2+1}-t_{\ell_2-1}}{t_{\ell_2}-t_{\ell_2-1}} \right\rceil$, we further upper bound this by
	\begin{multline}
		10R
		\left(\left(\sum_{i=1}^{a(\ell_2)}\exp\left(-\frac{i(t_{\ell_2}-t_{\ell_2-1})}{\rel}\right)\right)\max_z\int_{0}^{t_{\ell_2}-t_{\ell_2-1}} p_{2s}(z,z)\mathrm{d} s + \frac{t_{\ell_2+1}-t_{\ell_2-1}}{n} \right)  
		\\
		\leq
		10R\left(
		\frac{\exp(-\Delta_{\min}(\t)/\rel)}{1-\exp(-\Delta_{\min}(\t)/\rel)} \max_z
		\int_{0}^{t_{\ell_2}-t_{\ell_2-1}} p_{2s}(z,z)\mathrm{d} s + \frac{t}{n} \right).
	\end{multline}
	Thus we can bound $P_2$ (defined in \eqref{e:goodsetpf1}) by
	$$
	10R \left(
	\frac{\exp(-\Delta_{\min}(\t)/\rel)}{1-\exp(-\Delta_{\min}(\t)/\rel)}
	\frac{\max_z\int_{0}^{t_{\ell_2}-t_{\ell_2-1}} p_{2s}(z,z)\mathrm{d} s}{\min_z\int_{0}^{t_{\ell_2}-t_{\ell_2-1}} p_{2s}(z,z)\mathrm{d} s} + \frac{8tr_{\max}}{n} \right).
	$$
	Using the fact that $\t$ is $\Delta$-good, and Lemma \ref{l:deltagood}, we can bound this by
	$10R\left(2 e^{-\Delta} + \frac{8tr_{\max}}{n}\right)$.
	Finally, the conclusion follows by summing over all $0\leq \ell_1 < \ell_2 \leq k$.
\end{proof}

\subsection{Estimating $q(\t,I_k)$}  \label{ss:approxq}
\begin{lemma}   \label{l:approxq}
	For any Markov chain  $(V,\r)$ with $\abs{V}=n$ and $\M>3t$,  any $k\in \ZZ_+$, $I_k \in \Phi_k$, and any $\Delta$-good $\t \in \R^k_{<,t}$, we have that
	\begin{equation}
		\left|q(\t,I_k)-\left(\frac{n}{2\M}\right)^k \right|\leq \left(\frac{3tk}{\M} + \frac{2^{k-1}\rel}{G(t, \Delta)}\right)\left(\frac{n}{2\M}\right)^k.
	\end{equation}
\end{lemma}
\begin{proof}
Recall that we let $\U\sim\pi$ and $\tilde \U\sim \nu_{\U}$, and that $W_\U$ and $W_{\tilde \U}$ are two random walks starting from $\U$ and $\tilde \U$ respectively.
We first show that for all $\hat{t} > 0$,
	\begin{equation}   \label{e:approxq:pf1}
		\frac{n}{2\M}\left(1 - \frac{2\rel + \hat{t}}{\M}\right)\leq
		\E\left( r(\U) \1[W_\U(t')\neq W_{\tilde \U}(t'), \forall t'\in[ 0, \hat{t}] ] \right) 
		\leq \frac{n}{2\M}\left(1 + \frac{\rel}{2\hat{t}}\right).
	\end{equation}
	By Lemma \ref{l:reverdiffinit} (in the case $k=1$) and \eqref{kac1}, we have
	$$
	h(\mathbf{s},I_1)=\frac{n}{2}f(\hat{t}),
	$$
	when $\mathbf{s}:=(0,\hat{t},2\hat{t})$.
	Here $f$ is the density function of the meeting time of two independent random walks starting from $\pi^{\otimes 2}$.  
	By the definition of the branching structure,
	\begin{equation}\label{eq:densityfortwo}
		\E\left( r(\U) \1[W_\U(t')\neq W_{\tilde \U}(t') \text{ for all } t'\in[ 0, \hat{t}] ]\right)
		=h(\s,I_1)=
		\frac{n}{2}f(\hat{t}).
	\end{equation}

	Consider the product chain of two independent walks on $(V,\r)$. It has the same relaxation time as the original chain. Using \eqref{ABdensity} for the diagonal set in the product chain, we get
	\begin{equation}\label{eq:fhattctl}
		\frac{1}{\M} \left(1-\frac{2\rel(\r)+\hat{t}}{\M} \right)\leq
		f(\hat{t}) \leq \frac{1}{\M}\left(1+\frac{\rel(\r)}{2\hat{t}}\right),
	\end{equation}
	where we have used the definition  $\M=\E_{\pi,\pi}(\RM)$. 
	Thus equation \eqref{e:approxq:pf1} follows from \eqref{eq:densityfortwo} and \eqref{eq:fhattctl}.
	
	By taking $\hat{t} =\hat{t}_{\ell} = t_{\ell+1}-t_{\ell}$, for $1 \leq \ell \leq k$, we have
	\begin{equation*}
		\left(\frac{n}{2\M}\right)^k\prod_{\ell=1}^k\left(1 - \frac{2\rel + t_{\ell+1}-t_{\ell}}{\M}\right)\leq
		q(\t, I_k)
		\leq \left(\frac{n}{2\M}\right)^k\prod_{\ell=1}^k\left(1 + \frac{\rel}{2(t_{\ell+1}-t_{\ell})}\right),
	\end{equation*}
	and this implies that
	\begin{equation}   \label{e:approxq:pf3}
		\left(\frac{n}{2\M}\right)^k\left(1 - \frac{3t}{\M}\right)^k\leq
		q(\t, I_k)
		\leq \left(\frac{n}{2\M}\right)^k\left(1 + \frac{\rel}{2\Delta_{\min}(\t)}\right)^k.
	\end{equation}
	Since $3t < \M$, we have
	\begin{equation}\label{lobdq}
		q(\t, I_k) - \left(\frac{n}{2\M}\right)^k
		\geq - \frac{3tk}{\M}\left(\frac{n}{2\M}\right)^k.
	\end{equation}
	Since $\t$ is $\Delta$-good, we have $\Delta_{\min}(\t) > G(t, \Delta)$. Recall that we only consider $\Delta>2$, so we have $\rel \leq  \Delta_{\min}(\t)$. It follows that
	\begin{equation}\label{upbdq}
		q(\t, I_k) - \left(\frac{n}{2\M}\right)^k
		\leq \frac{2^{k-1}\rel}{G(t, \Delta)}\left(\frac{n}{2\M}\right)^k.
	\end{equation}
	Thus our conclusion follows from \eqref{lobdq} and \eqref{upbdq}. 
\end{proof}

\subsection{From moments to coalescing probability}  \label{ss:fram}
We start by assembling the above results to estimate $\E N_t^k$ for each $k\in \ZZ_+$, and large enough $t$.

\begin{lemma}  \label{l:ntk-est}
	For any Markov chain  $(V,\r)$ with $\abs{V}=n$, and any $k\in \ZZ_+$, $t < \M/3$, we have
	\begin{multline}
		\left| \E N_t^k - (k+1)!\left(\frac{nt}{2\M}\right)^k \right|
		\\
		\leq
		C_k\left((1+r_{\max}t)^{k-1} + 
		e_k(t) + \left(\frac{t}{\M} + \frac{\rel}{G(t, \Delta)}\right)\left(\frac{nt}{2\M}\right)^k
		\right.
		\\
		\left.+
		r_{\max}^k t^k
		R\left(e^{-\Delta}
		+ \frac{tr_{\max}}{n}\right)
		+
		G(t, \Delta)t^{k-1} \left(r_{\max}^k + \left(\frac{n}{2\M}\right)^k  \right)
		\right),
	\end{multline}
	where $G(t, \Delta)$ is defined in \eqref{G(t)}, and $C_k$ is some constant depending on $k$.
\end{lemma}
\begin{proof}
	By \eqref{moment}  and Lemma \ref{densityexpress},
	we have
	\begin{equation}\label{eq:ntk}
		\begin{split}
			&\left| \E(N_t^k) - (k+1)!\sum_{I_k \in \Phi_k} \int_{\R^{k}_{<, t}} h(\t,I_k)d\t\right| \\
			\leq &\sum_{1\leq \ell\leq k-1} (k+1)^{\ell+1}(\ell+1)! \sum_{I_\ell\in \Phi_\ell}\int_{\R^{\ell}_{<, t}  }  h(\t,I_\ell)\mathrm{d} \t+1
			\\
			\leq &
			\sum_{1\leq \ell\leq k-1} (k+1)^{\ell+1}(\ell+1)! (r_{\max}t)^\ell+1.
		\end{split}
	\end{equation}
	We also have
	\begin{multline}\label{eq:htikctl}
		\left|\sum_{I_k \in \Phi_k}\int_{\R^{k}_{<, t}} h(\t,I_k) \mathrm{d} \t - \left(\frac{nt}{2\M}\right)^k\right|
		=
		\left|\sum_{I_k \in \Phi_k}\int_{\R^{k}_{<, t}} \left(h(\t,I_k) - \left(\frac{n}{2\M}\right)^k \right) \mathrm{d} \t\right|
		\\
		\leq
		\left|\sum_{I_k \in \Phi_k}\int_{\t \text{ is } \Delta\text{-good}} \left(h(\t,I_k) - \left(\frac{n}{2\M}\right)^k \right)\mathrm{d} \t\right|
		+
		k!\cdot k G(t, \Delta)t^{k-1} \left(r_{\max}^k + \left(\frac{n}{2\M}\right)^k  \right),
	\end{multline}
	where in the last inequality we use the definition of $\Delta$-good (to argue that $\int_{\t:\, \t \text{ is not } \Delta\text{-good}}\mathrm{d}\t \le k G(t, \Delta)t^{k-1} $), as well as the bound  $h(\t,I_k) \le r_{\max}^k$.
	By Lemma \ref{l:approxhbyg} and Lemma \ref{l:approxq}, and the definition of $e_k(t)$, the first term in the second line of equation \eqref{eq:htikctl} can be bounded above by
	\[
		k!e_k(t) + \left(\frac{3tk}{\M} + \frac{2^{k-1}\rel}{G(t, \Delta)}\right)\left(\frac{nt}{2\M}\right)^k
		+
		r_{\max}^k t^k \frac{k(k+1)}{2}
		(10R+1)\left(2 e^{-\Delta} + \frac{8tr_{\max}}{n}\right).
	\]
	By putting everything together the conclusion follows.
\end{proof}

\begin{proof}[Proof of Proposition \ref{p:gengraph}]
	The value of $K$ is determined later.
	Throughout this proof we will use $C_K$ to denote a constant depending only on $K$, and its value may change from line to line.
	
	Our main task is to bound $\E N_t$ as well as $\E N_t^k/(\E N_t)^k$.
	By \eqref{N_t} and the fact that $\M r_{\max}/n \geq 1/4$ (by Lemma \ref{walkparamaters}), we have that
	\[
	\abs{\E N_t-\frac{nt}{\M}} \leq 
	\left(\frac{t}{\M}+\frac{\rel}{t}\right)\frac{nt}{\M}\leq \left(\frac{4r_{\max}t}{n} + \frac{G(t,\Delta)}{t}\right) \frac{nt}{\M}\leq C_K(E_1+E_2)\frac{nt}{\M}.
	\]
	To estimate $(\E N_t^k) \M^k /(nt)^k$, we control the (relative) error terms in Lemma \ref{l:ntk-est} one by one:
	\begin{equation*}
		\begin{split}
			&(r_{\max}t)^{k-1}\left(\frac{\M}{nt}\right)^k =
			\left(\frac{\M r_{\max}}{n}\right)^k \frac{1}{r_{\max}t} \leq C_KE_4, \qquad
			\left(\frac{\M}{nt}\right)^k \leq C_KE_4^k,\\
			& \sum_{k=2}^K e_k(t)\left(\frac{\M}{nt}\right)^k\leq E_3, \qquad
			\frac{t}{\M}+\frac{\rel}{G(t,\Delta)}\leq C_K(E_2+\delta),\\
			&r_{\max}^k
			R\left(e^{-\Delta}
			+ \frac{tr_{\max}}{n}\right)\left(\frac{\M}{n}\right)^k\leq C_K(\delta+E_2),\\
			&G(t, \Delta)t^{k-1} \left(r_{\max}^k + \left(\frac{n}{2\M}\right)^k\right)
			\left(\frac{\M}{nt}\right)^{k}\leq C_KE_1.
		\end{split}
	\end{equation*}
	It follows that  under the first condition of Proposition
	\ref{p:gengraph}, we have
	\begin{equation}  \label{eq:509}
		\begin{split}
			\abs{\E N_t-\frac{nt}{\M}}&\leq
			C_K\delta \frac{nt}{\M}, \\
			\abs{\frac{\E N_t^k}{(\E N_t)^k}-
				\frac{(k+1)!}{2^k}} &\leq C_K \delta, \quad \mbox{for } 1\leq k\leq K.
		\end{split}
	\end{equation}
	From this, we see that by picking a sufficiently small  $\delta$, we can make $\E N_t\M/(nt)$ arbitrarily close to $1$ and 
	$\E N_t^k/(\E N_t)^k$ arbitrarily close to $(k+1)!/2^k$ for $1\leq k\leq K$. 
	Consider a random variable $\Gamma$ under the Gamma distribution with 
	the probability density function  $\P(\Gamma=x)=4x e^{-2x}\1[x\geq 0]$.
	Its $k$-th moment is $(k+1)!/2^k$, and it also satisfies the Carleman's moment condition 
	$$
	\sum_{m=0}^{\infty}(\E\Gamma^{2m})^{-1/2m}= \sum_{k\text{ even}} \frac{1}{
		\left(  (k+1)!/2^k\right)^{1/k}}=\infty .
	$$
	Hence the Gamma distribution is the unique distribution with the prescribed moments $(k+1)!/2^k$.
	We claim that 
	for any bounded piecewise continuous function $\psi$ on $\R$ and any $\eta>0$, there exist $K_0,\delta_0$ depending on $\psi$ and $\eta$, s.t. for any $K\geq K_0,\delta<\delta_0$, if Condition $1$ \eqref{firstcond} is satisfied for such $(K,\delta)$, then we have
	\begin{equation}\label{ingre0}
		\abs{\E\left(\psi\left(\frac{N_t}{\E(N_t)}\right)  \right)-\E(\psi(\Gamma))}<\eta.
	\end{equation}
	Indeed, assuming this is not true, we can find a sequence of (non-negative) random variables $\{\Gamma^{(n)}\}_{n\in\ZZ_+}$, whose $k$-th moment converges to $(k+1)!/2^k$ but $|\E(\psi(\Gamma^{(n)})) - \E(\psi(\Gamma))|$ does not decay to zero.
	The moment condition also means that $\{\Gamma^{(n)}\}_{n\in\ZZ_+}$ is tight, so by taking a subsequence we can assume that $\abs{\Gamma}^{(n)}$ converges in distribution, and the limit must have its $k$-th moment equal to $(k+1)!/2^k$ (since $\abs{\Gamma^{(n)}}^k$ is uniformly integrable by the convergence of all moments) and therefore
	must be $\Gamma$ by the Carleman's moment condition.
	We conclude that when $\psi$ is bounded and continuous, we have $|\E(\psi(\Gamma^{(n)})) - \E(\psi(\Gamma))|\to 0$; and this extends to any bounded piecewise continuous $\psi$ since the probability density function of $\Gamma$ is bounded.

	For later reference, we remark that the arguments leading to \eqref{ingre0} also imply that for any fixed $0<a<1$, the $k$-th moment of  $N_{at}/\E(N_{at})$ is close to that of $\Gamma$ (for the same values of $k$ for which this holds for $a=1$). Hence, equation \eqref{ingre0} also holds for $N_{at}/\E(N_{at})$ (with possibly different $K_0,\delta_0$).
	
	Let $\hat{\ep}$ be some small number which satisfies 
	\begin{equation}\label{ingre2}
		\E(\Gamma^{-1} \1[\Gamma <\hat{\ep}])\leq \frac{\ep}{5}
	\end{equation}
	and is to be determined.
	Using \eqref{ingre0} 
	we can pick $K$ large enough and  $\delta$  small enough (depending on $\ep$ and $\hat{\ep}$), so that 
	\begin{equation}\label{ingre1}
		\abs{\E\left(\frac{\E N_t}{N_t};\frac{N_t}{\E N_t}\geq \hat{\ep}\right)-\E(\Gamma^{-1};\Gamma\geq \hat{\ep}) }\leq \frac{\ep}{5}.
	\end{equation}
	It remains to estimate $\E\left(\E N_t/N_t;N_t/\E N_t<\hat{\ep}\right)$. For this we adapt an argument from \cite[Lemma 2]{bramson1980asymptotics}.
	By Corollary \ref{c:Ntnt1} we see that
	\begin{equation}\label{ingre8}
		\E\left(\frac{\E N_t}{N_t};\frac{N_t}{\E N_t}<\hat{\ep}\right)
		=\E N_t\P(0<n_t<\hat{\ep}\E N_t),
	\end{equation}
	where $n_t$ is the size of the opinion cluster of a uniformly random vertex in the voter model, as defined at the end of Section \ref{s:prelim}. Write $\eta=\sqrt{\hat \varepsilon}$, then the probability on the right hand side of \eqref{ingre8} can be bounded from above by
	$$
	\P \left(n_{(1-\eta)t}\geq \hat{\ep}\E N_t\right)-
	\P \left(n_{t}\geq \hat{\ep}\E N_t\right)+
	\P \left(n_{(1-\eta)t}<\hat{\ep}\E N_t, n_t>0       \right)=:\frac{J_1-J_2+J_3}{\E N_t}.
	$$
	By choosing $\hat{\ep}<1/8$ (so that $(1-\hat{\ep})^{-1}<(1+\hat{\ep})^2$), taking $\delta$ small and using  \eqref{eq:509}, we have 
	\begin{equation}\label{ctlEN_t}
		\begin{split}
			&(1-\hat{\ep})
			\frac{nt}{{(1-\eta)\M}} \leq \E N_{(1-\eta)t}\leq (1+\hat{\ep})
			\frac{nt}{{(1-\eta)\M}},\\
			& (1-\hat{\ep})\frac{nt}{\M}\leq \E N_t\leq (1+\hat{\ep})\frac{nt}{\M}\leq (1+\hat{\ep})^3\E N_{(1-\eta)t}.
		\end{split}
	\end{equation}
	It follows that 
	\begin{align*}
	   J_1&=\E N_t \E(N_{(1-\eta)t}^{-1}\1[N_{(1-\eta)t}>\hat{\ep}\E N_t])\\
	   &\leq (1+\hat{\ep})^{4} \E\left(\frac{\E N_{(1-\eta)t}}{N_{(1-\eta)t}};\frac{N_{(1-\eta)t}}{\E N_{(1-\eta)t}}> \frac{\hat{\ep}}{(1+\hat{\ep})^4}\right).
	\end{align*}
	Using the definition of $J_2$ together with Corollary \ref{c:Ntnt1} we get that
	$$
	J_2=\E\left(\frac{\E N_t}{N_t};\frac{N_t}{\E N_t}>\hat{\ep}\right).
	$$ Using this, as well as equation \eqref{ingre0}, we see
	that if we choose $K$ large enough and $\delta$ small enough, we have  
	\begin{equation}\label{ingre3}
		J_1-J_2\leq (1+\hat{\ep})^5\E\left(\Gamma^{-1}\1\left[\Gamma> \frac{\hat{\ep}}{(1+\hat{\ep})^4}\right]\right) -(1+\hat{\ep})^{-1} \E(\Gamma^{-1} \1[\Gamma> \hat{\ep}]),
	\end{equation}
	which can be made smaller than $\ep/5$ by taking $\hat{\ep}$ sufficiently small. 
	
	Finally, using the second condition \eqref{secondcond} together with \eqref{ctlEN_t} (to bound $\E N_t$), we see that
	\begin{equation}\label{ingre4}
	\begin{split}
	    	J_3&\leq \E N_t\sum_{A\subset V:1\leq \abs{A}\leq \hat{\ep} \E(N_t)}
		\P(\zeta^\U_{(1-\eta)t}=A)
		\P(n_t>0|\zeta^\U_{(1-\eta)t}=A)\\
		&\leq \left(\E N_t\P(n_{(1-\eta)t}>0)\right) \left(\hat{\ep} \E N_t \sup_{x\in V}P_{\eta t}(x)\right)\\
		&\leq \frac{\hat{\ep}((1+\hat{\ep})K')^2}{\eta (1-\eta)}=\frac{\sqrt{\hat{\ep}}((1+\hat{\ep})K')^2 }{1-\sqrt{\hat{\ep}}},
	\end{split}
	\end{equation}
	which can be made smaller than $\ep/5$ by taking $\hat{\ep}$ small enough. Here, we have used a union bound and the duality between the voter model and CRW in the second inequality of \eqref{ingre4}: 
	\begin{equation*}
		\begin{split}
			\P(n_t>0|\zeta^\U_{(1-\eta)t}=A)
			=\P\left( \cup_{x\in A} \{\zeta^x_{\eta t}\neq \emptyset\} \right)
			&\leq \sum_{x\in A} 
			\P\left( \{\zeta^x_{\eta t}\neq \emptyset\} \right)\\
			&=\sum_{x\in A}  P_{\eta t}(x) \leq \abs{A}\sup_{x\in V}P_{\eta t}(x).  
		\end{split}
	\end{equation*}
	Combining equations \eqref{ingre2} to \eqref{ingre4} we get $$
	\abs{\E\left(\frac{\E N_t}{N_t}\right)-\E(\Gamma^{-1}) }\leq 4\ep/5.
	$$
	Hence by using \eqref{ctlEN_t} to control $(\E N_t)^{-1}$ we get
	$$
	\abs{P_t-\frac{2\M}{nt}}=\abs{  (\E N_t)^{-1}
		\E\left(\frac{\E N_t}{N_t}\right)-\frac{2\M}{nt}}
	\leq \left(\frac{10\hat{\ep}}{1-\hat{\ep}}(1+\frac{4\ep}{5})+\frac{4\ep}{5} \right)\frac{\M}{nt},
	$$
	which can be made smaller than $(\ep \M)/(nt)$  by choosing $\hat{\ep}<\ep/100$. 
	
	In conclusion, for the given $\ep>0$, we first pick $\hat{\ep}$ small enough s.t.\ \eqref{ingre2} holds and the right hand side of \eqref{ingre3} is smaller than $\ep/5$. Then we choose a large enough $K$ and a small enough $\delta$ s.t.\  \eqref{ingre1}, \eqref{ctlEN_t}, \eqref{ingre3} and \eqref{ingre4} hold. This completes the proof of Proposition \ref{p:gengraph}. 
\end{proof}

\section{Random graphs: the configuration model and Galton-Watson trees}
\label{s:pfrandom}
In this section, we use the framework given by Proposition \ref{p:gengraph} to study the decay of density for the configuration model and unimodular Galton-Watson trees. Throughout this section, we let $D$ be a degree distribution that is at least $3$ and upper bounded (i.e., the degree distribution has finite support).
For CRW, the jump rate is taken as $r_{x,y}=\1[x \sim y]$.

To prove Theorem \ref{t:config}, we will apply Proposition \ref{p:gengraph}, checking that the conditions are satisfied. The most efforts will be devoted to controlling the term $E_3$ defined in equation \eqref{d:capitalI}, i.e., showing that the behaviors of the particles after each branching are similar to being independent. 
In other words, we show that the coalescing process of $k+1$ particles, starting from i.i.d.\ stationary locations, evolves similarly to the case in which after each collision the positions of the remaining particles are re-sampled to be i.i.d.\ stationary. 
We also give estimates for $\frac{2\M}{n}$, showing that it converges to $\alpha(D)^{-1}$ in probability as $n\to \infty$, with $\alpha(D)$ defined in equation \eqref{eq:alpha(D)}. 

One important feature of the configuration model that helps in our analysis is that it is an expander graph with high probability. Recall that in equation \eqref{d:kappa} we defined $\kappa(\G)$ to be the vertex expansion constant of the graph $\G$, i.e., 
\begin{equation*}
	\kappa(\G):= \min_{0 < |S| \leq \frac{|V|}{2}} \frac{|\partial S|}{|S|},
\end{equation*}
where $\partial S:= \{v \in V\setminus S: \exists u \in S, u\sim v\}$ is the out-boundary of $S$.

\begin{lemma}\label{l:configroughbound}
    There exists a constant $\kappa_0$, such that the following is true. Take $\G_n \sim \CM_n(D)$, conditioned on that it is connected.
	Then we have $\P(\kappa(\G_n) \geq \kappa_0) \to 1$ as $n\to\infty$. Consequently, $\lim\limits_{n\to \infty}\P(\rel(\G_n) \leq 2\bar d/\kappa_0^2) = 1$, where $\rel(\G_n)$ is the relaxation time of $\G_n$ and $\bar d$ is the supremum of the support of $D$.
\end{lemma}
The first part of this result can be found in \cite[Theorem 6.3.2]{durrett2007random}.
The second part follows from the first part together with Lemma \ref{gapandtrel} (Cheeger's inequality).

\subsection{Truncate random walks within a ball}   \label{ss:trunc}
We define some intermediate quantities. For a vertex $v$ and a distance $\rho \in \ZZ_+$, define $N_{\rho}(v)=N_{\rho}(v,\G_n)$ to be the subgraph of $\G_n$ induced on $\{u\in V_n:d(u,v)\leq \rho\}$ (the ball of radius $\rho$ centered at $v$), where $d(\cdot,\cdot)$ is the graph distance in $\G_n$. Recall the construction of the random paths $\gamma_1,\ldots,\gamma_{k+1}$, given $\t$ and a collision pattern $I_k$. In particular, we let $A_{\ell}=\gamma_{i_{\ell}}(t_{\ell})$ and $B_{\ell}$ be sampled from $\nu_{A_{\ell}}$, for each $1 \le \ell \le k$. Let $\tau_\ell$ be the first time after $t_\ell$ when $\gamma_\ell$ or $\gamma_{i_{\ell}}$ exits $N_{\rho}(\gamma_\ell(t_\ell))$ (whichever occurs first). We consider the following quantity
\begin{equation}\label{defhat_g}
	\hat g(\t, I_k)=\hat g(\t, t, I_k,\rho):= \E
	\left(\prod_{\ell=1}^{k}  r(A_\ell) \1[\gamma_{\ell}(t')\neq \gamma_{i_{\ell}}(t')  \text{ for all } t'\in[t_{\ell},\tau_\ell \wedge t_{\ell+1}] ]\right),
\end{equation}
which describes the weighted probability that no coalescence happens between a newly born branch and its parent until one of them exits $N_{\rho}(\gamma_\ell(t_\ell))$ or until the next branching time (whichever happens first).

Take $\U\sim\pi$ and $\tilde \U\sim \nu_{\U}$, and two independent walks $W_\U$ and $W_{\tilde \U}$ starting from $\U$ and $\tilde \U$, respectively. We define $\tau_\U$ to be the first time that $W_\U$ or $W_{\tilde \U}$ exits $N_{\rho}(\U)$ (whichever occurs first), and set
\begin{equation}\label{defhat_q}
	\hat q(\t, I_k)=\hat q(\t, t, I_k,\rho):= \prod_{\ell=1}^{k}\E
	\left( r(\U) \1[W_\U(t')\neq W_{\tilde \U}(t') \text{ for all } t'\in[ 0,\tau_\U\wedge (t_{\ell+1}-t_\ell)]\right).
\end{equation}
The quantity $\hat{q}(\t, I_k)$ describes the weighted probability that no coalescence happens between a newly born branch and its parent until one of them exits $N_{\rho}(\gamma_\ell(t_\ell))$ or the next branching time, assuming that the place of branching is uniformly distributed. For the sake of ease of notation, below we suppress the dependence of $\hat g(\t, I_k)$ and $\hat q(\t, I_k)$ on $\rho$ and $t$ in our notations.

We show that  $\hat{g}(\t,I_k)$ and $\hat{q}(\t,I_k)$ well approximate $g(\t,I_k)$ and $q(\t,I_k)$ (defined in \eqref{gt,phi_k} and \eqref{cond_form4}) provided $1\ll t\ll n$ and $\rho \gg 1$, for any fixed graph with lower bounded vertex expansion constant and upper bounded degree. 
\begin{lemma}  \label{l:trung}
	Let $\kappa_0 > 0$ and $\bar{d}\in \ZZ_+$.
	For any graph $\G=(V,E)$ with $\kappa(\G)>\kappa_0$ and maximum degree bounded by $\bar{d}$, we have that for all $\rho>100$, $k\in \ZZ_+$, $I_k \in \Phi_k$ and any time sequence $\t$,
	\begin{equation}
		0 \leq \hat{g}(\t,I_k) -g(\t,I_k) \leq 
		8\bar{d}^{k} \left(
		kC_{\kappa_0, \bar{d}}\exp(-\sqrt{\rho}/\bar{d}\rel)
		+\frac{t\bar{d}}{n} 
		\right),
	\end{equation}
	\begin{equation}
		0 \leq \hat{q}(\t,I_k) -q(\t,I_k) \leq
		8k\bar{d}^{k} \left(
		kC_{\kappa_0, \bar{d}}\exp(-\sqrt{\rho}/\bar{d}\rel)
		+\frac{t\bar{d}}{n}
		\right),
	\end{equation}
	where $C_{\kappa_0, \bar{d}}$ is the constant from Lemma \ref{l:boundH}, and $n=|V|$.
\end{lemma}
\begin{proof}
	Comparing the definitions of $\hat g$ and $g$ and those of $\hat q$ and $q$ one readily sees that
	\begin{equation}\label{g-hg}
		\begin{split}
			0\leq &\hat{g}(\t,I_k)-g(\t,I_k)\\
			= &\E
			\biggl(  \left(\prod_{\ell=1}^{k}  r(A_\ell) \1[\gamma_{\ell}(t')\neq \gamma_{i_{\ell}}(t') \text{ for all } t'\in[ t_{\ell},\tau_\ell \wedge t_{\ell+1}]]\right)\\
			&\times \1[\gamma_{\ell}(t'')=\gamma_{i_{\ell}}(t'') \text{ for some } 1 \leq \ell \leq k \text{ and some } t'' \in [\tau_{\ell}\wedge t_{\ell+1}, t_{\ell+1}] ]
			\biggl)
			\\
			\leq & \bar{d}^{k} \sum_{\ell=1}^{k}
			\P \left(
			\{ \gamma_\ell(t')\neq \gamma_{i_\ell}(t') \text{ for all } t' \in [t_\ell, \tau_\ell \wedge t_{\ell+1}] \} \cap \{ \gamma_\ell(t'')= \gamma_{i_\ell}(t'')
			\text{ for some } t'' \in [\tau_\ell, t_{\ell+1}]\}
			\right),
		\end{split}
	\end{equation}
	where we used $\max_{\ell}r(A_{\ell}) \le \bar d$. Further, we have that
	\begin{equation}\label{g-hq}
		\begin{split}
			0\leq &\hat{q}(\t,I_k)-q(\t,I_k)
			\leq
			\bar{d}^k \sum_{\ell=1}^k
			\P (\{W_\U(t')\neq W_{\tilde{\U}}(t')
			\text{ for all } t' \in [0, \tau_\U \wedge t_{\ell+1}-t_\ell]  \}\\
			& \cap \{ W_\U(t'')= W_{\tilde{\U}}(t'') \text{ for some } t'' \in [\tau_\U, t_{\ell+1}-t_\ell]
			\}).
		\end{split}
	\end{equation}
	For any $0<T<\Delta_{\min}(\t)$ and $1\leq \ell \leq k$, we have,
	\begin{multline}
		\P\left(\{
		\gamma_\ell(t')\neq \gamma_{i_\ell}(t') \text{ for all } t' \in [t_\ell, \tau_\ell \wedge t_{\ell+1}] \} \cap \{\gamma_\ell(t')= \gamma_{i_\ell}(t') \text{ for some }
		t' \in [\tau_\ell, t_{\ell+1}], \}
		\right)
		\\
		\leq
		\P\left(\{ \gamma_\ell(t') \neq \gamma_{i_\ell}(t') \text{ for all } t' \in [t_\ell, t_{\ell}+T]\} \cap \{
		\gamma_\ell(t')= \gamma_{i_\ell}(t') \text{ for some }t' \in [t_\ell+T, t_{\ell+1}\}]\right)
		\\ +
		\P(\tau_{\ell} < t_\ell + T),
	\end{multline}
	and similarly,
	\begin{multline}
		\P\left(\{W_\U(t')\neq W_{\tilde{\U}}(t') \text{ for all }
		t' \in [0, \tau_\U \wedge t_{\ell+1}-t_\ell]\}  \cap \{ W_\U(t')= W_{\tilde{\U}}(t') \text{ for some } t' \in [\tau_\U, t_{\ell+1}-t_\ell]
		\} \right)\\
		\leq
		\P \left(\{ W_\U(t')\neq W_{\tilde{\U}}(t') \text{ for all }
		t' \in [0, T]\} \cap \{ W_\U(t')= W_{\tilde{\U}}(t') \text{ for some } t' \in [T, t_{\ell+1}-t_\ell]\} \right)
		\\ + \P(\tau_\U < T).
	\end{multline}
	By Lemma \ref{meetprob}, the first terms in the r.h.s.\ of the above two equations can be bounded by
	\begin{multline}
		2\exp(-T/\rel)
		\frac{\max_{z} \int_0^{2T}  p_{s}(z,z)\mathrm{d} s }{\min_{z} \int_0^{2T} p_s(z,z)\mathrm{d} s  }+\frac{8(t_{\ell+1}-t_\ell)}{n} (T^{-1}\vee \bar{d})
		\\
		\leq
		2C_{\kappa_0, \bar{d}}\exp(-T/\rel)
		+\frac{8(t_{\ell+1}-t_\ell)}{n} (T^{-1}\vee \bar{d}),
	\end{multline}
	where the inequality and the constant $C_{\kappa_0, \bar{d}}$  come from  Lemma \ref{l:boundH}.
	
	We now bound each of $\P(\tau_{\ell} < t_\ell + T)$ and $\P(\tau_{\U} < T)$. 
	Note that in order for the event $\{\tau_{\ell} < t_\ell + T\}$ (resp.\ $\{\tau_{\U} < T\}$) to occur we must have that either $\gamma_\ell$ or $\gamma_{i_\ell}$ (resp.\ $W_\U$ or $W_{\tilde \U}$) needs to jump at least $\rho$ times during $[t_\ell ,t_\ell +T]$ (resp.\ $[0,T]$). Since the jump rate is bounded by $\bar{d}$,
	the law of the number of jumps that a particle performs in a time interval of length $T$ is stochastically dominated by the $\Poi(T\bar{d})$ distribution.
	Hence we have that
	$$
	\P(\tau_{\ell} < t_\ell + T), \P(\tau_{\U} < T)
	\leq
	2\sum_{i=\rho}^{\infty} \frac{(T\bar{d})^ie^{-T\bar{d}}}{i!}.
	$$
	Now we take $T = \sqrt{\rho}/\bar{d}$.
	Then we have
	\begin{equation}
		\begin{split}
			0\leq &\hat{g}(\t,I_k)-g(\t,I_k), \;\hat{q}(\t,I_k)-q(\t,I_k)\\
			\leq & k \bar{d}^{k} \left(
			2C_{\kappa_0, \bar{d}}\exp(-\sqrt{\rho}/(\bar{d}\rel))
			+\frac{8t\bar{d}}{n} 
			+ \frac{4\rho^{\rho/2} e^{-\sqrt{\rho}}}{\rho!}
			\right),
		\end{split}
	\end{equation}
	which can be further upper bounded by
	\begin{equation}
		\begin{split}
			& k\bar{d}^{k} \left(
			2C_{\kappa_0, \bar{d}}\exp(-\sqrt{\rho}/(\bar{d}\rel))
			+\frac{8t\bar{d}}{n} 
			+ \frac{4\rho^{\rho/2} e^{-\sqrt{\rho}}}{\sqrt{2\pi}\rho^{\rho+1/2}e^{-\rho}}
			\right)
			\\
			\leq &
			k\bar{d}^{k} \left(
			2C_{\kappa_0, \bar{d}}\exp(-\sqrt{\rho}/(\bar{d}\rel))
			+\frac{8t\bar{d}}{n} 
			+ 
			4\exp(- \rho)
			\right).
		\end{split}
	\end{equation}
	This implies our conclusion since 
	$$
	\frac{\sqrt{\rho}}{\bar d \rel} \leq 
	\frac{2\bar d\sqrt{\rho}}{\bar d}\leq \rho
	$$
	by equation \eqref{relandrmax} and the assumption $\rho>100$.
\end{proof}

\subsection{Almost independence among branching: estimate $e_k(t)$}
In the previous subsection, we compared $\hat g(\t,I_k)$ to $g(\t,I_k)$ and $\hat q(\t,I_k)$ to $q(\t,I_k)$. In this subsection, we will proceed to compare $\hat g(\t,I_k)$ and $\hat q(\t,I_k)$. Recall the definition of $\hat{g}(\t,I_k)$ and $\hat{q}(\t,I_k)$ in \eqref{defhat_g} and \eqref{defhat_q}, respectively. 
Now if we write 
\begin{equation}  \label{eq:defM}
	M_\ell:=r(A_\ell) \1[\gamma_{\ell}(t')\neq \gamma_{i_{\ell}}(t') \text{ for all } t'\in[ t_{\ell},\tau_\ell \wedge t_{\ell+1}]],
\end{equation}
then,
\begin{equation}
	\hat g(\t,I_k)=\E\left(\prod_{\ell=1}^k M_\ell\right).
\end{equation}
Define 
\begin{equation}
	Q_\ell:=\E \left( r(\U) \1[W_\U(t')\neq W_{\tilde \U}(t')\text{ for all } t'\in[ 0,\tau_\U\wedge (t_{\ell+1}-t_{\ell})]]\right),
\end{equation}
then
\begin{equation}
	\hat q(\t,I_k)=\prod_{\ell=1}^k Q_\ell.
\end{equation}

We next define a few properties of the graph $\G$. 

\begin{definition}\label{def:delta_homo}
	For any $\delta > 0$, a graph $\G$ is said to be \emph{$\delta$-homogeneous} if for any $i \in \ZZ_+$, we have either $|\{v \in V:\deg(v)=i\}|\geq \delta n$ or $|\{v \in V:\deg(v)=i\}|=0$. 
	In addition, we say that $i \in \ZZ_+$ is \emph{feasible} (for $\G$) if $|\{v \in V:\deg(v)=i\}|>0$. For notational convenience, we also consider $0$ as feasible.
\end{definition}
The construction of $\CM_n(D)$ implies the following lemma. 
\begin{lemma} \label{lemma:delta-mix}
	For any degree distribution $D$ that is upper bounded by some constant,  there exists $\delta_D>0$,
	such that if we take $\G \sim \CM_n(D)$, then $\lim\limits_{n\to \infty} \P(\G \textrm{ is } \delta_D\textrm{-homogeneous}) = 1$.
\end{lemma}

\begin{definition}
	For $\rho \in \ZZ_+$, we call a graph $\G$  \emph{$\rho$-locally tree like} if  $|\{v \in V:N_\rho(v) \text{ is a tree}\}|\geq n-\sqrt{n}$.
\end{definition}

\begin{lemma}  \label{l:localtreelike}
	For any $\rho \in \ZZ_+$, and $\G \sim \CM_n(D)$, we have $$\lim_{n\to \infty}\P(\G \text{ is }\rho\textrm{-locally tree like})= 1.$$
\end{lemma}
\begin{proof}
	This is a direct consequence of the local weak convergence of the configuration model to unimodular Galton-Watson tree, as stated in Lemma \ref{CMLWC}. (A quantitative estimate can be derived via Markov's inequality, by noting that the expected number of vertices $v$ such that $N_{\rho}(v)$ is not a tree is upper bounded by some constant depending only on $\bar d$ and $\rho$.)
\end{proof}

In order to bound the difference between $\hat q(\t,I_k)$ and $\hat g(\t,I_k)$, we will use induction. We will show inductively that, for any $\ell\in\ZZ_+$ and any sequence of feasible numbers $m_1, \ldots, m_{\ell-1}$, we have $$\|\mathcal{D}(A_\ell \mid M_1=m_1,\ldots, M_{\ell-1}=m_{\ell-1})-\pi\|_{\TV} \mbox{  and  }\left|\E\left(\prod_{j=1}^\ell M_j\right)-\prod_{j=1}^{\ell} Q_j\right|$$ are small while $\P\left(\cap_{j=1}^{\ell} (M_j=m_j)\right)$ is not too small. Here $\mathcal{D}(\cdot)$ denotes the law of the random variable, and $\pi$ is the stationary distribution of the random walk (in our case, it is the uniform distribution over $V_n$). 

We now state the induction step as the following two lemmas. 
\begin{lemma}\label{l:induct1}
	Let $c_1, c_2, c_3>0$, $\delta >c_1$ and $\ell\in \ZZ_+$. Assume $t_{\ell+1}-t_{\ell}\geq 1$. Suppose $\G$ is $\delta$-homogeneous with maximum degree bounded by $\d$. Assume further that for any sequence of feasible numbers $m_1, \ldots, m_{\ell-1}$ (not necessarily mutually distinct), we have
	\begin{equation}\label{eq:assump1}
		\|\mathcal{D}(A_\ell \mid M_1=m_1,\ldots, M_{\ell-1}=m_{\ell-1})-\pi \|_{\TV} \leq c_1.
	\end{equation}
	Also assume that
	\begin{equation}\label{eq:assump2}
		\left|\E\left(\prod_{j=1}^{\ell-1} M_j\right)-\prod_{j=1}^{\ell-1} Q_j\right|\leq c_2 \prod_{j=1}^{\ell-1} Q_j,
	\end{equation}
	and
	\begin{equation}\label{eq:assump3}
		\P\left(\cap_{j=1}^{\ell-1} \{M_j=m_j\}\right)\geq c_3.
	\end{equation}
	Then for $m_\ell$ feasible, we have
	\begin{equation}
		\left|\E\left(\prod_{j=1}^\ell M_j\right)-\prod_{j=1}^\ell Q_j\right|\leq (c_1(1+c_2)\d^{2\rho+2}+c_2)\prod_{j=1}^{\ell} Q_j,
	\end{equation}
	and
	\begin{equation}
		\P\left(\cap_{j=1}^\ell \{M_j=m_j\}\right)\geq(\delta-c_1)c_3\frac{1}{\d^{2\rho+1}}.
	\end{equation}
\end{lemma}
\begin{lemma}\label{l:induct2}
	Let $c_4>0$ and $1\leq \ell\leq K-1$, and $C, \rho>0$. Let $\G$ be a $\rho$-locally tree like graph with $n$ vertices, maximum degree bounded by $\d$. Take any sequence of feasible integers $m_1,\cdots,m_\ell$, suppose we have
	\begin{equation}
		\P\left(\cap_{j=1}^\ell \{M_j=m_j\}\right)\geq c_4.
	\end{equation}
	Then when $\Delta_{\min}(\t)-C\rel-\rho>0$, we have
	\begin{equation}
		\|\mathcal{D}(A_{\ell+1} \mid M_1=m_1,\ldots, M_{\ell}=m_{\ell})-\pi\|_{\mathrm{TV}}\leq \d^{\ell}\exp(-C)\left(\frac{1}{c_4-\varepsilon}\right)+\frac{\varepsilon}{c_4},
	\end{equation}
	if $c_4>\varepsilon$.
	Here $\varepsilon$ is defined as
	\begin{align}\label{defep}
		\varepsilon = \varepsilon(\Delta_{\min} (\t),\rho, C, n):= \frac{3(\Delta_{\min}(\t)-C\rel)}{(\Delta_{\min}(\t)-C\rel-\rho)^2}+\frac{\d^K}{\sqrt{n}}.
	\end{align}
\end{lemma}
Deferring the proofs of these two lemmas to the next subsection, we now use them to do the induction.

\begin{lemma}\label{l:hatgq}
	For any $\delta_1$-homogeneous and $\rho$-locally tree like graph $\G$ and for $C>0$ and $1\leq k\leq K$, if
	\begin{enumerate}
		\item $\Delta_{\min}(\t)\geq 1$,
		\item $\Delta_{\min}(\t)-C\rel-\rho>0$,
		\item $\varepsilon\leq \min\{\frac{1}{4}(\frac{\delta_1}{2\d^{2\rho+1}})^{2k},\frac{1}{64k^2\d^{4\rho+4}}\}$ for $\varepsilon$ defined in Lemma \ref{l:induct2},
		\item $\exp(C)\geq\frac{\d^k}{\varepsilon(\Delta_{\min} (\t),\rho, C,n)^2}$,
	\end{enumerate}
	then
	\[
	|\hat g(\t,\phi_k)-\hat q(\t,\phi_k)|\leq 8k \d^{2\rho+2} \sqrt{\varepsilon(\Delta_{\min} (\t),\rho, C,n)}\hat q(\t,\phi_k),
	\]
	when $n$ is large enough.
\end{lemma}
We note that the constants in the lemma are far from being optimal. We take those for the convenience of computation as will be shown below. They are complicated at first sight, but they naturally arise in our computations.
\begin{proof} [Proof of Lemma \ref{l:hatgq}]
	The proof is done inductively using Lemmas \ref{l:induct1} and \ref{l:induct2}. We will show that for $1\leq \ell\leq k$,
	\begin{align}
		&\|\mathcal{D}(A_\ell \mid M_1=m_1,\ldots, M_{\ell-1}=m_{\ell-1})-\pi\|_{\TV} \leq 2\sqrt{\ep}, \label{eq:induct1}\\
		&\left|\E\left(\prod_{j=1}^\ell M_j\right)-\prod_{j=1}^{\ell} Q_j\right|\leq 8\ell\sqrt{\varepsilon} \d^{2\rho+2} \prod_{j=1}^{\ell} Q_j,\label{eq:induct2}\\
		&\P\left(\cap_{j=1}^{\ell} (M_j=m_j)\right)\geq \left(\frac{\delta_1}{2\d^{2\rho+1}}\right)^\ell.\label{eq:induct3}
	\end{align}
	When $\ell=1$, we have $\mathcal{D}(A_1)=\pi$, $\E M_1=Q_1$ and $\P(M_1=m_1)\geq \delta_1 \frac{1}{\d^{2\rho+1}}$, so equations \eqref{eq:induct1}$-$\eqref{eq:induct3} hold. Now assume that equations  \eqref{eq:induct1}$-$\eqref{eq:induct3} hold for $\ell-1$. The conditions $\Delta_{\min}(t) \geq 1$ and $\varepsilon(\Delta_{\min} (\t),\rho, C,n)\leq \frac{1}{4}(\frac{\delta_1}{2\d^{2\rho+1}})^{2k}$ and equations \eqref{eq:induct1} and \eqref{eq:induct2} ensure that the assumptions in Lemma \ref{l:induct1} hold. Therefore, by Lemma \ref{l:induct1} and the assumption that $\varepsilon\leq \frac{1}{64k^2\d^{4\rho+4}}$,
	\begin{align*}
		\left|\E\left(\prod_{j=1}^\ell M_j\right)-\prod_{j=1}^{\ell} Q_j\right|&\leq ((1+8(\ell-1)\sqrt{\varepsilon} \d^{2\rho+2} )(2\sqrt{\varepsilon})\d^{2\rho+2}+8(\ell-1)\sqrt{\varepsilon} \d^{2\rho+2} ) \prod_{j=1}^{\ell} Q_j\\
		&\leq 8\ell \sqrt{\varepsilon} \d^{2\rho+2} \prod_{j=1}^{\ell} Q_j.
	\end{align*}
	By Lemma \ref{l:induct1}, equation \eqref{eq:induct1}, equation \eqref{eq:induct3} and the assumption that $\epsilon\leq \frac{1}{4}(\frac{\delta_1}{2\bar{d}^{2\rho+1}})^{2k}$,
	\begin{align*}
		\P\left(\cap_{j=1}^{\ell} (M_j=m_j)\right)&\geq (\delta_1-2\sqrt{\varepsilon})\frac{1}{\d^{2\rho+1}}\left(\frac{\delta_1}{2\d^{2\rho+1}}\right)^{\ell-1}
		\geq \left(\frac{\delta_1}{2\d^{2\rho+1}}\right)^\ell.
	\end{align*}
	By Lemma \ref{l:induct2}, equation \eqref{eq:induct3} and the assumptions that $\varepsilon(\Delta_{\min} (\t),\rho, C,\delta_2)\leq \frac{1}{4}(\frac{\delta_1}{2\bar{d}^{2\rho+1}})^{2k}$ and $\exp(C)\geq\frac{\d^k}{\varepsilon^2}$, we have 
	\begin{align*}
		&\quad \|\mathcal{D}(A_\ell \mid M_1=m_1,\ldots, M_{\ell-1}=m_{\ell-1})-\pi\|_{\TV}  \\
		&\leq \d^k\exp(-C) \left( \left(\frac{\delta_1}{2\bar{d}^{2\rho+1}}\right)^{\ell-1}-\varepsilon \right)^{-1}+
		\varepsilon \left(\left(\frac{\delta_1}{2\bar{d}^{2\rho+1}}\right)^{\ell-1}\right)^{-1}\\
		&\leq \varepsilon^2 \left(\left(\frac{\delta_1}{2\bar{d}^{2\rho+1}}\right)^{\ell-1}-\varepsilon\right)^{-1}+\sqrt{\varepsilon}\\
		&\leq 2\sqrt{\varepsilon}.
	\end{align*}
	Therefore, equations \eqref{eq:induct1}$-$\eqref{eq:induct3} hold for $1\leq \ell\leq k$. In particular, when $\ell=k$, we have $\hat g(\t,\phi_k)=\E(\prod_{j=1}^k M_j)$ and $\hat q(\t,\phi_k)=\prod_{j=1}^k Q_j$, and
	\begin{equation}
		\left|\E\left(\prod_{j=1}^k M_j\right)-\prod_{j=1}^k Q_j\right|\leq 8k\d^{2\rho+2}\sqrt{\varepsilon(\Delta_{\min} (\t),\rho, C,n)} \prod_{j=1}^k Q_j.
	\end{equation}
	Thus the conclusion follows.
\end{proof}
\noindent Recall  that we have defined previously the error terms in Proposition \ref{p:gengraph}:
\[
\hat{e}_{k}(\t):=\sup_{I_k\in \Phi_k}\abs{q(\t,I_k)-g(\t,I_k)},\quad
e_k(t):=\int_{0<t_1<\ldots <t_k<t}\hat{e}_{k}(\t) \mathrm{d} \t,
\]
and $E_3=\sum_{k=2}^K \left(\frac{\M}{tn}\right)^{k}e_k(t)$.
We next control $E_3$, by choosing appropriate intermediate parameters and applying Lemma \ref{l:trung} and Lemma \ref{l:hatgq}. 
\begin{prop}   \label{p:control_ek}
	For any $K\in \ZZ_+$ with $K\geq 2$, $0<\delta_1\leq 1$, and $\delta,\kappa_0,C_{\alpha}>0$, there exists an appropriate choice of $\rho=\rho(\delta,\kappa_0,C_\alpha)\in \ZZ_+$, such that whenever $\G$ 
	has maximal degree upper bounded by $\d$,
	is $\delta_1$-homogeneous and $\rho$-locally tree like, $\kappa(\G)>\kappa_0$ and $\frac{\M}{n}\leq C_{\alpha}$, then
	\begin{equation}
		E_3(t)\leq \delta,
	\end{equation}
	for any $t$ large enough and $n$ such that $n/t$ is large enough.
\end{prop}
\begin{proof}
	Take appropriate $\rho>100$ such that
	\begin{align}\label{a:choice0}
		16\bar{d}^{K} \left(KC_{\kappa_0, \bar{d}}\exp(-\sqrt{\rho}/\bar{d}\rel) \right)\leq 16\bar{d}^{K}C_{\kappa_0, \bar{d}} \exp\left(-\sqrt{\rho}\frac{\kappa_0^2}{2\d^2}\right)\leq    \frac{\delta}{3KC_\alpha^K}.
	\end{align}
	For a chosen $\rho$, take $\Delta_0-C\rel$ large enough such that 
	\begin{align*}
		\frac{3(\Delta_0-C\rel)}{(\Delta_0-C\rel-\rho)^2}\leq \frac{1}{2} \min\left\{\frac{1}{4}\left(\frac{\delta_1}{2\d^{2\rho+1}}\right)^{2K},\frac{1}{64K^2\d^{4\rho+4}} ,\left(\frac{1}{3K^2C_\alpha^K\d^{2\rho+K+2}}\right)^2\right\}.
	\end{align*}
	Now, fix $\Delta_0-C\rel$ and pick $C$ large enough such that 
	\begin{align*}
		\exp(C)\geq\frac{\d^K}{\left(\frac{3(\Delta_0-C\rel)}{(\Delta_0-C\rel-\rho)^2}\right)^2}.
	\end{align*}
	This implies that when $n$ is large, we have
	\begin{align}
		&\varepsilon(\Delta_0,\rho,C,n)=\frac{3(\Delta_0-C\rel)}{(\Delta_0-C\rel-\rho)^2}+\frac{\d^K}{n} \nonumber \\ 
		&\quad \quad \qquad \qquad \leq  \min\left\{\frac{1}{4}\left(\frac{\delta_1}{2\d^{2\rho+1}}\right)^{2K},\frac{1}{64K^2\d^{4\rho+4}} ,\left(\frac{1}{3K^2C_\alpha^K\d^{2\rho+K+2}}\right)^2\right\}, \label{a:choose1}\\
		&\exp(C)\geq\frac{\d^K}{\varepsilon(\Delta_0,\rho,C,n)^2}=\frac{\d^K}{\left(\frac{3(\Delta_0-C\rel)}{(\Delta_0-C\rel-\rho)^2}+\frac{\d^K}{n}\right)^2}, \label{a:choose2}\\
		&8K\d^{2\rho+K+2}\sqrt{\frac{3(\Delta_0-C\rel)}{(\Delta_0-C\rel-\rho)^2}+\frac{\d^K}{n}} =8K\d^{2\rho+K+2}\sqrt{\varepsilon(\Delta_0,\rho,C,n)}\leq \frac{\delta}{3KC_\alpha^K}. \label{a:choose3}
	\end{align}
	Now we estimate $\hat e_k(\t)$ for $2\leq k\leq K$. 
	\begin{align*}
		\hat{e}_{k}(\t)&=\sup_{I_k\in \Phi_k}\abs{q(\t,I_k)-g(\t,I_k)}\\
		&\leq \sup_{I_k\in \Phi_k}\abs{ q(\t,I_k)-\hat q(\t,I_k)}+\sup_{I_k\in \Phi_k}\abs{\hat q(\t,I_k)-\hat g(\t,I_k)}+\sup_{I_k\in \Phi_k}\abs{\hat g(\t,I_k)-g(\t,I_k)}.
	\end{align*}
	For the first and the third term, by Lemma \ref{l:trung} and equation \eqref{a:choice0},
	\begin{align*}
		&\sup_{I_k\in \Phi_k}\abs{ q(\t,I_k)-\hat q(\t,I_k)}+\sup_{I_k\in \Phi_k}\abs{\hat g(\t,I_k)-g(\t,I_k)}\\
		&\leq 16\bar{d}^{k} k\left(C_{\kappa_0, \bar{d}}\exp(-\sqrt{\rho}/\bar{d}\rel)+\frac{t\bar{d}}{n} \right)\\
		&\leq \frac{\delta}{3KC_\alpha^K}+ 16k\d^{k+1}\frac{t}{n}.
	\end{align*}
	For the second term, by Lemma \ref{l:hatgq} and the fact that $Q_{\ell}\leq \d$,
	\begin{align*}
		&\sup_{I_k\in \Phi_k}\abs{\hat q(\t,I_k)-\hat g(\t,I_k)} \leq 8k \d^{2\rho+K+2} \sqrt{\frac{3(\Delta_{\min}(\t)-C\rel)}{(\Delta_{\min}(\t)-C\rel-\rho)^2}+\frac{\d^K}{n}}.
	\end{align*}
	When $\Delta_{\min}(\t)\geq \Delta_0$, this is further bounded by $\frac{\delta}{3KC_\alpha^K}$ using equation \eqref{a:choose3}. It follows that
	\begin{align*}
		&e_k(\t)=\int_{0<t_1<\ldots <t_k<t} \hat e_t(\t)\mathrm{d} \t\\
		&\leq \int_{\R^{k}_{<, t}}\left( \sup_{I_k\in \Phi_k}\abs{ q(\t,I_k)-\hat q(\t,I_k)}+\sup_{I_k\in \Phi_k}\abs{\hat g(\t,I_k)-g(\t,I_k)}+\sup_{I_k\in \Phi_k}\abs{\hat q(\t,I_k)-\hat g(\t,I_k)} \right) \mathrm{d} \t\\
		&\leq t^k \left(\frac{\delta}{3KC_\alpha^K}+16\d^{k+1}\frac{t}{n}\right)\\
		&+\int_{\R^{k}_{<, t}, \Delta_{\min}(\t) > \Delta_0} \sup_{I_k\in \Phi_k}\abs{\hat q(\t,I_k)-\hat g(\t,I_k)}\mathrm{d} \t+\int_{\R^{k}_{<, t}, \Delta_{\min}(\t) \leq  \Delta_0} \sup_{I_k\in \Phi_k}\abs{\hat q(\t,I_k)-\hat g(\t,I_k)}\mathrm{d} \t\\
		&\leq t^k \left(\frac{\delta}{3KC_\alpha^K}+16k\d^{k+1}\frac{t}{n}\right)+ t^k\frac{\delta}{3KC_\alpha^K}+\Delta_0t^{k-1} \sup_{\R^{k}_{<, t}} \hat e_k(\t)\\
		&\leq t^k \left(\frac{\delta}{3KC_\alpha^K}+16k\d^{k+1}\frac{t}{n}\right)+ t^k\frac{\delta}{3KC_\alpha^K}+2\Delta_0t^{k-1} \d^k.
	\end{align*}
	This implies that
	\begin{align*}
		E_3(t)&=\sum_{k=2}^K \left(\frac{\M}{tn}\right)^{k}e_k(t)\\
		&\leq \sum_{k=2}^K \frac{C_\alpha^k}{t^k} t^k\left(\frac{\delta}{3KC_\alpha^K}+16k\d^{k+1}\frac{t}{n}+\frac{\delta}{3KC_\alpha^K}+2\Delta_0\frac{1}{t}\d^k\right)\\
		&\leq \frac{2\delta}{3}+16K^2C_\alpha^k\d^{k+1}\frac{t}{n}+2KC_\alpha^k\Delta_0\d^k\frac{1}{t}.
	\end{align*}
	When $t$ and $n/t$ are large enough, we have that
	$
	E_3(t)\leq \delta
	$,
	thus our conclusion follows.
\end{proof}

\subsection{Proofs of induction steps}
In this subsection we prove Lemmas \ref{l:induct1} and \ref{l:induct2}.
\begin{proof}[Proof of Lemma \ref{l:induct1}]
	Let $m_1,\ldots,m_{\ell}$ be feasible. Denote $B_i:=\cap_{j=1}^i \{M_j=m_j\} $. Observe that $M_{\ell}$ is conditionally independent of $(M_1,\ldots,M_{\ell-1})$ given $A_{\ell}$ (indeed, $M_1,\ldots,M_{\ell-1}$ do not involve $\gamma_{\ell}$, and they involve $\gamma_{i_{\ell}}$ only at intervals contained in $[0,t_{\ell}]$). Hence
	\begin{align*}
		&\P\left(B_{\ell} \right)=\P\left(M_\ell=m_\ell \mid B_{\ell-1} \right) \P\left(B_{\ell-1} \right)\\
		&=\sum_{v\in V}\P\left(M_\ell=m_{\ell} \mid A_\ell=v,B_{\ell-1} \right)\P\left(A_\ell=v \mid B_{\ell-1} \right) \P\left(B_{\ell-1}\right)\\
		&=\sum_{v\in V}\P\left(M_\ell=m_\ell \mid A_{\ell}=v\right)\P\left(A_{\ell}=v \mid B_{\ell-1} \right) \P\left(B_{\ell-1}\right)
	\end{align*}
	which is at least
	\begin{align*}
		\sum_{v\in V}\P(M_\ell=m_\ell \mid A_\ell=v)\P\left(A_\ell=v \mid B_{\ell-1}\right)c_3,
	\end{align*}
	using assumption (\ref{eq:assump3}). Note that $M_\ell=m_\ell$ for some feasible $m_\ell \in \mathbb{N}$ if and only if one of the following two cases happen. Either $m_\ell>0$, $\deg(A_\ell)=m_\ell$ and the indicator from the definition of $M_\ell$ occurs or $m_\ell=0$ and indicator from the definition of $M_\ell$ doesn't occur. Now if $m_\ell>0$, when $A_{\ell}=v$ and $\deg (v)=m_\ell$, we have
	\begin{equation}\label{lbnoncollide}
		\P(M_\ell=m_\ell \mid A_\ell=v)=\P(\gamma_{\ell}(t')\neq \gamma_{i_{\ell}}(t') \text{ for all } t'\in[ t_{\ell},\tau_\ell \wedge t_{\ell+1}])
		\geq \frac{1}{\d^{2\rho+1}}.
	\end{equation}
	This can be seen by considering two arbitrary disjoint paths connecting $A_\ell$ and $N_\rho(A_\ell)^c$ and $\nu_{A_\ell}$ and $N_\rho(A_\ell)^c$ respectively, as possible trajectories for the two walks until they exit $N_{\rho}(A_\ell)$. This implies that when $m_\ell>0$,
	\begin{align*}
		&\sum_{v\in V}\P\left(M_\ell=m_\ell \mid A_{\ell}=v\right)\P\left(A_{\ell}=v \mid B_{\ell-1} \right) c_3\\
		\geq &\sum_{v\in V, \deg v=m_\ell}\P\left(A_{\ell}=v \mid B_{\ell-1} \right) \frac{1}{\d^{2\rho+1}}c_3\\
		= &\P\left(\deg (A_\ell)=m_\ell \mid B_{\ell-1} \right) \frac{1}{\d^{2\rho+1}}c_3.
	\end{align*} 
	As we are in the $\delta$-homogeneous graph $\mathcal{G}$, the number of vertices with feasible degree $m_\ell$ (when $m_\ell>0$) is lower bounded by $\delta n$. Using our assumption (\ref{eq:assump1}), when $m_\ell>0$,
	\begin{align*}
		\P\left(\cap_{j=1}^\ell \{M_j=m_j\}\right)\geq
		(\delta-c_1)c_3\frac{1}{\d^{2\rho+1}}.
	\end{align*}
	Now, when $m_\ell=0$ and  $t_{\ell+1}-t_{\ell}\geq 1$, we have
	\begin{equation}\label{lbnoncollide2}
		\P(M_\ell=m_\ell \mid A_\ell=v)=1-\P(\gamma_{\ell}(t')\neq \gamma_{i_{\ell}}(t') \text{ for all } t'\in[ t_{\ell},\tau_\ell \wedge t_{\ell+1}])
		\geq \frac{1}{8\d}>\frac{1}{\bar d^{2\rho+1}}.
	\end{equation}
	Indeed, the probability that $\gamma_{\ell}(t)$ stays at $A_\ell$ during $[t_\ell,t_\ell+\frac{\log 2}{\bar d}]$ and that $\gamma_{i_{\ell}}(t)$ jumps to $A_\ell$ by time $t_\ell+\frac{\log 2}{\bar d}$ is at least $\frac{1}{8\d}$. Therefore when $m_\ell=0$, we also have
	\begin{align*}
		\P\left(\cap_{j=1}^\ell \{M_j=m_j\}\right)\geq \frac{1}{8\d}c_3>
		(\delta-c_1)c_3\frac{1}{\d^{2\rho+1}}.
	\end{align*}
	Since $r(\U)\geq 1$, the above equations \eqref{lbnoncollide} and \eqref{lbnoncollide2} imply $Q_{\ell}\geq\d^{-(2\rho+1)}$.
	Now we consider
	\begin{align*}
		\E\left(\prod_{j=1}^\ell M_j\right)&=\E\left(\E(M_\ell \mid M_1,\ldots,M_{\ell-1})\prod_{j=1}^{\ell-1} M_j\right)\\
		&=\E\left(\E(\E(M_\ell \mid A_\ell) \mid M_1,\ldots,M_{\ell-1})\prod_{j=1}^{\ell-1} M_j\right).
	\end{align*}
	By (\ref{eq:assump1}) and as $0\leq M_\ell\leq r_{\max}\leq \d$,
	\begin{equation}
		|\E(\E(M_\ell \mid A_\ell) \mid  M_1,\ldots,M_{\ell-1})-Q_{\ell}|\leq c_1 \d \leq c_1 \d^{2\rho+2}Q_{\ell}.
	\end{equation}
	Therefore,
	\begin{equation}
		\left|\E\left(\prod_{j=1}^{\ell} M_j\right)-\E\left(\prod_{j=1}^{\ell-1} M_j\right)Q_\ell\right|\leq   c_1\d^{2\rho+2} \E\left(\prod_{j=1}^{\ell-1} M_j\right)Q_\ell.
	\end{equation}
	Using assumption \eqref{eq:assump2}, we have
	\begin{align*}
		\left|\E\left(\prod_{j=1}^\ell M_j\right)-\prod_{j=1}^\ell Q_j\right|&\leq \left|\E\left(\prod_{j=1}^{\ell} M_j\right)-\E\left(\prod_{j=1}^{\ell-1} M_j\right)Q_\ell\right|+\left|\E\left(\prod_{j=1}^{\ell-1} M_j\right)Q_\ell-\prod_{j=1}^\ell Q_j\right|\\
		&\leq c_1\d^{2\rho+2} \E\left(\prod_{j=1}^{\ell-1} M_j\right) Q_\ell +c_2 \prod_{j=1}^\ell Q_j\\
		&\leq \left((1+c_2)c_1\d^{2\rho+2}+c_2\right) \prod_{j=1}^\ell Q_j.
	\end{align*}
	Thus our conclusion follows.
\end{proof}

To prove Lemma \ref{l:induct2}, we present an intermediate result that controls the exit time $\tau_\ell$. 
Recall that $\varepsilon = \frac{3(\Delta_{\min}(\t)-C\rel)}{(\Delta_{\min}(\t)-C\rel-\rho)^2}+\frac{\d^K}{\sqrt{n}}$.
\begin{lemma}\label{l:escapetime}
	Let $c_4>0$ and $1\leq \ell\leq K-1$. For any $\rho$-locally tree like graph $\G$ with maximum degree bounded by $\d$ and any sequence of feasible integers $m_1,\ldots,m_\ell$, if we assume that 
	\begin{equation}
		\P\left(\cap_{j=1}^\ell \{M_j=m_j\}\right)\geq c_4,
	\end{equation}
	then
	\begin{equation}
		\P(t_{\ell+1}-\tau_{\ell}< C\rel \mid  M_1=m_1,\ldots, M_{\ell}=m_{\ell})\leq\frac{1}{c_4}\varepsilon,
	\end{equation}
	when $\Delta_{\min}(\t)-C\rel-\rho>0$ (the $C$ here is the same as the one in the definition of $\ep$ in \eqref{defep}).
\end{lemma}
\begin{proof}
	Intuitively, for most vertices $v$, its neighborhood $N_\rho(v)$ is a tree. In this case, a random walk starting from $v$ has a drift towards the leaves, so the time of exit cannot be too large. In the following, we write down the idea rigorously. Notice that,
	\begin{align*}
		\P(t_{\ell+1}-\tau_{\ell}< C\rel \mid M_1=m_1,\ldots, M_{\ell}=m_{\ell})
		\leq &\frac{\P(t_{\ell+1}-\tau_{\ell}< C\rel)} {\P(M_1=m_1,\ldots, M_{\ell}=m_{\ell})}\\
		\leq & \frac{\P(t_{\ell+1}-\tau_{\ell}< C\rel)}{c_4}.
	\end{align*}
	To control the numerator, we consider the probability
	$
	\P(\tau_\ell-t_\ell>s).
	$
	For $s\in [0,t_{\ell+1}-t_\ell]$, $\gamma_{i_\ell}(s+t_\ell)$ is a random walk starting at $A_\ell$. We denote the graph distance of $\gamma_{i_\ell}(s+t_\ell)$ to $A_\ell$ by $d_1(s)$. Define another random process $d_2(s)$ to be a (time-homogeneous) drifted continuous-time random walk on $\ZZ$ starting from $0$, where the transition rate from $x$ to $x+1$ is $2$ and from $x$ to $x-1$ is $1$ for all $x\in \ZZ$ and all other rates are zero. If $N_\rho(A_\ell)$ is a tree with minimal degree at least 3, then $d_1(s)$ stochastically dominates $d_2(s)$. If we define $\tau'$ to be the first time that $d_2(s)$ gets larger than $\rho$, then given $N_\rho(A_\ell)$ is a tree,
	\begin{align*}
		\P(\tau_\ell-t_\ell>s)\leq \P(\tau'> s).
	\end{align*}
	To bound the right hand side, we compute the second moment. Define $d_3(s):=d_2(s)-s$, then $d_3(s)$ is a continuous-time martingale, with $\E d_3(s)=0$ and $\mathrm{Var}(d_3(s))=3s$. Then, by Chernoff bound, given $N_\rho(A_\ell)$ is a tree,
	\begin{align*}
		\P(\tau_\ell-t_\ell>s)\leq \P(\tau'> s)\leq \P(d_2(s)\leq  \rho)=\P(d_3(s)\leq  \rho-s)\leq \frac{3s}{(s-\rho)^2}.
	\end{align*}
	\noindent Substituting $t_{\ell+1}-t_\ell-C\rel$ for $s$ in the above inequality, we have that 
	\begin{align*}
		\P(t_{\ell+1}-\tau_{\ell}< C\rel)&\leq\P(t_{\ell+1}-\tau_{\ell}< C\rel, N_\rho(A_\ell) \text{ is a tree })+\P(N_\rho(A_\ell) \text{ is not a tree})\\
		&\leq \frac{3(t_{\ell+1}-t_\ell-C\rel)}{(t_{\ell+1}-t_\ell-C\rel-\rho)^2}+\P(N_\rho(A_\ell) \text{ is not a tree}).
	\end{align*}
	By Lemma \ref{maxgamma_t(x)}, we have that for any $v\in V$, $\P(A_\ell=v)\leq \frac{\bar d^\ell}{n}$. So $\P(N_\rho(A_\ell) \text{ is not a tree })\leq \frac{\bar d^\ell}{\sqrt{n}}$. By definition of $\Delta_{\min} (\t)$, we know that $t_{\ell+1}-t_\ell\geq \Delta_{\min} (\t)$. This implies that, 
	\begin{align*}
		&\P(t_{\ell+1}-\tau_{\ell}< C\rel\mid M_1=m_1,\ldots, M_{\ell}=m_{\ell})\\
		\leq &\frac{1}{c_4} \left(\frac{3(t_{\ell+1}-t_\ell-C\rel)}{(t_{\ell+1}-t_\ell-C\rel-\rho)^2}+\frac{\d^K}{\sqrt{n}}\right)\\
		\leq &\frac{1}{c_4}  \left(\frac{3(\Delta_{\min}(\t)-C\rel)}{(\Delta_{\min}(\t)-C\rel-\rho)^2}+\frac{\d^K}{\sqrt{n}}\right),
	\end{align*}
	and our conclusion follows.
\end{proof}

\begin{proof}[Proof of Lemma \ref{l:induct2}]
	Observe that
	\begin{align*}
		&\P(A_{\ell+1}=v\mid M_1=m_1,\ldots, M_{\ell}=m_{\ell})\\
		&=\P(A_{\ell+1}=v\mid t_{\ell+1}-\tau_{\ell}\geq C\rel, M_1=m_1,\ldots, M_{\ell}=m_{\ell})\\
		&\qquad\qquad\qquad\P(t_{\ell+1}-\tau_{\ell}\geq C\rel \mid M_1=m_1,\ldots, M_{\ell}=m_{\ell})\\
		&+\P(A_{\ell+1}=v \mid t_{\ell+1}-\tau_{\ell}< C\rel, M_1=m_1,\ldots, M_{\ell}=m_{\ell})\\
		&\qquad\qquad\qquad\P(t_{\ell+1}-\tau_{\ell}< C\rel \mid M_1=m_1,\ldots, M_{\ell}=m_{\ell}).
	\end{align*}
	By Lemma \ref{l:escapetime}, $\P(t_{\ell+1}-\tau_{\ell}< C\rel \mid M_1=m_1,\ldots, M_{\ell}=m_{\ell})\leq \frac{\varepsilon}{c_4}$. Hence
	\[\P(t_{\ell+1}-\tau_{\ell}\geq C\rel, M_1=m_1,\ldots, M_{\ell}=m_{\ell})\geq c_4-\varepsilon.\]
	For any $v\in V$,
	\begin{align*}
		&\P(\gamma_{\ell}(t_{\ell+1}-C\rel)=v \mid t_{\ell+1}-\tau_{\ell}\geq C\rel, M_1=m_1,\ldots, M_{\ell}=m_{\ell})\\
		\leq &\P(\gamma_{\ell}(t_{\ell+1}-C\rel)=v)/\P(t_{\ell+1}-\tau_{\ell}\geq C\rel, M_1=m_1,\ldots, M_{\ell}=m_{\ell})\\
		\leq &\frac{\d^{\ell}}{(c_4-\varepsilon)n}.
	\end{align*}
	In the last step we used Lemma \ref{maxgamma_t(x)} with $R\leq \d$.
	When $t_{\ell+1}-\tau_{\ell}\geq C\rel$, given $\gamma_{\ell}(t_{\ell+1}-C\rel)$, $A_{\ell+1}$ is independent of $M_1,\ldots,M_\ell$. Therefore, by \eqref{tvd}, we have
	\begin{align*}
		&\sum_v \left(\P(A_{\ell+1}=v \mid t_{\ell+1}-\tau_{\ell}\geq C\rel, M_1=m_1,\ldots, M_{\ell}=m_{\ell})-\frac{1}{n}\right)^2\\
		\leq& \exp(-2C)\sum_v \left(\P(\gamma_{\ell}(t_{\ell+1}-C\rel)=v \mid t_{\ell+1}-\tau_{\ell}\geq C\rel, M_1=m_1,\ldots, M_{\ell}=m_{\ell})-\frac{1}{n}\right)^2,
	\end{align*}
	which is smaller than
	$$ \exp(-2C)\frac{1}{n}\left(\frac{\d^\ell}{c_4-\varepsilon}-1\right)^2.$$
	Therefore,
	\begin{equation}
		\sum_v \left|\P(A_{\ell+1}=v \mid t_{\ell+1}-\tau_{\ell}\geq C\rel, M_1=m_1,\ldots, M_{\ell}=m_{\ell})-\frac{1}{n}\right| \leq \exp(-C)\left(\frac{\d^\ell}{c_4-\varepsilon}\right).
	\end{equation}
	And finally using Lemma \ref{l:escapetime} again, we have that 
	\begin{align*}
		&\|\mathcal{D}(A_{\ell+1} \mid M_1=m_1,\ldots, M_{\ell}=m_{\ell})-\pi\|_{\TV}\\
		&\leq \exp(-C)\left(\frac{\d^\ell}{c_4-\varepsilon}\right)+\frac{\ep}{c_4}.
	\end{align*}
	Thus our conclusion follows.
\end{proof}

\subsection{Limit of $\frac{\M}{n}$ in the configuration model}\label{ssc:tmeetnconfig}
In this section we will prove the convergence of 
$\M/n$ to $(2\alpha(D))^{-1}$ (defined in  equation 
	\eqref{eq:alpha(D)}) in the configuration model. 
\begin{lemma}\label{tmmetlim}
	Let $\G_n$ be sampled from the configuration model $\CM_n(D)$ and condition on that $\G_n$ is connected. Then
	$\M(\G_n)/n$ converges in probability to $(2\alpha(D))^{-1}$ as $n$ goes to $\infty$.
\end{lemma}
\begin{remark}
	Recall that $\UGT(D)$ denotes the unimodular Galton-Watson tree and $o$ is the root of the tree. Also recall $\tau_{\mathrm{meet}}$ is the (random) meeting time of two independent random walks. Define
\begin{equation}
\alpha(D,k)=\lim_{t\to\infty}\P_{o,\nu_o}(\RM>t|D_o=k),
\end{equation}
where the probability is taken over both the law of random walks and the law of the Galton-Watson tree.
Then
	for all $k$ such that $\P(D=k)>0$, we have $\alpha(D,k)>0$. Indeed, since we assume the minimal degree is 3, the random walk on  $\UGT(D)$ with unit jump rate along each edge is transient. 
\end{remark}

\begin{proof}[Proof of Lemma \ref{tmmetlim}]
	We apply formula \eqref{kac2} to the product chain of two independent random walks and let $A$ be the diagonal set $A:=\{(x,x):x\in V_n\}$. 
	Since $r_{x,y}=\1[x\sim y]$, we have
	\[
	Q(A,A^c)=\sum_x \frac{1}{n^2}2 \sum_y \1[x\sim y]=\frac{2\sum_x D_x}{n^2}.
	\]
	For any pair  $(x,y)\in V_n\times V_n-A$,
	the exit measure $\nu_A$ puts mass 
	\[\frac{(r_{x,y}+r_{y,x})/n^2}{(2\sum_z r(z))/n^2}= \frac{\1[x\sim y]}{\sum_z D_z}\]
	on $(x,y)$. We get
	\begin{equation}\label{QAA}
		Q(A,A^c)\E_{\nu(A)}(T_A)=\frac{2\sum_z D_z}{n^2}\sum_{x\neq y} \frac{\1[x\sim y]}{\sum_{z'} D_{z'}}\E(\C(W_x,W_y))=\frac{2}{n^2}\sum_x D_x \E_{y \sim \nu_x}(\C(W_x,W_y)),
	\end{equation}
	where $\C$ is the coalescence time and $W_x,W_y$ are two independent random walks starting from $x$ and $y$, respectively, and $y$ is sampled from $\nu_x$. 
	
	Recall that $\U$ is defined to
	be a uniform random vertex and $\tilde{\U}$ a uniform random neighbor of $\U$. 
	The last expression in equation \eqref{QAA} can be rewritten as
	\[
	\frac{2}{n}\sum_k k\P(D_{\U}=k)\E(\C(W_{\U}, W_{\tilde{\U}}) \mid D_{\U}=k).
	\]
	Since we are considering the product chain,  we have $\pi^{\otimes 2}(A)=\sum_{x\in V} 1/n^2=1/n$ and $\pi^{\otimes 2}(A^c)=1-1/n$. Hence using \eqref{kac2} we have
	\begin{equation}\label{kac4}
		\sum_k k\P(D_{\U}=k)\E(\C(W_{\U}, W_{\tilde{\U}}) \mid  D_{\U}=k)=\frac{n}{2}\left(1-\frac{1}{n}\right).
	\end{equation}
	To compute $\M$ we make the following claim.
	\begin{claim}\label{claim2}
	The following two statements hold.
	\begin{enumerate}
		\item     The (quenched) conditional probability (i.e., we fix the random graph) $$
		\P(\C(W_{\U}, W_{\tilde{\U}})>4\rel(\G_n) \log n \mid D_{\U}=k,\G_n)
		$$ converges to $\alpha(D,k)>0$ in probability as $n\to\infty$. 
		\item For any $\ep>0$ and any $k$ such that $\P(D=k)>0$, the following event 
		\begin{equation*}
			\begin{split}
				\{&\abs{\E(\C(W_{\U}, W_{\tilde{\U}}) \mid D_{\U}=k, \G_n)-\P(\C(W_{\U}, W_{\tilde{\U}})>4\rel(\G_n) \log n \mid D_{\U}=k,\G_n)\M(\G_n) } \\ \leq & \ep \M(\G_n)+16\log n\}
			\end{split}
		\end{equation*}
		has probability tending to 1 as $n\to\infty$. 
	\end{enumerate}
	\end{claim}
	It now follows from  Claim \ref{claim2} and equation \eqref{kac4} that 
	$\M/n$ converges to $\frac{1}{2\alpha(D)}$ in probability
	since  the definition of $\alpha(D)$ implies
	$$
	\alpha(D)=\sum_k k\P(D_{\U}=k)\alpha(D,k). 
	$$
	It remains to prove the claim. For notational convenience we omit  $\G_n$ from the conditioning. 
	To prove the first part of claim, we fix some $\rho>0$ and replace the quantity $\P(\C(W_{\U}, W_{\tilde{\U}})>4\rel \log n \mid D_{\U}=k )$
	by $\P(\C(W_{\U}, W_{\tilde{\U}})>\tau_{\rho} \mid D_{\U}=k)$, where $\tau_{\rho}$ is the first time that $W_{\U}$ or $W_{\tilde{\U}}$ exits $N_{\rho}(\U)$. For fixed $\rho$, by Lemma \ref{CMLWC}, we see that  $\P(\C(W_{\U}, W_{\tilde{\U}})>\tau_{\rho} \mid D_{\U}=k)$ converges to its corresponding version on $\UGT(D)$ conditioned on the root having degree $k$.  We call this probability $\alpha(D,k,\rho)$. One can show that $\alpha(D,k,\rho)$ converges to $\alpha(D,k)$ as $\rho\to\infty$, using the transience of the simple random walk on a rooted tree whose minimal degree is $\ge 3$. On the other hand, we can repeat the argument in Section \ref{ss:trunc} to show the difference $\P(\C(W_{\U}, W_{\tilde{\U}})>\tau_\U \mid D_{\U}=k)-\P(\C(W_{\U}, W_{\tilde{\U}})>4\rel \log n \mid D_{\U}=k )$ goes to 0 in probability by first sending $n$ to $\infty$ and then $\rho$ to $\infty$. Thus the first part of the claim follows. 
	
	To prove the second part of Claim \ref{claim2}, we condition on $D_{\U}=k$ as well as $W_{\U}$ and $W_{\tilde{\U}}$ don't collide before time  $4\rel \log n$. Define the random time $\tau$ to be the first time that $W_{\U}$ coalesced with $W_{\tilde{\U}}$ in a modified CRW model where the meeting of $W_{\U}$ and $W_{\tilde{\U}}$ in the time interval $ [4\rel\log n, 8\rel \log n]$ is ignored. Clearly we can couple $\tau$ and $\C(W_{\U}, W_{\tilde{\U}})$ together s.t. $\tau\geq \C(W_{\U}, W_{\tilde{\U}})$ almost surely and on the event $\{\C(W_{\U}, W_{\tilde{\U}})> 8\rel \log n\}$ we have $\C(W_{\U}, W_{\tilde{\U}})=\tau$. Note that in the modified model, conditioned on that $W_{\U}$ doesn't collide with $W_{\tilde{\U}}$ before time $4\rel \log n$, the $\ell_{\infty,\pi^{\otimes 2}}$ distance between the joint distribution of $W_{\U}$ and $W_{\tilde{\U}}$ at time $8\rel \log n$ and $\pi^{\otimes 2}$ is smaller than $1/n^2$ by Corollary \ref{linftyctl} applied to the product chain. This implies
	\begin{equation}\label{difftaumeet}
		\abs{\E(\tau \mid \C(W_{\U},W_{\tilde{\U}})>4\rel \log n,D_{\U}=k)-\M}\leq \M/n^2+8\rel \log n.
	\end{equation}
	Using the coupling between $\tau$ and $\C(W_{\U}, W_{\tilde{\U}})$, we see 
	\begin{equation}\label{cu1}
		\begin{split}
			&\abs{ \E(\C(W_{\U}, W_{\tilde{\U}}) \mid 
				\C(W_{\U}, W_{\tilde{\U}})>4\rel \log n,D_{\U}=k)-\M
			}\\
			\leq & \abs{ \E(\C(W_{\U}, W_{\tilde{\U}}) \mid 
				\C(W_{\U}, W_{\tilde{\U}})>4\rel \log n,D_{\U}=k)-\E(\tau \mid \C(\U,\tilde{\U}>4\rel \log n,D_{\U}=k)}\\ +&\abs{\E(\tau \mid \C(W_{\U}, W_{\tilde{\U}})>4\rel \log n,D_{\U}=k)-\M}\\
			\leq & \E(\C(W_{\U}, W_{\tilde{\U}})\1[\C(W_{\U}, W_{\tilde{\U}})\leq 8\rel \log n]|
			\C(W_{\U}, W_{\tilde{\U}})>4\rel \log n,D_{\U}=k)\\
			+&\E(\tau \1[ \C(W_{\U}, W_{\tilde{\U}}) \leq 8\rel \log n]|
			\C(W_{\U}, W_{\tilde{\U}})>4\rel \log n,D_{\U}=k)+\M/n^2+8\rel \log n\\
			\leq & \E(\tau \1[ \C(W_{\U}, W_{\tilde{\U}}) \leq 8\rel \log n]|
			\C(W_{\U}, W_{\tilde{\U}}))>4\rel \log n,D_U=k)+16 \rel \log n +\M/n^2. 
		\end{split}
	\end{equation}
	To control the term $\E(\tau \1[\C(\U,\tilde{\U}) \leq 8\rel \log n]|\C(W_{\U}, W_{\tilde{\U}})>4\rel \log n,D_{\U}=k)$, we consider another modified CRW model where we ignore collisions of $
	\U$ and $\tilde{
		\U}$ in the time interval $[4\rel\log n, 12\rel \log n]$. Denote the  collision time in this new model by $\tau^*$. It follows that $\tau\leq \tau^*$ almost surely. Similar to \eqref{difftaumeet}, we have
	\[
	\abs{\E(\tau^* \mid 4\rel \log n <\C(W_{\U}, W_{\tilde{\U}})\leq 8 \rel \log n, D_\U=k)-\M} \leq \M/n^2+12\rel \log n.
	\]
	It follows that 
	\begin{multline}
		\label{cu2}
		\E(\tau \1[ \C(W_{\U}, W_{\tilde{\U}}) \leq 8\rel \log n]|
		\C(W_{\U}, W_{\tilde{\U}})>4\rel \log n,D_{\U}=k)\\
		\leq 
		\left (12\rel \log n+(1+1/n^2)\M \right)
		\P(\C(W_{\U}, W_{\tilde{\U}})\leq 8\rel \log n|
		\C(W_{\U}, W_{\tilde{\U}})\geq 4\rel \log n,D_{\U}=k)
	\end{multline}
	The first part of Claim \ref{claim2} implies that the quenched probability 
	\[
	\P(\C(W_{\U}, W_{\tilde{\U}})\geq 4\rel \log n,D_{\U}=k)
	\]
	is bounded from below with high probability. This together with Lemma 
	\ref{meetprob}
	implies that with high probability we have 
	\begin{equation}\label{cu3}
		\P(\C(W_{\U}, W_{\tilde{\U}})\leq 8\rel \log n|
		\C(W_{\U}, W_{\tilde{\U}})\geq 4\rel \log n,D_{\U}=k)\leq C\rel \log n \left(n^{-4}+n^{-1}\right),
	\end{equation}
	where $C$ is some constant depending on $\bar d$. 
	Finally, we recall that $\rel$ is bounded from above by some constant with high probability (Lemma \ref{l:configroughbound}). Hence the second part of Claim \ref{claim2} follows from  \eqref{cu1}, \eqref{cu2} and \eqref{cu3}. This completes the proof of Claim \ref{claim2} and hence also Lemma \ref{tmmetlim}. 
\end{proof}

\subsection{Proof of Theorem \ref{t:config} and \ref{t:gwtree}}
We now  prove Theorem \ref{t:config}, using Proposition \ref{p:gengraph} and the bounds on the error terms we obtained in previous sections.

\begin{proof}[Proof of Theorems \ref{t:config}]
	Suppose the degree distribution $D$ is upper bounded by $\bar{d} \in \ZZ_+$.
	Let $\kappa_0$ be defined as in Lemma \ref{l:configroughbound}, and $\delta_D$ be defined as in Lemma \ref{lemma:delta-mix}. For any $\varepsilon > 0$, let $\delta'$ be some number to be determined.
	Take $\rho=\rho(\delta_D, \kappa_0, (2\alpha(D))^{-1} + \delta')$ as given by Proposition \ref{p:control_ek}.
	
	For any $n \in \ZZ_+$, we call a graph $\G$  \emph{good}, if
	\begin{enumerate}
		\item $|\M/n - (2\alpha(D))^{-1}| \leq \delta'$.
		\item $\kappa(\G) \geq \kappa_0$.
		\item $\G$ is $\delta_D$-homogeneous.
		\item $\G$ is $\rho$-locally tree like.
	\end{enumerate}
	Note that when $\G\sim \CM_n(D)$, by
	taking $n$ large enough, we have $\P(\G\textrm{ is good}) > 1 - \ep$, by  Theorem \ref{t:upperbound}, Lemmas \ref{lemma:delta-mix}, \ref{l:localtreelike}, and \ref{tmmetlim}. Now we take a good graph $\G$ on $n$ vertices. 
	By Lemma \ref{l:boundH}, we have  
	\begin{equation}\label{boundH2}
		\frac{\max_z\int_0^s  p_{s'}(z,z)\mathrm{d} s'}{\min_z \int_0^s  p_{s'}(z,z)\mathrm{d} s'} + \max_z \int_0^s p_{s'}(z,z)\mathrm{d} s'\leq  C_{\kappa_0, \bar{d}}
	\end{equation}
	for any $s \leq n$.

	Now we verify the conditions in Proposition \ref{p:gengraph}.
	We first check the second condition. 
	Since we assume $\lim_{n\to\infty}n/t_n=\infty$, we have $t_n\leq n/2$ for $n$ large enough.
	Since $\G$ is a good graph, by Theorem \ref{t:upperbound},
	\[
	\sup_{x\in V, 0\leq s \leq t_n } P_s(x)\frac{ns}{\M}\leq C_0\frac{n}{\M}\sup_{x\in V} \int_0^{t_n} p_{t'}(x,x)\mathrm{d} t'
	\leq 
	4 C_0  C_{\kappa_0,\d} \alpha(D),
	\]
	if $\delta'\leq 1/(4\alpha(D))$. Here $C_0$ is the  constant given in Theorem \ref{t:upperbound}. Hence if we take $K'$ to be $4C_0C_{\kappa_0,\d}\alpha(D)$, the second condition 
	of Proposition \ref{p:gengraph}
	is satisfied. Given $\ep$ and $K'$, we let $K$ and $\delta$ be given as in the statement of Proposition \ref{p:gengraph} (see also Remark \ref{variant}). We then set $ \delta'=\min\{\delta, 1/\alpha(D)\}/4$ and set the value of $\Delta$ to be
	$$
	\Delta:=K\log\left(\frac{\M r_{\max}}{n}+1\right)+ \log R+ \frac{1}{\delta'},
	$$
	where $R=r_{\max}/r_{\min}$. 
	We now check the first condition.
	Since $\G$ is a good graph and $R \leq \bar{d}$, we have \[
	\Delta \leq K\log\left(((2\alpha(D))^{-1} + \delta')\bar{d} + 1\right) + \log(\bar{d}) + (\delta')^{-1}.
	\]
	It follows from the definitions of $E_1, E_2$ and $E_4$ as well as \eqref{boundH2} that
	\[
	E_1 \leq \frac{1}{t_n} \frac{2\bar d}{\kappa_D^2}\left(((2\alpha(D))^{-1} + \delta')\bar{d}\right)^K(2 \left(\log( C_{\kappa_0, \bar{d}}) + K\log\left(((2\alpha(D))^{-1} + \delta')\bar{d} + 1\right) + \log(\bar{d}) + (\delta')^{-1}\right),
	\]
	\[
	E_2 \leq \frac{t_n}{n}\left(((2\alpha(D))^{-1} + \delta')\bar{d}\right)^K \bar{d}^2,
	\]
	and
	\[
	E_4 \leq \frac{1}{t_n}\left(((2\alpha(D))^{-1} + \delta')\bar{d}\right)^K.
	\]
	For $E_3$, we  apply Proposition \ref{p:control_ek} with $\delta_1=\delta_D$ and $C_{\alpha} = (2\alpha(D))^{-1} + \delta'$. Since $\lim_{n\to\infty}t_n=\lim_{n\to\infty}n/t_n=\infty$, we have $E_1+E_2+E_3+E_4 < \delta$ for all $n$ large enough.
	
	Hence by Proposition \ref{p:gengraph} we have
	$
	|t_n P_{t_n} - \frac{1}{\alpha(D)}| < \varepsilon
	$, for any $n$ large enough whenever $\G$ is a good graph. Thus we have
	$$
	\liminf_{n\to\infty} \P\left(\{|t_n P_{t_n} - \frac{1}{\alpha(D)}| < \varepsilon\}\right) \geq \liminf_{n\to\infty}\P(\G_n \mbox{ is good})\geq 1-\ep.
	$$
	Our conclusion follows since $\ep$ is arbitrary. 
\end{proof}

We deduce Theorem \ref{t:gwtree} from Theorem \ref{t:config} and the local weak convergence of the configuration model.

\begin{proof}[Proof of Theorem \ref{t:gwtree}]
	Fix $t \in \RR_{\ge 0}$, let $\G_n\sim \CM_n(D)$ and condition on $\G_n$ being connected.  By Lemma \ref{CMLWC} and Proposition \ref{local1} we see $P_t(\G_n)$  
	converges to $\E_{\UGT(D)}P_t(o)$, where the subscript indicates the expectation is taken over the law of $\UGT(D)$.
	
	Take any increasing and diverging sequence of times $t_1,t_2,\cdots$. By repeating each entry many times, we can construct another sequence of times $t_1', t_2',\cdots$ containing the previous sequence $t_1,t_2,\ldots$ as well as a sequence $\ep_n\to 0$; and for each large enough $n$, $|t_n'P_{t_n'}(\G_n) - t_n'\E_{\mathbb{UST}(D)}P_{t_n'}(o)| < \ep_n$ with probability at least $1-\ep_n$,
	and $\lim_{n\to \infty} n/t_n' = \infty$.
	By Theorem \ref{t:config}, $t_n' P_{t_n'}(\G_n)$ converges to $\frac{1}{\alpha(D)}$ in probability as $n\to \infty$.
	Thus $t_n'\E_{\UGT(D)}P_{t_n'}(o)$ converges to $\frac{1}{\alpha(D)}$, as $n\to \infty$. Since $t_n$ is a subsequence of $t_n'$, we have that $t_n\E_{\UGT(D)}P_{t_n}(o)$ converges to $\frac{1}{\alpha(D)}$ as well. 
\end{proof}
\begin{remark}
	We have to repeat each entry many times because only in this way can we guarantee
	$|t_n'P_{t_n'}(\G_n) - t_n'\E_{\mathbb{UST}(D)}P_{t_n'}(o)|$  is small with high probability, since now we can treat $t_n'$ as almost being a constant and let $n$ be as large as needed. 
\end{remark}

\section{Mean field behavior for transitive Markov chains}\label{s:finitetrans}
In this section we will prove Theorems \ref{transgraph}, \ref{uniformtran} and \ref{infinitetrans}. Recall that for transitive chains we always assume $r(x)=\sum_y r_{x,y}=1$. We also have $R=r_{\max}/r_{\min}=1$. We will need to refine the estimate of the integral of $h(t,\phi_k)$. 

Recall that we let $\U\sim\pi$ and $\tilde \U\sim \nu_{\U}$, and that $W_\U$ and $W_{\tilde \U}$ are two random walks starting from $\U$ and $\tilde \U$ respectively.
The key observation is that due to transitivity, the conditional probability
\begin{equation}\label{const}
	\P(W_{\U}(s)\neq W_{\tilde{\U}}(s),0\leq s\leq t \mid \U)
\end{equation}
is a constant independent of $\U$. Hereafter we denote this quantity by $q(t)$, which is also equal to $\frac{n}{2}f(t)$ where $n$ is the size of the chain and  $f(t)$ is the probability density function of the meeting time of two independent random walks starting from $\pi^{\otimes 2}$ 
(see equation \eqref{eq:densityfortwo}). 
Recall the definition of branching structures in Section \ref{ss:branchstruc} and the notation $\t=(t_1, \ldots, t_k)$ where $0<t_1<\cdots<t_k<t$. For convenience we let $t_0=0$ and $t_{k+1}=t$. 
From the definition of $q(t)$ we get
\begin{equation}\label{q}
	\P\left( 
	\forall t'\in [t_k,t_{k+1}],\gamma_{k}(t')\neq \gamma_{i_{k}}(t')
	|\gamma_{\ell}(t_k),1\leq \ell\leq k-1\right)
	=q(t_{k+1}-t_k).
\end{equation} 
Hence $g(\t, I_k)$ equals to
\begin{equation}\label{gtiktrans}
	\begin{split}
		&\E\left(
		\left(\prod_{\ell=1}^{k-1}   \1[\forall t'\in[ t_{\ell},t_{\ell+1}],\gamma_{\ell}(t')\neq \gamma_{i_{\ell}}(t')]\right)
		\P\left( 
		\forall t'\in [t_k,t_{k+1}], \gamma_{k}(t')\neq \gamma_{i_{k}}(t')
		|\gamma_{\ell}(t_k),1\leq \ell\leq k-1\right)
		\right)\\
		=& q(t_{k+1}-t_k) \E
		\left(\prod_{\ell=1}^{k-1}   \1[\forall t'\in[ t_{\ell},t_{\ell+1}],\gamma_{\ell}(t')\neq \gamma_{i_{\ell}}(t')]
		\right)\\
		=&q(t_{k+1}-t_k) q(t_k-t_{k-1})\E
		\left(\prod_{\ell=1}^{k-2}   \1[\forall t'\in[ t_{\ell},t_{\ell+1}],\gamma_{\ell}(t')\neq \gamma_{i_{\ell}}(t')]
		\right)\\
		=&\prod_{\ell=1}^k q(t_{\ell+1}-t_{\ell})=q(\t,I_k).
	\end{split}
\end{equation}
Recall the definition of $\alpha_t(x)$ and $\alpha_t$ in \eqref{alpha_t(x)}. 
In the transitive setup we have
$\alpha_{t}=\P_{x,\nu_x}(\RM\geq t)$  since $r(x)=1, \forall x$.
We first give a result that relates $\M$ to $\P_{x,\nu_x}(\RM\geq t)$, which is  reminiscent of Lemma \ref{tmmetlim} in  Section \ref{ssc:tmeetnconfig}. As a corollary, the two statements in Theorem \ref{transgraph} are equivalent. 
\begin{prop}\label{tmeetandn} 
	For any sequence of transitive Markov chains $(V_n,\r_n)$ satisfying  \[\lim_{n\to\infty}\rel(\r_n)/\M(\r_n)=0,\]
	and for any sequence $t_n$ satisfying $\rel(\r_n)\ll t_n \ll \M(\r_n)$, we have
	\[
	\lim_{n\to\infty}\frac{2\M(\r_n)\alpha_{t_n}}{n}=1.
	\]
\end{prop}
\begin{proof}[Proof of Proposition \ref{tmeetandn}]
	Proposition \ref{tmeetandn} follows directly from equation \eqref{kac1} and 
	equation \eqref{ABdensity}. Indeed, as we mentioned in the beginning of Section \ref{s:finitetrans}, by transitivity, for all $x\in V$,
	\[
	\P_{x,\nu_x}(\RM>t)=\frac{n}{2}f(t).
	\]
	Let $\rel(\r^2)$ be the relaxation time for the product chain of two independent copies of $(V,\r)$. The definition of relaxation time implies that $\rel(\r^2)= \rel(\r)$. Using \eqref{ABdensity} we get
	\begin{equation}\label{taumeet>t}
		\frac{n}{2\M} \left(1-\frac{2\rel(\r)+t}{\M} \right)\leq
		\P_{x,\nu_x}(\RM>t) \leq \frac{n}{2\M}\left(1+\frac{\rel(\r)}{2t}\right).
	\end{equation}
	where we have used the definition  $\M=\E_{\pi,\pi}(\RM)$. For a sequence of Markov chains $(V_n,\r_n)$ such that $\rel(\r_n)\ll t_n \ll \M(\r_n)$, we have
	\[
	\lim_{n\to\infty} \left( 1-\frac{2\rel(\r_n)+t_n}{\M(\r_n)}\right)=\lim_{n\to\infty} \left(
	1+\frac{\rel(\r_n)}{2t_n}\right)=1.
	\]
	Combining  this with equation \eqref{taumeet>t} we get 
	\[
	\lim_{n\to\infty} \frac{2\P_{x,\nu_x}(\RM\geq t_n)\M(\r_n) }{n}=1,
	\]
	as desired. 
\end{proof}

\subsection{General transitive Markov chains}\label{pftrans}
Let $(V_n,\r_n)$ be a sequence of transitive chains 
such that 
$\M(\r_n)/n$ is bounded from above and $\rel(\r_n)\ll \M(\r_n)$.
Let $t_n$ be any sequence of times satisfying 
$\rel(\r_n)\ll t_n\ll \M(\r_n)$. We will show that the conditions in Proposition \ref{p:gengraph} hold for any $\ep>0$ when $n$ is sufficiently large.  

Recall the definition of $H(t)$ and $G(t,\Delta)$ in equation \eqref{defofH} and \eqref{G(t)} respectively. Since the graph is transitive, $H(t)=1$ and $G(t,\Delta)=\rel\Delta$. Also recall that $\t$ is $\Delta$-good if $\Delta_{\min}(\t)> \rel \Delta$. Note by  fact \eqref{relandrmax} we have $\rel(\r_n)\geq 1/2$ (since $r_{\max}=1$) so
$t_n\gg 1$.
Theorem \ref{t:upperbound} and equation \eqref{integralofp} 
together imply that there exists a universal constant $K'$ such that
\[
\sup_{x\in V_n} P_{s}(x)\frac{ns}{\M(\r_n)}
\leq K' \mbox{ for all } 0\leq s\leq t_n.
\]
Thus the second condition of Proposition \ref{p:gengraph} holds. To verify the first condition, we fix any $\ep>0$, let $K>0$ 
and $\delta$ be given by the statement of Proposition \ref{p:gengraph}.  By transitivity  we have $e_k(t)=0$ for all $k$ and $t$, which gives $E_3=0$. 
Since $\M(\r_n)/n$ is
bounded above, $\Delta$ (defined in equation \eqref{defDelta}) is also a bounded number as $n\to\infty$. 
Hence for $n$
large enough we have
\[
E_1=\left(\frac{\M(\r_n)}{n}\right)^K \Delta \frac{\rel(\r_n)}{t_n} \leq \frac{\delta}{3},
\]
and 
\[
E_4=\left(\frac{\M(\r_n)}{n}\right)^K \frac{1}{t_n}\leq \frac{\delta}{3},
\]
since $t_n\gg \rel(\r_n)$, $t_n\gg 1$ and  $\M(\r_n)/n$ stays bounded as $n\to\infty$. For $E_2$, we have
\[
E_2=\left(\frac{\M(\r_n)}{n}\right)^K \frac{t_n}{n} \leq \frac{\delta}{3},
\]
since $t_n\ll \M(\r_n)$.
Thus for $n$ large we have
\[
E_1+E_2+E_3+E_4\leq \frac{\delta}{3}+\frac{\delta}{3}+0+\frac{\delta}{3}\leq \delta.
\]
Thus both conditions of Proposition \ref{p:gengraph} hold and consequently, for any $\ep>0$, when $n$ is large enough we have
\begin{equation}
	\abs{P_{t_n}-\frac{2\M(\r_n)}{nt_n}}\leq \ep \frac{\M(\r_n)}{nt_n}.
\end{equation}
This implies 
\[
\lim_{n\to\infty}  \abs{ \frac{nt_n}{2\M(\r_n)} P_{t_n}-1}=0.
\]
By Proposition \ref{tmeetandn}, we also get
\[
\lim_{n\to\infty}\abs{t_n \alpha_{t_n}P_{t_n}-1}=0. 
\]
This completes the proof of Theorem \ref{transgraph}. 

\subsection{Uniformly transient transitive Markov chains}
For uniformly transient Markov chains we have $\rel(\r_n)\ll \M(\r_n)$ and $\M(\r_n) \asymp n$ (we include a proof of this fact in Proposition \ref{p:treluniform}). If $\rel(\r_n)$ is bounded when $n$ goes to infinity, then Theorem \ref{transgraph} already implies Theorem \ref{uniformtran}. Thus we only need to consider the case when $\rel(\r_n)$ goes to infinity. As in the proof of Theorem \ref{transgraph} we will establish a proposition giving conditions under which mean-field behavior occurs  for transitive chains. 
To this end we 
first upper bound $g(\t,I_k)-h(\t,I_k)$ (see Lemma \ref{l:goodset:tran} below). Then we control the integral of $h(\t,I_k)$ using estimates for $g(\t,I_k)$.

In this section the `good' set is simply defined to be $\{\t:\Delta_{\min}(\t)>\Delta\}$ (while previously we used $\{\t:\Delta_{\min}(\t)>G(t,\Delta)=\rel \Delta\}$ as our $\Delta$-good set). Recall the definition of the branching structure in Section \ref{ss:branchstruc}.
We now state a lemma that controls the meeting probability, from which we will deduce Lemma \ref{l:goodset:tran}.  

\begin{lemma}\label{meetuniformtrans}
	There exists an absolute constant $C$ such that the following two statements hold for all transitive chains $(V,\r)$ with $\abs{V}=n$.
	For any $1<\Delta<\rel/4$ and $t>\Delta$ we have 
	\begin{equation}\label{meetw}
		\begin{split}
			\max_{x,y\in V}\P_{x,y}\left(\exists t'\in [\Delta,t] \text{ such that } W_x(t')=W_y(t')\right) \leq C \left(\int_{2\Delta}^{\rel} p_{s}(z,z)\mathrm{d} s+\frac{t}{n}\right),
		\end{split}
	\end{equation}
	where $W_x,W_y$ are two independent random walks starting from $x$ and $y$ and $z$ is any element of $V$. 
	Moreover, suppose $\ell_1<i_{\ell_2}$, then  we  have
	\begin{equation}\label{meetgamma}
		\begin{split}
			&\max_{x,y}\P(\exists t'\in [t_{\ell_2},t_{\ell_2+1}] \text{ such that } 
			\gamma_{\ell_1}(t')=\gamma_{\ell_2}(t')
			|\gamma_{\ell_1}(t_{\ell_2-1})=x,\gamma_{i_{\ell_2}}(t_{\ell_2-1})=y)\\
			\leq &C \left(\int_{2\Delta}^{\rel} p_{s}(x,x)\mathrm{d} s+\frac{t}{n}\right). 
		\end{split}
	\end{equation}
\end{lemma}

\begin{proof}
	We start with the first statement. Note that $r(x)=1$ for all $x$. Define $N(x,y,t)$ to be the Lebesgue measure of the set
	$$
	\{s\in [\Delta, t+1]: W_x(s)=W_y(s)\}.
	$$
	Since $r(x)=1$ for all $x$, we have that
	$$
	\E(N(x,y,t)|W_x(s)=W_y(s) \mbox{ for some }s\in [\Delta, t])\geq \exp(-2).
	$$
	This implies that 
	\begin{equation*}
		\begin{split}
			& \P_{x,y}(\exists t'\in [\Delta,t]\text{ such that } W_x(t')=W_y(t')) \\
			\leq &\frac{\E(N(x,y,t))}{\E(N(x,y,t)|W_x(s)=W_y(s) \mbox{ for some }s\in [\Delta, t])}\\
			\leq &\exp(2)\E(N(x,y,t))=\exp(2) \int_{\Delta}^{t+1} \left(\sum_{z'}  p_s(x,z')p_s(y,z')\right)\mathrm{d} s\\
			=& \exp(2) \int_{\Delta}^{t+1} \left(\sum_{z'} p_s(x,z')p_s(z',y)\right)\mathrm{d} s=\exp(2)
			\int_{\Delta}^{t+1}p_{2s}(x,y)\mathrm{d}s.
		\end{split}
	\end{equation*}
	It follows that
	\begin{equation}
		\begin{split}
			&    \max_{x,y\in V}\P_{x,y}(\exists t'\in [\Delta,t]\text{ such that } W_x(t')=W_y(t'))\\ 
			= &C  \max_{x,y\in V}\int_{\Delta}^{t+1} p_{2s}(x,y)\mathrm{d} s
			\leq  C \left( \max_{x,y\in V}\int_{\Delta}^{t+1} \left(p_{2s}(x,y)-\frac{1}{n}\right)\mathrm{d} s+\frac{t+1}{n}\right)\\
			\leq &C\left( \sum_{\ell=1}^{\lceil 2(t+1)/(\rel-2\Delta) \rceil}\max_{x,y\in V}\int_{\Delta+(\ell-1) (\rel/2-\Delta)}^{\Delta+\ell (\rel/2-\Delta)} \left(p_{2s}(x,y)-\frac{1}{n}\right)\mathrm{d} s+\frac{t+1}{n}\right).
		\end{split}
	\end{equation}
	For transitive graphs, $p_s(x,x)$ is independent of $x$. Thus using equation \eqref{poincare} and $\rel>4\Delta$, we get
	\begin{equation}
		\begin{split}
			\int_{\Delta+(\ell-1) (\rel/2-\Delta)}^{\Delta+\ell (\rel/2-\Delta)} \left(p_{2s}(x,y)-\frac{1}{n}\right)\mathrm{d}s \leq &
			\exp\left(-(\ell-1) \frac{\rel-2\Delta}{\rel}\right)\int_{\Delta}^{\rel/2} p_{2s}(x,x)\mathrm{d} s\\
			\leq& \exp\left(-\frac{\ell-1}{2}\right)\int_{\Delta}^{\rel/2} p_{2s}(x,x)\mathrm{d} s.
		\end{split}
	\end{equation}
	Summing $\ell$ from 1 to infinity we have
	\[
	\max_{x,y\in V}\P_{x,y}(\exists t'\in [\Delta,t] \text{ such that } W_x(t')=W_y(t')) \leq C \left(\int_{2\Delta}^{\rel} p_{s}(x,x)\mathrm{d} s+\frac{t}{n}\right),
	\]
	which proves the first assertion of Lemma \ref{meetuniformtrans} (note that we assume $t>1$ so that $t+1\leq 2t$). For the second assertion, 
	using equation \eqref{bdhatp} and the definition of the branching structure, the first line in equation \eqref{meetgamma} is bounded from above by 
	$$C \int_{\Delta}^{t+1} \sum_{z'}p_s(x,z')\hat{p}_s(y,z')\mathrm{d} s$$ where we
	recall that $\hat{p}_{\cdot}$ is a transition probability related to $p_{\cdot}$ defined in equation \eqref{defofhatp}. Using equation \eqref{bdhatp} we have
	\begin{equation}
		C \int_{\Delta}^t \sum_{z'}p_s(x,z')\hat{p}_s(y,z')\mathrm{d} s \leq 
		C' \int_{\Delta}^t \sum_{z'}p_s(x,z')p_{s+1}(y,z')\mathrm{d} s=C' \int_{2\Delta}^{t+1} p_{2s+1}(x,y)\mathrm{d} s,
	\end{equation}
	which can be bounded by $$
	C'' \left(\int_{2\Delta}^{\rel} p_{s}(x,x)\mathrm{d} s+\frac{t}{n}\right),
	$$
	by repeating the proof of first assertion of Lemma \ref{meetuniformtrans}. This completes the proof of Lemma \ref{meetuniformtrans}. 
\end{proof}

Using this lemma we can give a refined upper bound for $g(\t,\phi_k)-h(\t,\phi_k)$ for transitive chains.
\begin{lemma}\label{l:goodset:tran}
	Consider coalescing random walks on a transitive Markov chain. There exists some constant $C_k$ depending on $k$ s.t.
	for any $\t\in \R^k_{<, t}$ and $\Delta>1$ that satisfy $\Delta_{\min}(\t)>\rel>4\Delta$ we have
	\begin{equation}
		0\leq g(\t,I_k)-h(\t,I_k)\leq
		C_k \left(\sum_{i=2}^{k}
		\prod_{\ell=1,\ell \neq i-1,i}^{k} 
		q(t_{\ell+1}-t_{\ell})
		\right)  \left(\int_{2\Delta}^{\rel} p_{s}(x,x)\mathrm{d} s+\frac{t}{n}\right).
	\end{equation}
\end{lemma}

\begin{proof}
	Using equation \eqref{g-h} we get
	\begin{equation}
		\begin{split}
			0\leq& g(\t,I_k)-h(\t,I_k)\\
			\leq &
			\sum_{0\leq \ell_1 < \ell_2 \leq \ell_3 \leq k}
			\E\Big(\1[\{\forall 1 \leq \ell \leq k \text{ and } t' \in [t_\ell, t_{\ell+1}] \text{ we have }
			\gamma_{\ell}(t')\neq \gamma_{i_\ell}(t')\} 
			\\ & \cap \{
			\exists  t''\in [t_{\ell_3}, t_{\ell_3+1}] \text{ such that } \gamma_{\ell_1}(t'')=\gamma_{\ell_2}(t'')\}] \vphantom{\frac12}  \\
			&\times  \left(\1[\ell_1= i_{\ell_2}]+\1[\ell_1\neq i_{\ell_2} \text{ and } \forall t''' \in [t_{\ell_2-1},t_{\ell_2}] \text{ we have } \gamma_{\ell_1}(t''')\neq 
			\gamma_{i_{\ell_2}}(t''')]\right) \vphantom{\frac12} \Big).
		\end{split}
	\end{equation}
	Using equation \eqref{const} and 
	proceed in a similar fashion to \eqref{gtiktrans}, we see the summand on the right hand side is equal to (recall the equation of $q(t)$ in \eqref{q})
	\begin{equation}\label{pro1} 
		\begin{split}
			& \prod_{\ell=\ell_3+1}^k q(t_{\ell+1}-t_{\ell}) \E\Big( \1[\{\forall 1 \leq \ell \leq \ell_3 \mbox{ and } t' \in [t_\ell, t_{\ell+1}], \mbox{ we have }
			\gamma_{\ell}(t')\neq \gamma_{i_\ell}(t')\}\\
			& \cap \{
			\exists t''\in [t_{\ell_3}, t_{\ell_3+1}] \mbox{ such that } \gamma_{\ell_1}(t'')=\gamma_{\ell_2}(t'')\}] \vphantom{\frac12}   \\
			&\times \left(\1[\ell_1= i_{\ell_2}]+\1[\ell_1\neq i_{\ell_2} \mbox{ and } \forall t''' \in [t_{\ell_2-1},t_{\ell_2}], \mbox{ we have } \gamma_{\ell_1}(t''')\neq 
			\gamma_{i_{\ell_2}}(t''')] \right) \vphantom{\frac12} \Big).
		\end{split}
	\end{equation}
	
	Similar to the  proof of Lemma \ref{l:approxhbyg}, we consider two cases $\ell_3=\ell_2$ and $\ell_3>\ell_2$ separately. In the first case we must have $\ell_1\neq i_{\ell_2}$. Using the second statement in  Lemma \ref{meetuniformtrans} we have the following bound for the conditional probability
	\begin{equation*}
		\begin{split}
			&\P\left(\forall  t' \in [t_{\ell_3-1}, t_{\ell_3}],
			\gamma_{\ell_1}(t')\neq \gamma_{i_{\ell_2}}(t'),
			\exists t''\in [t_{\ell_3}, t_{\ell_3+1}], \gamma_{\ell_1}(t'')=\gamma_{\ell_2}(t'')\right. \\
			&\left. |\gamma_{\ell}(t''), 0\leq \ell \leq \ell_3-1,t_{\ell}\leq t''\leq t_{\ell_3-1} \right)
			\leq C \left(\int_{2\Delta}^{\rel} p_{s}(x,x)\mathrm{d} s+\frac{t}{n}\right). 
		\end{split}
	\end{equation*}
	Therefore, we can upper bound the quantity in \eqref{pro1} by 
	\begin{equation*}
		\begin{split}
			& C\E\left(
			\1[\forall 1 \leq \ell \leq \ell_3-1, t' \in [t_\ell, t_{\ell+1}],
			\gamma_{\ell}(t')\neq \gamma_{i_\ell}(t')
			\right)\left(\int_{2\Delta}^{\rel} p_{s}(x,x)\mathrm{d} s+\frac{t}{n}\right)\prod_{\ell=\ell_3+1}^k q(t_{\ell+1}-t_{\ell})\\
			=&C \left(\int_{2\Delta}^{\rel} p_{s}(x,x)\mathrm{d} s+\frac{t}{n}\right) \prod_{\ell=1}^{\ell_3-2}
			q(t_{\ell+1}-t_{\ell}) \prod_{\ell=\ell_3+1}^k q(t_{\ell+1}-t_{\ell}).
		\end{split}
	\end{equation*}
	In the second case $\ell_3>\ell_2$, we can also get an upper bound of
	$$C \left(\int_{2\Delta}^{\rel} p_{s}(x,x)\mathrm{d} s+\frac{t}{n}\right) \prod_{\ell=1,\ell \neq \ell_3,\ell_3-1}^{k}
	q(t_{\ell+1}-t_{\ell})$$ for a possibly different $C$.
	Hence Lemma \ref{l:goodset:tran} follows by summing over all possible $\ell_1,\ell_2,\ell_3$.
\end{proof}

Now we estimate the integral of $q(\t,I_k)$. Fix any $\Delta\in (1,\rel/4)$.  We split the domain of the integral into $\{\Delta_{\min}(\t)>\Delta\}$ and 
$\{\Delta_{\min}(\t)\leq \Delta\}$. 
On the set $\{\Delta_{\min}(\t)\leq \Delta\}$ we simply use the bound $q(\t,I_k)\leq 1$ since $r(x)=1$ for all $x$. For the set 
$\{\Delta_{\min}(\t)>\Delta\}$, Lemma \ref{l:goodset:tran}  implies that
\begin{equation}\label{g-huniftrans}
	0\leq \int_{\Delta_{\min}(\t)>\Delta}(g(\t,I_k)-h(\t,I_k))\mathrm{d} \t\leq C_k \left(\int_0^t q(s)\mathrm{d} s\right)^{k-2}t^2  \left(\int_{2\Delta}^{\rel} p_{s}(x,x)\mathrm{d} s+\frac{t}{n}\right).
\end{equation}
By definition of $q(t)$, for any $\Delta\leq t_1\leq t_2$, we have
\begin{equation}\label{ctlf1}
	\begin{split}
		0\leq q(t_1)-q(t_2)&\leq  \P(\exists t'\in [t_1,t_2], W_{\U}(t')=W_{\tilde{\U}}(t') )\\
		&\leq  \P(\exists t'\in [\min\{t_1,\rel/4\},t_2], W_{\U}(t')=W_{\tilde{\U}}(t') )\\
		&\leq  C\left(\int_{\min\{2t_1,\rel/2 \}}^{\rel} p_{s}(x,x)\mathrm{d} s+\frac{t_2}{n} \right),
	\end{split}
\end{equation}
where  $(W_{\U},W_{\tilde{\U}})$ is a pair of independent random walks starting from $(\U,\nu_{\U})$ and we have used equation \eqref{meetw} in the last inequality
(with the $\Delta$ there replaced by $\min\{2t_1,\rel/2\}$).
Now using \eqref{eq:densityfortwo} and \eqref{eq:fhattctl},
\begin{equation}\label{ctlf2}
	\frac{n}{2\M}\left(1-\frac{(2+\Delta)\rel}{\M}\right)
	\leq q(\Delta \rel) \leq \frac{n}{2\M}\left(1+\frac{1}{\Delta}\right).
\end{equation} 
Combining equations \eqref{ctlf1} and \eqref{ctlf2}, we see that for all $\Delta\leq s\leq t$,
\begin{equation}\label{ctlf4}
	\abs{q(s)-\frac{n}{2\M}}\leq  C\left(\frac{n}{\M}
	\left(\frac{1}{\Delta}+\frac{(2+\Delta)\rel}{\M}  \right)+\int_{2\Delta}^{\rel} p_s(x,x)\mathrm{d} s+\frac{t+\Delta \rel}{n}
	\right).
\end{equation}
Using \eqref{ctlf4} together with facts  $q(s)\leq 1$ for all $s$
and $n/(2\M)\leq 2$ (by equation \eqref{mrmax/n}), we get 
\begin{equation}\label{ctlf5}
	\begin{split}
		&  \abs{\int_0^t q(s)\mathrm{d} s-\frac{nt}{2\M}}\leq 
		\int_0^{\Delta} q(s)\mathrm{d} s+\frac{\Delta n}{\M}+
		\abs{\int_{\Delta}^t q(s)\mathrm{d} s-\frac{n(t-\Delta)}{2\M}}
		\\
		\leq &3\Delta+ Ct \left(\frac{n}{\M}
		\left(\frac{1}{\Delta}+\frac{(2+\Delta)\rel}{\M}  \right)+\int_{2\Delta}^{\rel} p_s(x,x)\mathrm{d} s+\frac{t+\Delta \rel}{n}
		\right).
	\end{split}
\end{equation}
It follows from \eqref{g-huniftrans} -- \eqref{ctlf5} that
\begin{equation}\label{g-h11}
	\begin{split}
		0\leq &\int_{\Delta_{\min}(\t)>\Delta}(g(\t,I_k)-h(\t,I_k))\mathrm{d} \t
		\leq Ct^2 \left(\int_{2\Delta}^{\rel} p_{s}(x,x)\mathrm{d} s+\frac{t}{n}\right)\\
		&\times \left(\Delta+ t \left(\frac{n}{\M}
		\left(1+\frac{1}{\Delta}+\frac{(2+\Delta)\rel}{\M}  \right)+\int_{2\Delta}^{\rel} p_s(x,x)\mathrm{d} s+\frac{t+\Delta \rel}{n}
		\right)\right)^{k-2}.
	\end{split}
\end{equation}
It remains to control the integral of $g(\t,I_k)(=q(\t,I_k))$. 
Using equation \eqref{ctlf4} together with the facts that $q(t,I_k)\leq 1$ and $n/(2\M)\leq 2$  we have
\begin{equation}\label{q-11}
	\begin{split}
		\abs{ q(\t,I_k)-\left(\frac{n}{2\M}\right)^{k}}=&
		\abs{\prod_{\ell=1}^k f(t_{\ell+1}-t_{\ell}) -\left(\frac{n}{2\M}\right)^k}\\
		\leq &
		C_k  \left(\frac{n}{\M}
		\left(\frac{1}{\Delta}+\frac{(2+\Delta)\rel}{\M}  \right)+\int_{2\Delta}^{\rel} p_s(x,x)\mathrm{d} s+\frac{t+\Delta \rel}{n}
		\right).
	\end{split}
\end{equation}

We can now ensemble above estimates to prove the following proposition, which is an analogue of Proposition \ref{p:gengraph} for transitive chains.
\begin{prop}  \label{p:uniftrans}
	For any $\ep>0$ and $K'>0$, there exists some $K\in \ZZ_+$ and $\delta>0$, such that the following is true for all $\delta' \leq \delta$.
	Take any transitive chain $(V,\r)$, with $n=|V|$. Take 
	\begin{equation}
		\Delta=\left(\frac{\M}{n}+1\right)^K+ \frac{1}{\delta'}.
	\end{equation}
	Define 
	\begin{equation}
		\begin{split}
			E_1&=\frac{\Delta \rel}{\M}\left(\frac{\M}{n}\right)^K, \quad 
			E_2=\left(\frac{\M}{n}\right)^K\int_{2\Delta}^{\rel} p_{s}(x,x)\mathrm{d} s,\\
			E_3&=\frac{t+\Delta \rel}{n} \left(\frac{\M}{n}\right)^K, \quad 
			E_4=\frac{\Delta}{t}  \left(\frac{\M}{n}\right)^K.
		\end{split}
	\end{equation}
	If the following three conditions hold
	\begin{enumerate}
		\item $\Delta<\rel/4$,
		\item  $E_1+E_2+E_3+E_4
		\leq \delta$,
		\item $P_s\frac{ns}{\M}\leq K'$ for $0\leq s\leq t$,
	\end{enumerate}
	then we have 
	\begin{equation}
		\abs{P_t-\frac{2\M}{nt}}\leq \ep \frac{\M}{nt}.
	\end{equation}
\end{prop}
\begin{proof}[Proof of Proposition \ref{p:uniftrans}]
	Using \eqref{g-h11}, \eqref{q-11} and the fact  $n/\M\leq 4$ we have 
	(recall  $\R^k_{<, t}= \{(t_1,\ldots, t_k):0<t_1<\ldots<t_k<t\}$)
	\begin{equation}\label{eq:intest}
		\begin{split}
			&\abs{\sum_{I_k\in \Phi_k}\int_{\R^{k}_{<, t}} h(\t,I_k)\mathrm{d} \t-\left(\frac{nt}{2\M}\right)^k  } \\
			\leq& \abs{\sum_{I_k\in \Phi_k}\int_{\R^{k}_{<, t}} q(\t,I_k)\mathrm{d} \t-\left(\frac{nt}{2\M}\right)^k  }+\abs{\sum_{I_k\in \Phi_k}\int_{\R^{k}_{<, t}} (h(\t,I_k)-q(\t,I_k))\mathrm{d}\t
			}  \\
			\leq  &   \int_{\Delta_{\min}(\t)>\Delta}
			\sum_{I_k\in \Phi_k}\abs{q(\t,I_k)-\left(\frac{n}{2\M}\right)^k  }\mathrm{d} \t  + \int_{\Delta_{\min}(\t)\leq \Delta} \sum_{I_k\in \Phi_k}q(\t,I_k)\mathrm{d} \t \\
			&+ k!    \left(\frac{n}{2\M}\right)^k \Delta t^{k-1}+
			\int_{\Delta_{\min}(\t)>\Delta}
			\sum_{I_k\in \Phi_k} \abs{ h(\t,I_k)-q(\t,I_k)  }\mathrm{d} \t \\    
			\leq & C_k \left[ t^k \left(\frac{n}{\M}
			\left(\frac{1}{\Delta}+\frac{(2+\Delta)\rel}{\M}  \right)+\int_{2\Delta}^{\rel} p_s(x,x)\mathrm{d} s+\frac{t+\Delta \rel}{n}
			\right)+ \Delta t^{k-1}
			\right.\\
			&\left. +\left(\Delta+ t \left(\frac{n}{\M}
			\left(1+\frac{1}{\Delta}+\frac{(2+\Delta)\rel}{\M}  \right)+\int_{2\Delta}^{\rel} p_s(x,x)\mathrm{d} s+\frac{t+\Delta \rel}{n} \right)  \right)^{k-2} \right.\\
			& \times \left. t^2 \left(\int_{2\Delta}^{\rel} p_{s}(x,x)\mathrm{d} s+\frac{t}{n}\right)
			\right].
		\end{split}
	\end{equation}
	Using this one can get a bound for $\abs{\E(N_t^k)-(k+1)!(\frac{nt}{\M})^{k}}$ (similar to the proof of Lemma \ref{l:ntk-est}). The proof of  Proposition \ref{p:uniftrans} follows from the same argument as the proof of Proposition \ref{p:gengraph} via Lemma \ref{l:ntk-est}. We omit the details here.
\end{proof}

\begin{proof}[Proof of Theorem \ref{uniformtran}]
	As before we will show that the conditions in Theorem \ref{uniformtran} imply conditions of Proposition \ref{p:uniftrans} for any fixed $\ep$ and large $n$. Let $\delta$ be given by Proposition \ref{p:uniftrans}.  Recall that we only need to deal with the case where $\rel(\r_n)\to\infty$ as $n\to\infty$. The third condition in Proposition \ref{p:uniftrans} is satisfied, as we showed in proof of Theorem \ref{transgraph}. The first condition also holds for large $n$ since $\M(\r_n)/n$ remains upper bounded (by uniform transience)  and $\rel(\r_n)\to\infty$. For the second condition, we note
	\begin{itemize}
		\item $E_1\leq \frac{\delta}{4}$ for large $n$ since $\rel(\r_n)\ll \M(\r_n)$;
		\item $E_2\leq \frac{\delta}{4}$ for small $\delta'$ and large $n$ by the definition of uniform transience;
		\item $E_3\leq \frac{\delta}{4}$ for large $n$  since we have $t_n\ll \M(\r_n)$ and $\rel(\r_n)\ll \M(\r_n)$ (by uniform transience);
		\item $E_4\leq \frac{\delta}{4}$ for large $n$ since $t_n$ goes to $\infty$. 
	\end{itemize}
	
	We also need to use the fact that $\M(\r_n)/n$ stays bounded to control $E_2, E_3$ and $E_4$. 
	Hence all  conditions of Proposition \ref{p:uniftrans} are satisfied for any $\ep>0$ by first picking a sufficiently small $\delta'$ and then making $n$ large enough.  
	Thus we have
	$$
	\lim_{n\to\infty}  \frac{nt_n}{2\M(\r_n)} \abs{  P_{t_n} -\frac{2\M(\r_n)}{nt_n}}=0,
	$$ as desired. To prove the second statement of Theorem \ref{uniformtran}, it suffices to show
	\[
	\lim_{n\to\infty}\abs{\alpha_{t_n}-\frac{n}{2\M(\r_n)}}=0.
	\]
	
	For a fixed $\Delta$, using equation \eqref{ctlf4} we have (recall $\alpha_{t_n}=q(t_n)$ by the definition of $q(t)$)
	\begin{equation}\label{q-n/2m}
		\begin{split}
			\limsup_{n\to\infty}    \abs{q(t_n)-\frac{n}{2\M(\r_n)}}\leq 
			C\left( \frac{1}{\Delta}+\limsup_{n\to\infty}\int_{2\Delta}^{\rel(\r_n)}p_s(x,x)\mathrm{d}s\right).
		\end{split}
	\end{equation}
	Since the l.h.s. of equation \eqref{q-n/2m} is independent of $\Delta$, we can send $\Delta$ to infinity to get
	\begin{equation*}
		\begin{split}
			\limsup_{n\to\infty}    \abs{q(t_n)-\frac{n}{2\M(\r_n)}}\leq \limsup_{\Delta\to\infty}
			C\left( \frac{1}{\Delta}+\limsup_{n\to\infty}\int_{2\Delta}^{\rel(\r_n)}p_s(x,x)\mathrm{d}s\right)=0,
		\end{split}
	\end{equation*}
	by the definition of uniform transience. This completes the proof of Theorem \ref{uniformtran}. 
\end{proof}

\subsection{Infinite transient transitive unimodular graphs}\label{s:infinitetrans}
The goal of this section is to prove Theorem \ref{infinitetrans}. 
We first recall \cite[Lemma 2.1]{foxall2018coalescing}, which shows that Lemma \ref{P_t=E(N_tinverse)} holds true for unimodular graph as well, provided that we define $N_t$ be the number of walkers that coalesced with the walker starting from $o$, the root of the unimdular graph, in the sense of the graphical representation of CRW. The strategy to control $\E(N_t^{-1})$ is similar to the framework for finite graphs. As in finite graphs, we 
write $$\E(N_t^{-1})=\E(N_t)^{-1}\E\left(
\frac{\E(N_t)}{N_t}\right).$$

We estimate $\E(N_t^{-1})$ in two steps. We first  show the convergence of $N_t/\E(N_t)$ to a Gamma distribution with density function $4xe^{-2x}\1[x\geq 0]$ using the method of moments. We will also determine the leading order of $\E(N_t)$ in this step.  We then show  the contribution to $\E(\E(N_t)/N_t)$ when $N_t/\E(N_t)$ is near $0$ is small. Once these two steps are completed Theorem \ref{infinitetrans} follows immediately. 

We let $\mathbf{X}(t):=(X_0(t),\ldots, X_k(t))$ where $X_0, X_1, \ldots, X_{k-1}, X_k$ be $k+1$ coalescing random walks.   Then we have 
\begin{equation}
	\begin{split}
		\E(N_t^{k})&=\sum_{o_1,\ldots, o_k\in V} \P(\C(X_0,X_1,\cdots, X_k)\leq t\mid 
		\mathbf{X}(0)=(o,o_1,\ldots, o_k)
		)\\
		&=\sum_{o',o_1,\ldots, o_k\in V}
		\P(\C(X_0,X_1,\cdots, X_k)\leq t, X_0(t)=o' \mid  \mathbf{X}(0)=(o,o_1,\ldots, o_k))\\
		&=\sum_{o',o_1,\ldots, o_k\in V}
		\P(\C(X_0,X_1,\cdots, X_k)\leq t, X_0(t)=o \mid  \mathbf{X}(0)=(o',o_1,\ldots, o_k)),
	\end{split}
\end{equation}
where in the last equality we have applied mass transport principle to the diagonally invariant function $F(x,y)$ defined by
$$
F(x,y):=\sum_{o_1,\ldots, o_k\in V}
\P(\C(X_0,X_1,\cdots, X_k)\leq t, X_0(t)=y
\mid  \mathbf{X}(0)=(x,o_1,\ldots, o_k)).
$$

Here $F(x,y)$ being `diagonally invariant' means $F(x,y)=F(\phi(x), \phi(y))$ for all $x,y\in \G$ and $\phi$ in the set of automorphism group of $\G$. See \cite[Chapter 8]{lyons2017probability}
for more details about mass transport principle.
By  the same argument as in the proof of Lemma \ref{l:reverdiffinit}  we get
\begin{equation}
	\begin{split}
		&\sum_{o',o_1,\ldots, o_k\in \G}
		\P(\C(X_0,X_1,\cdots, X_k)\leq t, X_0(t)=o \mid  \mathbf{X}(0)=(o',o_1,\ldots, o_k))\\
		&=(k+1)!\sum_{I_k \in \Phi_k} \int_{\R^{k}_{<, t}} h(\t, I_k)\mathrm{d}\t.
	\end{split}
\end{equation}

Here  $h(\t,I_k)$ is define using equation \eqref{cond_form} with $\gamma_0(0)=o$. Similarly we define $g(\t,I_k)$ and $q(\t,I_k)$ using \eqref{gt,phi_k} and \eqref{cond_form4}, respectively. 
As in the proof of Theorem \ref{uniformtran} we let the `good' set be $\{\t: \Delta_{\min}(\t)>\Delta\}$. On the good set, using the method of the proof of Lemma \ref{l:approxhbyg}, the difference $\abs{g(\t,I_k)-h(\t,I_k)}$ can be upper bounded by 
$C_k\int_{2\Delta}^{\infty}p_s(x,x)\mathrm{d} s$.
Let $X(s),s\geq 0$ and $Y(s),s\geq 0$ be two independent random walks.
We have the bound
$$
\abs{q(\t,I_k)-\P_{o,\nu_o}(X(s)\neq Y(s),\forall s\geq 0)^k}\leq k\P_{x,\nu_x}(\exists s\geq \Delta, X(s)=Y(s)) \leq Ck \int_{2\Delta}^{\infty}p_s(x,x)\mathrm{d} s
$$
on the good set. It follows that (note that $q(\t,I_k)=g(\t,I_k)$ for transitive chains)
\begin{equation*}
	\abs{ \int_{\R^k_{<.t}} h(\t,I_k)\mathrm{d}\t-\P_{o,\nu_o}(X(s)\neq Y(s),\forall s\geq 0)^k t^k/k!}\leq
	C_k\left(\left(\int_{2\Delta}^{\infty}p_s(o,o)\mathrm{d}s\right) t^k+ \Delta t^{k-1}\right).
\end{equation*}

Summing over all possible $I_k\in \Phi_k$ and using
Lemma \ref{densityexpress}
we get, for $\Delta, t\geq 1$ and $n$ large enough,
\begin{equation}\label{folner}
	\abs{\E(N_t^k)-\P_{o,\nu_o}(X(s)\neq Y(s),\forall s\geq 0)^k t^k (k+1)!} \leq  C_k\left(\left(\int_{2\Delta}^{\infty}p_s(o,o)\mathrm{d} s\right) t^k+ \Delta t^{k-1}\right).
\end{equation}

The transience of $\G$ implies 
$\int_{2\Delta}^{\infty}p_s(o,o)\mathrm{d} s$ converges to 0 as $\Delta$ goes to infinity. 
For any fixed $K$ and $\ep$, when $t$ is large enough, using \eqref{folner} we get
\begin{equation}\label{unimoment}
	\abs{\frac{\E(N_t^k)}{(\E(N_t))^k}-(k+1)!/2^k}\leq \ep,
\end{equation}
for $1\leq k\leq K$. This implies the convergence of $N_t/\E(N_t)$ to the Gamma  distribution with density $4x e^{-2x}\1[x\geq 0]$. It remains to show
$$\lim_{\delta\to 0}\limsup_{t\to\infty}
\E\left(
\frac{\E(N_t)}{N_t};\frac{N_t}{\E(N_t)} \leq \delta\right)=0.
$$
This can be proven directly by  using \cite[Lemma 2]{bramson1980asymptotics}, where we set the $f_t$ there to be $t$. Then $f_tP_t$ is uniformly bounded over $t$ by Theorem \ref{t:upperbound} and thus the condition of that lemma is satisfied. 
Combining the above two steps we have $$
\lim_{t\to\infty} tP_t(\G)=\lim_{t\to\infty} t\E(N_t^{-1})
=\lim_{t\to\infty}\frac{2t}{\E(N_t)}
=\frac{1}{\P_{o,\nu_o}
	(X(s)\neq Y(s),\forall s\geq 0)}.
$$
This completes  the proof of Theorem 
\ref{infinitetrans}.

\begin{appendix}

\section{Classical results on Markov chains}\label{as:markov}
We provide proofs to some classical results on Markov chains from Section \ref{ss:relaxmixtime}.

\subsection{The Poincar\'{e} inequality}
\begin{proof}[Proof of Lemma \ref{l:poincareineq}]
	Note that now the stationary distribution $\pi$ is uniform.
	Equation \eqref{tvd} is a direct consequence of equation \eqref{prepoincare} by noting that  $\sum_x \left(\mu_t(x)-\frac{1}{|V|} \right)^2=|V|\Var(p_{t} \mu)$. To prove equation \eqref{poincare},
	we first note that by Cauchy-Schwartz inequality and symmetry of $p_{t/2}$, we have
	\[p_{t}(x,y)^2=\left(\sum_z p_{t/2}(x,z)p_{t/2}(z,y)\right)^2 \le \sum_z p_{t/2}(x,z)^2\sum_z p_{t/2}(y,z)^2=p_{t}(x,x)p_{t}(y,y),\]
	and hence $\max_{x,y}p_{t}(x,y)=\max_x p_t(x,x)$ for all $t \ge 0$.
	Also, if $\mu$ is the Dirac measure at state $x$, then by the symmetry of $p_{t/2}$,
	\[
	p_{t}(x,x)=\sum_{y}p_{t/2}(x,y)p_{t/2}(y,x)=\sum_y \mu_{t/2}(y)^2=
	\left(\sum_y \left(\mu_{t/2}(y)-\frac{1}{|V|}\right) \right)^2 + \frac{1}{|V|}.\] Thus \eqref{poincare} follows from \eqref{tvd}. \end{proof}

\subsection{Mixing time and relaxation time}
\begin{proof}[Proof of Lemma \ref{l:boundH}]
	Assume $n=\abs{V}$.
	Since $p_s(x,x)\geq \exp(-\bar d s)$, we see for $t\leq 1/\bar d$
	$$
	\min_x \int_0^t p_{s}(x,x)\mathrm{d}s\geq t\exp(-1),
	$$
	while $$
	\max_x \int_0^t p_{s}(x,x)\mathrm{d}s\leq t.
	$$
	Hence it suffices to  prove that  for some constant $C$ depending on $\kappa_0$ and $\bar d$, we have 
	$$
	\max_x \int_0^n p_{s}(x,x)\mathrm{d}s\leq C.
	$$
	for all $x\in V$.
	Using equation \eqref{poincare}, we see that
	$$
	p_s(x,x)-\frac{1}{n}\leq \exp\left(-\frac{s}{\rel}\right).
	$$
	It follows that
	$$
	\int_0^n \left(p_s(x,x)-\frac{1}{n}\right)\mathrm{d}s\leq 
	\int_0^{\infty} \exp\left(-\frac{s}{\rel}\right)\mathrm{d}s\leq \rel.
	$$
	Hence we have $$
	\int_0^n p_s(x,x)\mathrm{d}s\leq 1+\rel\leq 1+\frac{2\bar d}{\kappa_0^2}, 
	$$
	as desired. 
\end{proof}

\subsection{Relations between Markov chain parameters}
\begin{proof}[Proof of Lemma \ref{walkparamaters}]
	To prove equation \eqref{relandrmax}, we see by definition $\rel=1/\lambda$  where $-\lambda$ is the largest non-zero eigenvalue of the $Q$-matrix associated with the transition rate $\r$ (i.e., $Q_{ij}=r_{i,j} $ for all $i\neq j$ and $Q_{i,i}=-r(i)$ for all $i$). 
	One can show the matrix $Q+2r_{\max}I$ is positive definite ($I$ is the identity matrix), implying $-\lambda\geq -2r_{\max}$. Hence $\rel\geq 1/(2r_{\max})$. 
	On the other hand, equation \eqref{mrmax/n} follows directly from
	Lemma 2 in \cite{aldous} applied to the product chain of two independent Markov chains by considering the diagonal set. 
	Finally equation \eqref{maxxyp} follows from the Cauchy-Schwartz inequality
	\begin{equation*}
		\begin{split}
			p_s(x,y)&=\sum_{z}p_{s/2}(x,z)p_{s/2}(z,y)\leq
			\sqrt{\left(\sum_z p_{s/2}(x,z)^2\right)\left(\sum_z p_{s/2}(z,y)^2\right)}\\
			&\leq \frac{1}{2}\left(\left(\sum_z p_{s/2}(x,z)^2\right)+
			\left(\sum_z p_{s/2}(z,y)^2\right)
			\right)
		\end{split}
	\end{equation*} and 
	\begin{align*}
		&\sum_z p_{s/2}(x,z)^2=\sum_{z}p_{s/2}(x,z)p_{s/2}(z,x)=p_s(x,x),\\
		&\sum_z p_{s/2}(z,y)^2=\sum_{z}p_{s/2}(y,z)p_{s/2}(z,y)=p_s(y,y),
	\end{align*}
	owing to the fact that $\r$ is symmetric (so $p$ is also symmetric).  
\end{proof}

\subsection{Bounds for meeting probability}
\begin{proof}[Proof of Lemma \ref{meetprob}]
	Denote by $q(x,y)$ the probability in question.
	Denote by $N_1(t,t+T,x,y)$ the time two independent random walks that starts from     $x$ and $y$ spend together between time $t$ and $t+T$ and by $N_2(t,z)$ the time two random walks both starting from $z$ spend together between time 0 and time $t$.
	It follows that
	\begin{equation}
		q(x,y)\leq \frac{\E(N_1(t,t+T,x,y))}{\E(N_1(T,t+T,x,y)  \mid  
			\mbox{a collision happens between } T \mbox{ and }t)  }. 
	\end{equation}
	Note that the denominator is bounded below by
	$$
	\min_{z}\E(N_2(z,t))\geq \min_z \int_0^T \sum_w p_s(z,w)^2\mathrm{d} s=\min_z \int_0^T
	p_{2s}(z,z)\mathrm{d} s.
	$$
	The numerator is bounded above by 
	$$
	\int_t^{t+T}\sum_w p_s(x,w)p_s(y,w)\mathrm{d} s\leq \max_{z,w} \int_t^{t+T} p_{2s}(z,w) \mathrm{d} s.
	$$
	The inequality \eqref{poincare} implies
	$$
	\max_{z,w}\int_0^{t+T} p_{2s}(z,w) \mathrm{d} s\leq \left(\sum_{i=1}^{\lceil  t/T\rceil}\exp\left(-\frac{iT}{\rel}\right) \right)
	\max_x \int_0^T p_{2s}(x,x)\mathrm{d} s+\frac{t}{n}.
	$$
	We claim that
	\begin{equation}\label{intmin}
		\min_z\int_0^T p_{2s}(z,z)\mathrm{d} s
		\geq \frac{1}{8}
		\left(T\wedge \frac{1}{r_{\max}}\right).
	\end{equation}
	To see that this is true, note $\min_z p_{2s}(z,z)\geq \exp(-2sr_{\max})$. If $T<1/r_{\max}$ then $\min_z p_{2s}(z,z)$  will be be at least  $e^{-2}$ for all $0\leq s\leq T$ and hence  
	the integral will be at least $T/8$.
	
	It follows that $q(x,y)$ is bounded above by
	$$
	\frac{\exp(-T/\rel)}{1-\exp(-T/\rel)}
	\frac{\max_{z} \int_0^{2T} p_{s}(z,z)\mathrm{d} s }{\min_{z} \int_0^{2T} p_s(z,z)\mathrm{d} s  }+\frac{8t}{n} \left(\frac{1}{T} \vee r_{\max}\right).
	$$
	Since $q(x,y)$ is trivially bounded above by $1$ we arrive at the form presented in the statement of Lemma \ref{meetprob}. 
\end{proof}

\subsection{Uniformly transient Markov chains}
\begin{proof}[Proof of Proposition \ref{p:treluniform}]
	Without loss of generality we may assume $\abs{V_n}=n$. 
	We first prove $\rel(\r_n)=o(\M(\r_n))$ by contradiction. Suppose $\rel(\r_n)\geq cn$ for some constant $c>0$ and all $n$.
	Using the spectral decomposition (see, e.g., \cite[Section 12.1]{levin2017markov}) and the fact that the stationary measure is the uniform measure in our case, we see that
	$$
	p_t^{(n)}(x,x)\geq 1/n
	$$
	for any $t\geq 0$. 
	Using this 
	and the fact $\M(\r_n)\geq n/4$ (by \eqref{mrmax/n}), we see that for any fixed $s$ and $n>s/c$, 
	\begin{equation}
		\max_{x,y\in V_n}\int_{s\wedge \rel(\r_n)}^{\rel(\r_n)}
		p_t^{(n)}(x,y)\mathrm{d}t \geq \int_{s}^{cn} \frac{1}{n}\mathrm{d}t\geq \frac{cn-s}{n}.
	\end{equation}
	Sending $n\to\infty$ and then $s\to\infty$, we get
	$$
	\lim_{s\to\infty}\limsup_{n\to\infty} \max_{x,y\in V_n}\int_{s\wedge \rel(\r_n)}^{\rel(\r_n)}
	p_t^{(n)}(x,y)\mathrm{d}t \geq \lim_{s\to\infty}c>0,
	$$
	which contradicts the definition of uniform transience.
	To prove $\M(\r_n)=\Theta(n)$ it suffices to show $\M(\r_n)=O(n)$ since the other direction is true by \eqref{mrmax/n}. 
	Again by using the spectral decomposition  we see for all $x$ and all $t,s\geq 0$
	$$
	0\leq p^{(n)}_{t+s}(x,x)-\pi(x)\leq  \exp(-s/\rel(\r_n))(p_{t}^{(n)}(x,x)-\pi(x)).
	$$
	This implies for any $j\geq 0$,
	\begin{equation*}
		0\leq \int_{(j+1)\rel}^{(j+2)\rel} ( p^{(n)}_{t}(x,x)-\pi(x))\mathrm{d}t \leq e^{-1}
		\int_{j\rel}^{(j+1)\rel} ( p^{(n)}_{t}(x,x)-\pi(x))\mathrm{d}t. 
	\end{equation*}
	It follows that
	\begin{equation}\label{utrans}
		\begin{split}
			&\int_0^{\infty} ( p^{(n)}_{t}(x,x)-\pi(x))\mathrm{d}t =
			\sum_{j=0}^{\infty}  \int_{j\rel}^{(j+1)\rel}(p_t(x,x)-\pi(x))\mathrm{d}t\\
			\leq &\left(\sum_{j=0}^{\infty}e^{-j}\right) \int_0^{\rel}
			( p^{(n)}_{t}(x,x)-\pi(x))\mathrm{d}t\leq 
			\frac{e}{e-1}  \int_0^{\rel}
			( p^{(n)}_{t}(x,x)-\pi(x))\mathrm{d}t.
		\end{split}
	\end{equation}
	The first line of \eqref{utrans} is equal to 
	$
	\pi(x)\E_{\pi}(T_x)
	$ by \cite[Proposition 10.26]{levin2017markov}. 
	On the other hand the
	condition of uniform transience implies the existence of a constant $C>0$ such that
	\begin{equation}\label{utrans2}
		\sup_{x\in V_n} \int_0^{\rel}
		( p^{(n)}_{t}(x,x)-\pi(x))\mathrm{d}t\leq C. 
	\end{equation}
	It follows from \eqref{utrans}, \eqref{utrans2} and the inequality $t_{\mathrm{hit}}\leq 2\max_{x\in V_n}\E_{\pi}(T_x)$ that
	$$
	t_{\mathrm{hit}}(\r_n)  \leq \frac{2Ce}{e-1}n
	$$
	since $\pi(x)=1/n$ in our case. Since $\M=t_{\mathrm{hit}}/2$ for transitive chains, we conclude that $\M(\r_n)=O(n)$, as desired. 
\end{proof}

	
	\section{Local weak convergence of random graphs}\label{finitetoinfinite}
	To explain the notion of local weak convergence we need a few definitions. 
	\begin{definition}
		For any graph $\G$, $v \in V$, and $\rho\in \ZZ_+$, let $N_{\rho}(v,\G)$ be the subgraph induced by $\{u:d(u,v)\leq \rho\}$, where $d$ is the graph distance.
	\end{definition}
	
	\begin{definition}
		We consider the space $\scG_*$ of all rooted graphs.
		We make it a metric space: for any $(\G_1, x_1), (\G_2, x_2) \in \scG_*$,
		their distance is defined to be
		$$
		d_*((\G_1, x_1), (\G_2, x_2)) := (1 + \rho_*)^{-1}
		$$
		where $\rho^*$ is the maximum integer such that 
		$ N_{\rho^*}(x_1,\G_1)$ is isomorphic to $ N_{\rho^*}(x_2,\G_2)$.
	\end{definition}
	
	\begin{definition}
		For a sequence of  finite (possibly random) graphs $\G_n$, make each of them a random element in the space $\scG_*$, by choosing a root $\U_n$ uniformly among all vertices of $\G_n$.
		The \emph{local weak limit}  of $(\G_n,\U_n)$ is defined as the weak limit of these random elements $(\G_n,\U_n)$ in $\scG_*$, under the topology given by the metric $d_*$. Equivalently, we say $(\G_n,\U_n)$ converges to $(\G,o)$ in the local weak sense if 
		for any function $\psi:\scG_* \to \R$ that is continuous under the metric $d_*$, we have
		$$
		\E_{\U_n}(\psi(\G_n,\U_n)) \mbox{ converges to } \E(\psi(\G,o)) \mbox{ in probability},
		$$
		where $\E_{\U_n}(\psi(\G_n,\U_n))$ indicates the expectation is only taken w.r.t. the randomness of $\U_n$ and  $\E(\psi(\G,o))$  is the expectation of $\psi(\G,o)$ taken w.r.t. the law of $(\G,o)$ in $\scG_*$.
	\end{definition}

	Let $\UGT(D)$ be the unidomular Galton-Watson tree as introduced before Theorem \ref{t:config}.

	\begin{lemma}\label{CMLWC}
		Let $\G_n\sim \CM_n(D)$. Let $\U_n$ be a uniformly chosen vertex of $\G_n$. Let $\G\sim \UGT(D)$ and $o$ be the root of $\G$. Then $(\G,o)$ is the local weak limit of 
		$(\G_n,\U_n)$ 
	\end{lemma}
	A proof of this fact can be found in, e.g., \cite[Theorem 2.11]{van2017stochastic}.
	The following proposition gives a connection between CRW on finite and infinite graphs. 
	\begin{prop}\label{local1}
		Suppose $(\G_n,\U_n)$ is a sequence of random rooted graphs with uniformly bounded degree 
		that converges in the local weak sense to $(\G,o)$. Consider CRW on $\G_n$ and $\G$ with unit edge rate. 
		Then for any fixed finite time $t$, 
		$P_{t}(\G_n)$ converges in distribution to $\E_{\G}(P_t(o))$,
		where $P_{t}(\G_n)$ is the density of CRW in $\G_n$  and $\E_{\G}(P_t(o))$ is the expectation of $P_t(o)$ under the law of $(\G,o)$.
	\end{prop}
	\begin{proof}
		Using the duality between the voter model and CRW, it is equivalent to proving the convergence of the survival probability of the voter model.
		For $o\in \G$, we let $\psi(\G,o)$ be the probability that the initial opinion of $o$ survives to time $t$, i.e., $\P(\zeta^o_t\neq \eset)$   and similarly define $\psi(\G_n,\U_n)$. Then $P_t(\G_n)=\E_{\U_n}(\psi(\G_n,\U_n))$ and $P_t(o)=\psi(\G,o)$. We also define $\psi_{\rho}(\G,o)$ to be 
		$\P(\zeta^o_t \neq \eset, \zeta^o_s \cap N_{\rho}(o)^c =\eset, \forall 0\leq s\leq t)$ and similarly define 
		$\psi_{\rho}(\G_n,\U_n)$. Then $\psi_{\rho}$ is a continuous function in the space $\scG_*$ under the metric $d_*$. 
		Due to that the degree of $\G_n$ is uniformly bounded, the quantity $\abs{\psi_{\rho}(\G_n,\U_n)-\psi(\G_n,\U_n)}$ converges to $0$ as $\rho\to \infty$, uniformly in $n$. Similarly $\abs{\psi_{\rho}(\G,o)-\psi(\G,o)}$ also converges to $0$. Proposition \ref{local1} now follows from applying the definition of local weak convergence to the function $\psi_{\rho}$ and then  sending $\rho\to\infty$. 
	\end{proof}

\end{appendix}

\begin{acks}[Acknowledgments]
The authors are grateful to Rick Durrett for introducing the problem and sharing his perspectives.
We also thank anonymous referees for carefully reading this manuscript and providing many useful comments and suggestions.
The project was initiated in the workshop 2019
AMS-MRC: Stochastic Spatial Models. 
\end{acks}

\begin{funding}
	SL and LZ were provided 2019 AMS-MRC Collaboration Travel Funds, which enabled them a visit to Duke University, during which part of this work was completed. Both the workshop and the travel funds are supported by the National Science Foundation under Grant Number NSF DMS-1641020. JH is supported by NSERC grants.
	\end{funding}




\bibliographystyle{imsart-number}
\bibliography{bibliography}

\end{document}